\documentclass[english]{article} 
\usepackage[absolute,overlay]{textpos} 
\usepackage{amsthm}
\usepackage{amsmath}
\usepackage{amssymb}
\usepackage[utf8]{inputenc}
\usepackage[english]{babel}
\usepackage{graphicx}
\usepackage{geometry}
\usepackage{caption}
\usepackage{subcaption}
\usepackage{pdfpages}
\usepackage{comment}
\usepackage{bm}
\usepackage{etoolbox}
\usepackage{mathtools}
\usepackage{tikz}
\usetikzlibrary{shapes,snakes}
\usetikzlibrary{shapes.geometric}
\usepackage{hyperref}
\usepackage{cleveref}
\usepackage{pgfplots}
\usepackage{pgfplotstable}
\usepackage{booktabs}
\usepackage{stmaryrd} 
\usepackage{multirow}

\usetikzlibrary{fillbetween}
\usepgfplotslibrary{fillbetween}

\usetikzlibrary{shapes.misc}
\usetikzlibrary{math} 

\usetikzlibrary{arrows}

\definecolor{darkspringgreen}{rgb}{0., 0.55, 0.3}
\definecolor{dartmouthgreen}{rgb}{0.05, 0.5, 0.06}
\definecolor{etonblue}{rgb}{0.59, 0.78, 0.64}
\definecolor{airforceblue}{rgb}{0., 0.4, 0.66}
\definecolor{arylideyellow}{rgb}{0.91, 0.84, 0.42}
\definecolor{emerald}{rgb}{0.31, 0.78, 0.47}
\definecolor{uclagold}{rgb}{1.0, 0.7, 0.0}
\definecolor{cadmiumorange}{rgb}{0.93, 0.53, 0.18}

\newtheorem{theorem}{Theorem}
\numberwithin{theorem}{section}

\numberwithin{definition}{section}
\newtheorem{remark}{Remark}
\numberwithin{remark}{section}
\newtheorem{proposition}[theorem]{Proposition}

\newcommand{\lopd}[0]{\mathcal{L}_\Delta}
\newcommand{\lopdt}[0]{\mathcal{L}_{\Delta}}
\newcommand{\usol}[0]{\underline{\uvec{u}}_\Delta}

\newcommand{\uapp}[0]{\uvec{u}_h}

\newcommand{\tess}[0]{\mathcal{T}_h}

\newcommand{\uvec}[2][3]{\boldsymbol{#2\mkern-#1mu}\mkern#1mu}

\newcommand\norm[1]{\left\lVert#1\right\rVert}

\newcommand{\ST}[0]{\boldsymbol{ST}_i^K}

\newcommand{\elres}[0]{\uvec{\Phi}^K(\uapp)}
\newcommand{\noderes}[0]{\uvec{\Phi}^K_i(\uapp)}

\newcommand{\spacestuff}[0]{\bm{\phi}_i}

\newcommand{\cund}[0]{\underline{\uvec{c}}}

\newcommand{\lopdi}[0]{\mathcal{L}_{\Delta,i}}

\newcommand{\GF}[0]{\text{GF}}
\newcommand{\HS}[0]{\text{HS}}

\newcommand{\faces}[0]{\mathcal{F}_h}
\newcommand{\undu}[0]{\underline{\uvec{u}}}
\newcommand{\undz}[0]{\underline{\uvec{z}}}

\newcommand{\dt}{\Delta t}

\newcommand\swapifbranches[3]{#1{#3}{#2}}
\makeatletter{}
\MHInternalSyntaxOn{}
\patchcmd{\DeclarePairedDelimiter}{\@ifstar}{\swapifbranches\@ifstar}{}{}
\MHInternalSyntaxOff{}
\makeatother{}
\DeclarePairedDelimiter\abs{\lvert}{\rvert}
\DeclarePairedDelimiter\restrictDelimiters{.}{\vert}
\newcommand\restrict[2]{\ensuremath{{\restrictDelimiters{#1}}_{#2}}}

\newcommand{\RIcolor}[1]{{\leavevmode\color{black} #1}}

\begin{document}
\author{L. Micalizzi\footnote{Affiliation: Department of Mathematics, North Carolina State University, SAS Hall, 2108, 2311 Stinson Dr, Raleigh, 27607, United States. Email: lmicali@ncsu.edu}, M. Ricchiuto\footnote{Affiliation: Team CARDAMOM, INRIA, University of Bordeaux, CNRS, Bordeaux INP, IMB, UMR 5251, France.
Email: mario.ricchiuto@inria.fr} and R. Abgrall\footnote{Affiliation: Institut für Mathematik, Universität Zürich, Winterthurerstrasse 190, Zurich, 8057, Switzerland. Email: remi.abgrall@math.uzh.ch}}
\title{Novel well-balanced continuous interior penalty stabilizations}

\maketitle

\begin{abstract}
In this work, the high order accuracy and the well-balanced (WB) properties of some novel continuous interior penalty (CIP) stabilizations for the Shallow Water (SW) equations are investigated. 
The underlying arbitrary high order numerical framework is given by a Residual Distribution (RD)/continuous Galerkin (CG) finite element method (FEM) setting for the space discretization coupled with a Deferred Correction (DeC) time integration, to have a fully-explicit scheme. 
If, on the one hand, the introduced CIP stabilizations are all specifically designed to guarantee the exact preservation of the lake at rest steady state, on the other hand, some of them make use of general structures to tackle the preservation of general steady states, whose explicit analytical expression is not known.
Several basis functions have been considered in the numerical experiments and, in all cases, the numerical results confirm the high order accuracy and the ability of the novel stabilizations to exactly preserve the lake at rest steady state and to capture small perturbations of such equilibrium. 
Moreover, some of them, based on the notions of space residual and global flux, have shown very good performances and superconvergences in the context of general steady solutions not known in closed-form.
Many elements introduced here can be extended to other hyperbolic systems, e.g., to the Euler equations with gravity.
\end{abstract}

\newpage
\tableofcontents
\section{Introduction}
\label{intro}
In the context of the numerical resolution of hyperbolic partial differential equations (PDEs), one has to deal with several challenges, among which: the presence of instabilities and the exact preservation of some analytical solutions at the discrete level, namely well-balancing. 

The instability issues are usually solved through an upwinding in the Finite Volume/discontinuous Galerkin FEM setting, through stabilization techniques in the RD/CG FEM setting. In particular, in this last context, the existing literature offers many possible options, for example: Streamline-Upwind Petrov-Galerkin, orthogonal subscale stabilization and CIP, respectively introduced in \cite{brooks1982streamline}, \cite{codina1997finite} and \cite{CIP}. For more details, the reader is referred to \cite{michel2021spectral} and \cite{michel2022spectral} in which a complete Fourier analysis and numerical investigation of the mentioned stabilizations, with different basis functions and time discretizations, has been performed.

We refer to well-balancing or C-property as the ability of a numerical scheme to exactly preserve a particular analytical solution or to be superconvergent, toward such solution, with respect to the general accuracy of the underlying discretization. %
In many applications, one is interested in embedding such feature in the adopted numerical method. This happens, for example, in the context of the study of physical systems admitting nontrivial stationary equilibria.
In fact, such systems can stay for very long time in a neighborhood of a steady state. 
For this reason, researchers are interested in studying the evolution of small perturbations of steady solutions and, in this regard, it is desirable not to confuse the evolution of the perturbations with the natural noise arising from the numerical discretization. 
In such context, there are essentially two possibilities: using very refined meshes, with consequent increase in the computational cost, or modifying the numerical scheme in such a way that it preserves exactly the analytical solution of interest, without wasting the accuracy toward any other general solution.
The latter option looks indeed very appealing, however, it is also very challenging since the steady states are usually not available in closed-form and, more in general, we rely on numerics because we do not have the analytical solutions to the systems of PDEs that we are trying to numerically solve.

Several strategies have been introduced to achieve well-balancing. The interested reader is referred to \cite{CaPa,Klin,Var,Klin2,Dumb,Maria,ricchiuto2011c,ciallella2023arbitrary,mantri2024fully} and references therein. In particular, a successful approach is the one introduced in \cite{gascon2001construction} and is based on the definition of a global flux, i.e., a new flux which keeps into account the source term, allowing to recast the initial problem into an equivalent one which is homogeneous.

In this work, we introduce some novel arbitrary high order WB CIP stabilizations for the SW equations in an RD/CG setting. 
They all are designed in such a way to exactly preserve the lake at rest steady state, however, some of them address the challenge of the preservation of general steady equilibria not known in closed-form.
The time discretization is achieved via the bDeCu method, introduced in \cite{micalizzi2023new} as an efficient modification of the DeC for hyperbolic problems designed in \cite{Decremi} to get arbitrary high order fully explicit schemes avoiding the issues associated with the mass matrix.
Even though the results obtained in some multidimensional tests on unstructured meshes will be reported, the numerical validation is mostly performed in a \RIcolor{one-dimensional} setting, which is therefore the reference setting for the theoretical description of the presented notions.

The structure of this work is the following.
We will start by introducing the SW equations and their steady solutions in Section \ref{chapSW}.
Then, we will introduce, in Section \ref{chapRD}, the CG FEM and explain how this can be put into an RD formulation.
The issue of achieving well-balancing in the presented formulation is addressed in Section \ref{chapWB}, which is the main section of this work. There, we will present two WB space discretizations and the novel CIP stabilizations.
In Section \ref{chapDECRD}, we will describe the time-stepping strategy, the bDeCu.
In Section \ref{chapNum}, we will present the numerical results.
Finally, Section \ref{chapConc} is dedicated to conclusions and to future perspectives.


\section{Shallow water equations}\label{chapSW}
The SW equations are a system of hyperbolic PDEs used to model water flows, e.g., flows in sees, rivers, lakes or channels. Their \RIcolor{one-dimensional} formulation, without rain and assuming a bottom topography fixed in time, reads
\begin{equation}
	\label{eq:sys}
	\frac{\partial}{\partial t}\uvec{u}+\frac{\partial}{\partial x} \boldsymbol{F}(\uvec{u})=\boldsymbol{S}(x,\uvec{u}), \quad (x,t) \in \Omega\times \mathbb{R}^+_0,
\end{equation}
where $\Omega\subseteq \mathbb{R}$ is the space domain and the vector of the conserved variables, the flux and the source term are respectively defined as
\begin{equation}
	\uvec{u}:=\begin{pmatrix}
		H\\
		q
	\end{pmatrix},\quad\uvec{F}(\uvec{u}):=\begin{pmatrix}
		q\\
		\frac{q^2}{H}+g\frac{H^2}{2}
	\end{pmatrix},\quad
	\boldsymbol{S}(x,\uvec{u}):=-\begin{pmatrix}
		0\\
		gH\frac{\partial}{\partial x} B(x)+g\frac{n_M^2\vert q \vert}{H^{\frac{7}{3}}}q
	\end{pmatrix},
	\label{eq:sw}
\end{equation}
where $H$ is water height, $q:=Hv$ is the momentum of the flow, with $v$ being the water speed averaged in the vertical direction, $g$ is the gravitational constant, $B$ is the bathymetry (or bottom topography) and $n_M$ the Manning friction coefficient. Further, we introduce the total height $\eta:=H+B$ and the sound speed $c:=\sqrt{gH}$.

The Jacobian of the flux with respect to the conserved variables is given by
\begin{equation}
	\uvec{J}(\uvec{u}):=\frac{\partial \uvec{F}}{\partial \uvec{u}}=\begin{pmatrix}
		0 & 1\\
		-\frac{q^2}{H^2}+gH & 2\frac{q}{H}
	\end{pmatrix},
	\label{eq:jacobian}
\end{equation}
with the two real eigenvalues given by $\lambda_{1,2}=v\pm c$.

When no friction is present, the SW system is also endowed with an entropy pair $(s,F_s)$, with entropy $s$ and entropy flux $F_s$ respectively given by \cite{ranocha2017shallow,gassner2016well}
\begin{equation}\label{eq:ntrop}
	s:=\frac{1}{2}\frac{q^2}{H}+\frac{1}{2}gH^2+gHB\,,\;\;
	F_s := q\left( \frac{1}{2}\frac{q^2}{H^2}+g\eta\right),
\end{equation} 
with associated entropy variables
\begin{equation}
	\uvec{w}:=\left(g\eta-\frac{1}{2}\frac{q^2}{H^2}, \frac{q}{H} \right)^T\,.
	\label{eq:entropy_vars}
\end{equation}

Due to their relevance in many applications, the numerical resolution of the SW equations is a very active area of research, see \cite{xingsurvey,galland1991telemac,busto2022staggered,ciallella2022arbitrary,ciallella2023arbitrary,chertock2022well,cao2023flux,cao2022flux,kurganov2023well,gassner2016well,ranocha2017shallow,mPRKBL2,mantri2024fully,ricchiuto2015explicit,ricchiuto2009stabilized,abgrall2023new,ricchiuto2011c,ciallella2024high} and references therein for a non-exhaustive literature.

\subsection{Steady states}
The SW equations are well known to be characterized by nontrivial stationary solutions satisfying, in the weak sense, the ODE 
\begin{equation}
	\label{eq:steady_general}
	\frac{\partial}{\partial x} \boldsymbol{F}(\uvec{u})=\boldsymbol{S}(x,\uvec{u}).
\end{equation}
The simplest stationary solution is the so-called ``lake at rest'' steady state given by constant total height and zero velocity 
\begin{equation}
	\label{eq:lakeatrest}
	\eta=H+B\equiv \overline{\eta} \in \mathbb{R}^+_0, \quad v\equiv 0.
\end{equation}
When no friction is present, through basic analysis, from \eqref{eq:steady_general} one can show that smooth steady states are characterized by constant momentum and energy
\begin{equation}
	\label{eq:steady_no_friction}
	q(x,t)\equiv \overline{q} \in \mathbb{R}, \quad E=\frac{1}{2}\frac{\overline{q}^2}{H^2}+g(H+B)\equiv \overline{E}\in \mathbb{R}^+_0.
\end{equation}
When also the friction is present, one can easily prove \cite{macdonald1996analysis,macdonald1997analytic} that smooth steady solutions satisfy
\begin{equation}
	\label{eq:steady_with_friction}
	q(x,t)\equiv \overline{q} \in \mathbb{R}, \quad \left(-\frac{\overline{q}^2}{H^3}+g\right)\frac{\partial}{\partial x}H=-g\frac{\partial}{\partial x}B-g\frac{n_M^2\vert \overline{q} \vert}{H^{\frac{10}{3}}}\overline{q}.
\end{equation}
In general, steady states are not available in closed-form and are obtained by solving \eqref{eq:steady_general}. The interested reader is referred to \cite{delestre2013swashes}, in which a wide collection of analytical solutions (not only steady) is provided.

\section{Continuous Galerkin FEM and Residual Distribution}\label{chapRD}
We will introduce in this section the CG FEM for hyperbolic problems and show how such method can be put in an RD formalism. For more information, the interested reader is referred to \cite{RemiMarioRD}.

\subsection{CG}
We would like to numerically solve a hyperbolic system of balance laws in the form \eqref{eq:sys} over the bounded domain $\Omega:=(x_L,x_R)$ in the time interval $[0,T_f]$, with some initial and boundary conditions.
The main ingredients of the CG method are
\begin{itemize}
	\item[•] a tessellation $\tess$ of the space domain made by non-overlapping closed elements $K$, segments in this case as we consider a \RIcolor{one-dimensional} setting, covering its closure exactly;
	\item[•] the finite dimensional space $W_M:=\lbrace \varphi \in C^0(\overline{\Omega}) ~ : ~ \varphi \vert_K \in \mathbb{P}_M(K), ~ \forall K \in \tess\rbrace$ of the continuous functions $\varphi$ which are such that their restriction to each element $K$ of the tessellation is a polynomial of degree $M$;
	\item[•] a basis $\left\lbrace \varphi_i \right\rbrace_{i=1,\dots,I}$ of $W_M$ normalized in such a way that $\sum_{i=1}^I \varphi_i \equiv 1$ and which is such that each basis function $\varphi_i$ can be associated to a spatial node $x_i \in \overline{\Omega}$, usually referred to as ``degree of freedom'' (DoF). 
\end{itemize}
Further, the adopted bases are such that each basis function $\varphi_i$ has support in the union of the elements $K\in K_i$, with $K_i:=\left\lbrace K\in \tess~:~ x_i \in K \right\rbrace$ being the set of the elements containing the DoF $x_i$ to which the function is associated. Let us notice that the previous assumptions imply
\begin{equation}
	\sum_{x_i\in K} \varphi_i(x) \equiv 1, \quad \forall x\in K.
	\label{eq:local_normalization_bases}
\end{equation}
Examples of such bases, considered in the numerical tests, are the Bernstein polynomials and the Lagrange polynomials associated to equispaced nodes or to Gauss--Lobatto (GL) ones.

Finally, the CG method is given by a projection of the weak formulation in space of \eqref{eq:sys} over $W_M$, i.e., we look for an approximated solution $\uapp(x,t):=\sum_{i=1}^I \uvec{c}_i(t) \varphi_i(x)$, linear combination of the basis functions through unknown coefficients which depend on time, such that it satisfies the following system of equations
\begin{align}
	\int_\Omega \biggl( \frac{\partial}{\partial t}\uvec{u}_h + \textcolor{black}{\biggl[\frac{\partial}{\partial x}   \boldsymbol{F}(\uvec{u}_h) -\boldsymbol{S}(x,\uvec{u}_h) \biggr]_h} \biggr) \varphi_i(x)dx +\textcolor{black}{\uvec{ST}_i(\uvec{u}_h)} =\uvec{0},  \quad \forall i=1,\dots,I,
	\label{eq:CG}
\end{align}
where the term in square brackets is a consistent discretization of the spatial part of our initial PDE and $\uvec{ST}_i$ is a stabilization term introduced at the discrete level to prevent the instabilities of central schemes.

Equation \eqref{eq:CG} is the semidiscretization of the CG method and consists of a nonlinear system of ODEs in all the coefficients $\uvec{c}_i(t)$, collected in a single vector $\uvec{c}(t)$, characterized by a mass matrix which is big and sparse.
By numerically solving such system, one gets the evolution in time of the approximated solution $\uapp$.

Let us leave aside for one moment the problem of the time-stepping, which is not central in the context of this work, and let us observe that, if the discretization of the spatial part of the equation and the stabilization term in \eqref{eq:CG} are defined in such a way to be exactly zero for a particular steady state, then the resulting numerical scheme will be WB with respect to such steady state.
This will be the main topic of Section \ref{chapWB} but, before addressing the problem of well-balancing, we introduce here a short subsection, in which we show how the CG method can be easily embedded in an RD framework.

\subsection{RD and link with with CG}
We assume a classical CG FEM setting, i.e., a tessellation $\tess$ of the space domain, the space $W_M$ of continuous piecewise polynomial functions and a basis $\left\lbrace \varphi_i \right\rbrace_{i=1,\dots,I}$ of such space satisfying the properties previously mentioned. 
This allows us to consider the continuous approximation $\uapp(x,t):=\sum_{i=1}^I \uvec{c}_i(t) \varphi_i(x)$ of the analytical solution.
Then, the RD approach can be summarized in three main steps
\begin{itemize}
	\item[i)] \textbf{Definition of the element residuals}\\
	For each element $K$ of the tessellation $\tess$, we define the element residual 
	\begin{align}
		\elres:=\int_K \biggl( \frac{\partial}{\partial t}\uvec{u}_h + \textcolor{black}{\biggl[\frac{\partial}{\partial x}   \boldsymbol{F}(\uvec{u}_h) -\boldsymbol{S}(x,\uvec{u}_h) \biggr]_h} \biggr) dx,\quad \forall K\in \tess,
		\label{eq:element_res}
	\end{align}
	which represents an integral balance at the considered element;
	\item[ii)] \textbf{Definition of the node residuals}\\
	For each element $K$, we consider the DoFs belonging to it, $x_i \in K$, and define the node residuals $\noderes$ satisfying the following conservation relation
	\begin{equation}
		\sum_{x_i \in K} \noderes=\elres, \quad \forall K\in \tess,
		\label{eq:conservation}
	\end{equation}
	which corresponds to isolating the contribution of each DoF $x_i\in K$ to the integral balance introduced in the previous step;
	\item[iii)] \textbf{Imposition of the balance at the nodes}\\
	For each DoF $x_i$, we impose an equilibrium between all the node residuals $\noderes$ of the elements that contain that DoF
	\begin{equation}
		\label{eq:nodebalance}
		\sum_{K \in K_i} \noderes=\uvec{0},  \quad \forall i=1,\dots,I,
	\end{equation}
	where we recall that $K_i$ is the set of the elements of the tessellation containing the node $x_i$. This amounts to imposing that the global contribution of each node to all the balances of all the elements that share it is $0$, which is indeed a reasonable constraint: nothing is created or destroyed at the nodes.
\end{itemize}
Equation \eqref{eq:nodebalance} is a system of ODEs in the coefficients $\uvec{c}_i(t)$, which must be solved in time.
The recipe is quite general, as we did not specify how to choose the node residuals $\noderes$. In fact, under this point of view, there are plenty of possibilities and the properties of the resulting scheme depend on this choice. 
In particular, due to \eqref{eq:local_normalization_bases}, one can easily verify that the following definition of the node residuals
\begin{align}
	\noderes:=\int_K \biggl( \frac{\partial}{\partial t}\uvec{u}_h + \textcolor{black}{\biggl[\frac{\partial}{\partial x}   \boldsymbol{F}(\uvec{u}_h) -\boldsymbol{S}(x,\uvec{u}_h) \biggr]_h} \biggr) \varphi_i(x)dx +\textcolor{black}{\uvec{ST}_i^K(\uvec{u}_h)} =\uvec{0}  
	\label{eq:CG_RD}
\end{align}
fulfills the conservation relation \eqref{eq:conservation}, provided that the terms $\uvec{ST}_i^K(\uvec{u}_h)$ are defined in such a way that $\sum_{x_i \in K}\uvec{ST}_i^K(\uvec{u}_h)=\uvec{0}$. Moreover, the resulting RD scheme given by system \eqref{eq:nodebalance}, for such choice of the node residuals, is equivalent to the CG semidiscretization \eqref{eq:CG} if~$\sum_{K \in K_i}\uvec{ST}_i^K(\uvec{u}_h)=\uvec{ST}_i(\uvec{u}_h)$. 
The equivalence is essentially based on the fact that the support of the basis function $\varphi_i$ is in the union of the elements $K\in K_i$. 
The interested reader is referred to \cite{RemiMarioRD}, where several possible choices of the node residuals are presented and the link between RD and several classical approaches, e.g., discontinuous Galerkin and Finite Volume, are analyzed in depth.

\begin{remark}[Generalization to a multidimensional setting]\label{rmk:generalization_to_mulitd}
	Let us remark that the presented formulations, as well as the equivalence between them, extend in a natural way to a multidimensional unstructured framework. 
\end{remark}

\section{Well-balancing}\label{chapWB}
The evolution in time of the numerical solution $\uapp$ is given by the CG/RD formulation \eqref{eq:CG}, which is recalled here for clarity
\begin{align}
	\int_\Omega \biggl( \frac{\partial}{\partial t}\uvec{u}_h + \textcolor{black}{\biggl[\frac{\partial}{\partial x}   \boldsymbol{F}(\uvec{u}_h) -\boldsymbol{S}(x,\uvec{u}_h) \biggr]_h} \biggr) \varphi_i(x)dx +\textcolor{black}{\uvec{ST}_i(\uvec{u}_h)} =\uvec{0},  \quad \forall i=1,\dots,I.
	\label{eq:CG2}
\end{align}
As anticipated, if we are able to design the discretization of the spatial part of the equation and the stabilization term in such a way that they are exactly zero for a particular steady state, then, we will get an exact well-balancing with respect to such steady state for any general time-stepping method. The goal of this section is to do precisely this. 

Generally speaking, there are two main possibilities to achieve well-balancing:
\begin{itemize}
	\item[•] choosing a particular steady equilibrium and define ad hoc the mentioned ingredients of the scheme to be zero with respect to it;
	\item[•] introducing some general structures aiming at preserving \eqref{eq:steady_general} at the discrete level.
\end{itemize}
The second strategy is the most desirable since, as already specified, the analytical expression of the steady states is almost never known in closed-form.
In accordance with the first approach, all the WB elements that will be presented in this section are designed to be exactly zero with respect to the lake at rest steady state \eqref{eq:lakeatrest}; 
nevertheless, some of them address the problem of the preservation of general stationary solutions not known in closed-form.

We will start by presenting a basic non-WB reference framework and, afterwards, we will continue with the definition of some WB alternatives. 
In order to light the notation, in this section we drop the dependence on time, which is not central in this context, being clear that all the space discretizations are performed for a given $\uapp$ and a fixed time $t$.
\subsection{A reference non-well-balanced framework}\label{sec:ref_nonwb}
A consistent space discretization is given by a simple interpolation of the flux and the source onto the functional space $W_M$ 
\begin{align}
	\biggl[\frac{\partial}{\partial x}   \boldsymbol{F}(\uvec{u}_h) -\boldsymbol{S}(x,& \uvec{u}_h) \biggr]_h =\frac{\partial}{\partial x}\boldsymbol{F}_h-\boldsymbol{S}_h, \label{eq:non_wb_space_discretization1}\\
	\boldsymbol{F}_h:=\sum_{i=1}^I \boldsymbol{F}_i \varphi_i(x),& \quad  \boldsymbol{S}_h:=\sum_{i=1}^I \boldsymbol{S}_i \varphi_i(x),
	\label{eq:non_wb_space_discretization2}
\end{align}
with $\boldsymbol{F}_i$ and $\boldsymbol{S}_i$ interpolation coefficients, coinciding with the evaluations at the DoFs, respectively $\boldsymbol{F}(\uapp(x_i))$ and $\boldsymbol{S}(x_i,\uapp(x_i))$, if one assumes a Lagrange basis for $W_M$.

For what concerns the stabilization term, we adopt the CIP stabilization, firstly introduced in \cite{CIP} in an elliptic-parabolic setting by Douglas and Dupont and then applied to the hyperbolic framework in \cite{BURMAN} by Burman and Hansbo. Such stabilization is based on the introduction of a penalization term based on the jump of the normal derivatives of the numerical solution $\uapp$ across the faces of the tessellation, reading in general
\begin{equation}
	\uvec{ST}_i(\uapp):=\sum_{f\in\faces}\sum_{r=1}^R \alpha_{f,r} \int_f \Big\llbracket \nabla^r_{\uvec{\nu}_f} \varphi_i \Big\rrbracket  \Big\llbracket \nabla^r_{\uvec{\nu}_f} \uvec{u}_h \Big\rrbracket  d\sigma(\uvec{x}),\quad \alpha_{f,r}=\delta_r \bar{\rho}_f h_f^{2r},     
	\label{eq:CIP_multiD}
\end{equation}
where $\faces$ denotes the set of the faces shared by two elements of the tessellation, $\llbracket \cdot \rrbracket$ is the jump across the face $f$, $\nabla^r_{\uvec{\nu}_f}$ is the $r$-th partial derivative in the direction $\uvec{\nu}_f$ normal to the face $f$, $\bar{\rho}_f$ is a local reference value for the spectral radius of the normal Jacobian of the flux, $h_f$ is a reference characteristic size of the elements containing $f$ and $\delta_r$ are constant parameters to be tuned. The orientation of the normal $\uvec{\nu}_f$ and the direction of evaluation for the jump can be chosen freely.
Originally, only the jump of the first derivative was taken into account, the stabilization on higher order derivatives has been introduced in \cite{larson2020stabilization}.

Clearly, in a \RIcolor{one-dimensional} context, the faces between the elements are just points and the integrals reduce to point-evaluations. Hence, \eqref{eq:CIP_multiD} reduces to
\begin{equation}
	\uvec{ST}_i(\uapp):=\sum_{f\in\faces}\sum_{r=1}^R \alpha_{f,r} \Bigg\llbracket \frac{\partial^r}{\partial x^r} \varphi_i \Bigg\rrbracket  \Bigg\llbracket \frac{\partial^r}{\partial x^r} \uvec{u}_h \Bigg\rrbracket.
	\label{eq:CIP}
\end{equation}

The CG/RD formulation \eqref{eq:CG}, along with the space discretization \eqref{eq:non_wb_space_discretization1}-\eqref{eq:non_wb_space_discretization2} coupled with the jump stabilization \eqref{eq:CIP}, properly solved in time through a suitable ODE integrator, provides an arbitrary high order framework for the numerical solution of the PDE \eqref{eq:sys}.
Nevertheless, as no particular attention has been paid to design the space discretization and the stabilization in such a way to achieve well-balancing, the resulting formulation is not WB.

Actually, neither the space discretization nor the jump stabilization, taken individually, are zero with respect to any particular steady state.
In fact, the naive interpolation \eqref{eq:non_wb_space_discretization2} of the flux and the source leads to a natural mismatch preventing any possibility of well-balancing, as  $\frac{\partial}{\partial x}\boldsymbol{F}_h$ and $\boldsymbol{S}_h$ belong to two different polynomial spaces and their difference can be zero only in very trivial cases. Further, in the context of the lake a rest steady state,
the jump of the derivatives of $H_h$ across the interfaces, in the first component of \eqref{eq:CIP}, leads to a lack of well-balancing.

In the following, we will introduce some possible WB substitutes.
We conclude this section with some final remarks.


\begin{remark}
	The CIP stabilizations can be naturally put in an RD formalism, even in a general multidimensional setting, as shown in the next proposition.
\end{remark}
\begin{proposition}\label{prop:CIP_RD}
	Under the assumption of a conformal tessellation, if we define
	\begin{equation}
		\ST(\uapp):= \sum_{\substack{f\subset\partial{K}\\f\in \faces}}\sum_{r=1}^R \alpha_{f,r} \int_f  \nabla^r_{\uvec{\nu}_f}  \varphi_i\vert_K  \Big\llbracket \nabla^r_{\uvec{\nu}_f} \uvec{u}_h \Big\rrbracket   d\sigma(\uvec{x}),
		\label{eq:CIP_RD}
	\end{equation}
	where here the jump is evaluated from the inside of $K$ to the neighboring element $K'$ sharing $f$, $\llbracket z \rrbracket:=z\vert_K-z\vert_{K'}$, then we have that
	\begin{itemize}
		\item[•] $\sum_{\uvec{x}_i \in K}\uvec{ST}_i^K(\uvec{u}_h)=\uvec{0}$;
		\item[•] the stabilization term \eqref{eq:CIP_multiD} is given by $\sum_{K \in K_i}\uvec{ST}_i^K(\uvec{u}_h)=\uvec{ST}_i(\uvec{u}_h)$.
	\end{itemize}
\end{proposition}
In the previous proposition, the bold font has been used for the generic DoF $\uvec{x}_i\in K$ in order to emphasize the fact that the result holds in a general multidimensional setting. The proof can be found in Appendix \ref{app:CIP_RD}.

\begin{remark}[Arbitrary high order stabilizations]
	Not all the stabilizations allow to reach arbitrary high order. For example, the Lax-Friedrichs stabilization presented in \cite{abgrall2019high}, in the context of the local Lax–Friedrichs node residuals in an RD setting, is at most first order accurate.
\end{remark}

\subsection{Well-balanced space discretizations}
		
		We will introduce here two WB discretizations of the spatial part of Equation~\eqref{eq:sys} with respect to the lake at rest steady state \eqref{eq:lakeatrest}. While the second one is strongly based on the assumption of a \RIcolor{one-dimensional} framework, the first one can be easily generalized to a multidimensional setting. Before starting, it is useful to introduce here the following splitting of the flux and the source 
		\begin{alignat}{3}
			\uvec{F}&=\uvec{F}^{V}+\uvec{F}^{HS},\quad &\uvec{F}^{V}&:=
			\begin{pmatrix}
				q\\
				\frac{q^2}{H}\end{pmatrix}, \quad &\uvec{F}^{HS}&:= \begin{pmatrix}
				0\\
				g\frac{H^2}{2}
			\end{pmatrix},\label{eq:flux_VHS}\\ 
			\boldsymbol{S}&=\uvec{S}^{V}+\uvec{S}^{HS},\quad &\uvec{S}^{V}&:=-\begin{pmatrix}
				0\\
				g\frac{n_M^2\vert q \vert}{H^{\frac{7}{3}}}q
			\end{pmatrix}, \quad &\uvec{S}^{HS}&:=-\begin{pmatrix}
				0\\
				gH\frac{\partial}{\partial x} B\end{pmatrix},
			\label{eq:source_VHS}
		\end{alignat}
		where the superscripts ``$V$'' and ``$HS$'' are used in order to identify respectively the velocity and the hydrostatic parts.

		\subsubsection{WB-$\HS$}\label{sec:WBHS}
		This WB discretization, presented in \cite{ricchiuto2009stabilized} and here denoted by ``WB-$\HS$'', relies on a particular treatment of the terms  $\frac{\partial}{\partial x}\left(g\frac{H^2}{2}\right)$ and $gH\frac{\partial}{\partial x} B$. 
		Rather than simply interpolating $\boldsymbol{F}$ and $\boldsymbol{S}$, we consider the splitting \eqref{eq:flux_VHS}-\eqref{eq:source_VHS}.
		In the context of a lake at rest steady state, the velocity parts of the flux and of the source are identically zero as $q\equiv 0$, therefore, one can easily discretize such terms with a simple interpolation
		\begin{align}
			\boldsymbol{F}^V_h=\sum_{i=1}^I \boldsymbol{F}^V_i \varphi_i(x),& \quad  \boldsymbol{S}^V_h=\sum_{i=1}^I \boldsymbol{S}^V_i \varphi_i(x).
		\end{align}
		A WB treatment of the hydrostatic part is less trivial. The mentioned approach consists in interpolating separately the water height and the bathymetry, thus getting
		\begin{align}
			&\biggl[\frac{\partial}{\partial x}   \boldsymbol{F}^{HS} -\boldsymbol{S}^{HS} \biggr]_h=
			\begin{pmatrix}
				0\\
				gH_h \frac{\partial}{\partial x} H_h
			\end{pmatrix}+\begin{pmatrix}
				0\\
				gH_h\frac{\partial}{\partial x} B_h
			\end{pmatrix}=\begin{pmatrix}
				0\\
				gH_h\frac{\partial}{\partial x}(H_h+B_h)
			\end{pmatrix},
		\end{align}
		where, by linearity of the interpolation, $(H_h+B_h)=(H+B)_h=\eta_h$ which is constant in the context of the lake at rest steady state, leading to an exact well-balancing. To sum up, the final WB discretization reads
		\begin{align}
			\biggl[\frac{\partial}{\partial x}   \boldsymbol{F} -\boldsymbol{S} \biggr]_h:=\frac{\partial}{\partial x}\boldsymbol{F}^V_h-\boldsymbol{S}^V_h+\begin{pmatrix}
				0\\
				gH_h\frac{\partial}{\partial x}(H_h+B_h)
			\end{pmatrix},
			\label{eq:WBHS}
		\end{align}
		where the subscript $h$ at the right-hand side indicates a simple interpolation.
		\RIcolor{
			\begin{remark}[On the conservation property of WB-HS]\label{rmk:conservation_WBHS}
				Due to the fact that $\sum_{i=1}^I\varphi_i\equiv 1$, it holds that
				\begin{align}
					\begin{split}
						\sum_{i=1}^I\int_K \left(gH_h \frac{\partial}{\partial x} H_h \right)\varphi_i(x)dx  =\int_K gH_h \frac{\partial}{\partial x} H_hdx =\int_K \frac{\partial}{\partial x} \left(g\frac{H_h^2}{2}\right)dx.
					\end{split}
				\end{align}
				Conservation of discretization~\ref{eq:WBHS} is thus guaranteed, with respect to the usual definition adopted in the context of RD schemes \cite{RemiMarioRD,ricchiuto2010explicit,abgrall2023personal}, if all integrals are computed exactly.
				Nonetheless, thanks to the linearity of quadrature formulas, such property is not spoiled as long as the adopted quadrature formula is exact for polynomials of degree~$2M-1$. 
				This is always the case in our simulations. In fact, we consider an exact computation of the integrals in all cases but for Lagrange polynomials associated to GL nodes, for which the associated quadrature formula, exact up to degree~$2M-1$, is adopted.		
			\end{remark}
		}

		\subsubsection{WB-$\GF$}
		The global flux approach has been firstly introduced in \cite{gascon2001construction} and has already been employed in many works \cite{chertock2018well,ciallella2023arbitrary,mantri2024fully,chertock2022well,cao2023flux,cao2022flux,kurganov2023well}
		to design WB methods. 
		In particular, in \cite{abgrall2022hyperbolic}, it has been shown how the notion of global flux can be naturally embedded in RD formulations.

		The underlying idea is to define a new flux $\boldsymbol{G}$ keeping into account the source term in order to rephrase the original PDE \eqref{eq:sys} into an equivalent homogeneus formulation
		\begin{equation}
			\frac{\partial}{\partial t}\uvec{u}+\frac{\partial}{\partial x} \boldsymbol{G}=\uvec{0}.
		\end{equation}
		Despite being absolutely non-trivial in a multidimensional setting, in the \RIcolor{one-dimensional} case one can easily define the global flux through a simple integration of the source term
		\begin{equation}
			\uvec{R}:=-\int^x_{x_L} \boldsymbol{S}(s,\uvec{u}(s))ds,\quad \boldsymbol{G}=\boldsymbol{F}+\uvec{R}.
		\end{equation}
		
		At the discrete level, $\boldsymbol{G}_h$ is got by interpolation, providing an approximation at each DoF, and the simplest idea that one could have is to set for any $i$
		\begin{align}
			\boldsymbol{G}_h(x_i) :=\textcolor{black}{\boldsymbol{F}_h(x_i)}&+\textcolor{black}{\uvec{R}_h}(x_i), \\
			\textcolor{black}{\boldsymbol{F}_h(x_i)}:=\sum_{j=1}^I\uvec{F}_j\varphi_j(x_i), \quad &\textcolor{black}{\textcolor{black}{\uvec{R}_h}(x_i)}:=-\int^{x_i}_{x_L}  \sum_{j=1}^I\uvec{S}_j \varphi_j(s)ds,
		\end{align}
		with $\boldsymbol{F}_h$ being the interpolation of the flux and $\uvec{R}_h$ the integral of the interpolation of the source.
		Again, we remark that the dependence on time is dropped in order to light the notation.
		
		Unfortunately, despite this choice providing a consistent discretization of the spatial part of our PDE, 
		given by $\frac{\partial}{\partial x}\uvec{G}_h$,
		this formulation is not WB with respect to the lake at rest steady state ($\uvec{G}_h\neq \text{const}$), as no special care has been taken under this point of view. 
		In fact, in such a case, the flux and the source reduce to their hydrostatic part
		\begin{align}
			\uvec{F}=\uvec{F}^{HS}=\begin{pmatrix}
				0\\
				g\frac{H^2}{2}
			\end{pmatrix}&,\quad
			\boldsymbol{S}=\uvec{S}^{HS}=-\begin{pmatrix}
				0\\
				gH\frac{\partial}{\partial x} B
			\end{pmatrix},
			\label{eq:flux_source_lake_at_rest}
		\end{align}
		and there is no reason why the integral of the interpolation of the second component of the source should match the second component of the flux.
		A WB alternative is the one presented in \cite{xing2005high,ciallella2023arbitrary} and consists in adopting, in each element $K$, the following discretization for the second component of $\uvec{S}^{HS}$
		\begin{equation}
			\left[gH\frac{\partial}{\partial x} B\right]_h:=\left[ g (H_h+B_h) \frac{\partial}{\partial x} B_h \right]_h- \frac{\partial}{\partial x} \left[ \frac{gB^2}{2}\right]_h,\quad \forall x\in K, \quad \forall K \in \tess,
			\label{eq:WB_second_component_of_the_source_GF}
		\end{equation}
		where again the subscripts at the right-hand side stand for simple interpolations. 
		More formally, we can state the following proposition.
		\begin{proposition}\label{prop:G_lake_at_rest}
			By adopting the discretization \eqref{eq:WB_second_component_of_the_source_GF} for the second component of the hydrostatic part of the source and a simple interpolation of the velocity part, the resulting global flux got by interpolating its values at the DoFs 
			\begin{align}
				\boldsymbol{G}_h(x_i):=\boldsymbol{F}_h&(x_i)+\textcolor{black}{\uvec{R}_h}(x_i), \label{eq:Gh1}\\
				\boldsymbol{F}_h(x_i):=\sum_{j=1}^I\uvec{F}_j\varphi_j(x_i), \quad \textcolor{black}{\uvec{R}_h}(x_i):=&-\int^{x_i}_{x_L}\left[ -\begin{pmatrix}
					0\\
					\left[gH\frac{\partial}{\partial x} B\right]_h(s)
				\end{pmatrix}+  \sum_{j=1}^I\uvec{S}^{V}_j \varphi_j(s) \right] ds,\label{eq:Gh2}
			\end{align}
			is constant for a lake at rest steady state.
		\end{proposition}
		The proof can be found in Appendix \ref{app:WB_proof}. 
		Summarizing, the WB discretization based on the notion of global flux, here denoted as ``WB-$\GF$'', reads
		\begin{align}
			\biggl[\frac{\partial}{\partial x}   \boldsymbol{F} -\boldsymbol{S} \biggr]_h:=\frac{\partial}{\partial x}\boldsymbol{G}_h,
			\label{eq:WBGF}
		\end{align}
		with $\uvec{G}_h$ defined by interpolating its values at the DoFs given by \eqref{eq:Gh1}-\eqref{eq:Gh2}.
		
		\begin{remark}[Local interpolation]\label{rmk:local_interpolation}
			We remark that the discretization \eqref{eq:WB_second_component_of_the_source_GF} is meant to be performed separately in each element, as the term $\frac{\partial}{\partial x} B_h$ is in general discontinuous across the interfaces between the elements.
		\end{remark}
		
		\begin{remark}[Global flux quadrature]
			The global flux can be also interpreted as a suitable modification of the quadrature formula adopted in order to evaluate the source integral in \eqref{eq:CG}. For more information, the reader is referred to \cite{mantri2024fully}.
		\end{remark}
		
		\subsection{Well-balanced continuous interior penalty stabilizations}
		We start by recalling the original non-WB CIP stabilization \eqref{eq:CIP} in a \RIcolor{one-dimensional} setting
		\begin{align}
			\uvec{ST}_i(\uapp)&:=\sum_{f\in\faces}\sum_{r=1}^R \alpha_{f,r} \Bigg\llbracket \frac{\partial^r}{\partial x^r} \varphi_i \Bigg\rrbracket   \Bigg\llbracket \frac{\partial^r}{\partial x^r} \uvec{u}_h \Bigg\rrbracket\\
			&=\sum_{f\in\faces}\sum_{r=1}^R \alpha_{f,r} \Bigg\llbracket \frac{\partial^r}{\partial x^r} \varphi_i \Bigg\rrbracket   \Bigg\llbracket \frac{\partial^r}{\partial x^r} \begin{pmatrix}
				H_h\\
				q_h
			\end{pmatrix} \Bigg\rrbracket.  
			\label{eq:jc}
		\end{align}
		This stabilization is based on the jump of the derivatives of the conserved variables, this is why we will refer to it as ``\textbf{jc}''.
		
		We propose here some novel WB alternative CIP stabilizations. The main idea is to change the object of the stabilization in such a way to achieve well-balancing. In the following definitions the subscript $h$ denotes an interpolation.
		\begin{itemize}
			\item[•] \textbf{Total height (jt)}
			\begin{align}
				\uvec{ST}_i(\uapp):=\sum_{f\in\faces}\sum_{r=1}^R \alpha_{f,r} \Bigg\llbracket \frac{\partial^r}{\partial x^r} \varphi_i \Bigg\rrbracket   \Bigg\llbracket \frac{\partial^r}{\partial x^r} \begin{pmatrix}
					\eta_h\\
					q_h
				\end{pmatrix} \Bigg\rrbracket  , \quad \eta_h:=H_h+B_h.
				\label{eq:jt}
			\end{align}
			\item[•] \textbf{Entropy variables (je)}
			\begin{align}
				\uvec{ST}_i(\uapp):=\sum_{f\in\faces}\alpha_f  \Bigg\llbracket \frac{\partial}{\partial x} \varphi_i\Bigg\rrbracket \mathcal{A}_f  \Bigg\llbracket \frac{\partial}{\partial x} \uvec{w}_h \Bigg\rrbracket, \quad \uvec{w}:=\begin{pmatrix}
					g\eta-\frac{v^2}{2}\\
					v
				\end{pmatrix},
				\label{eq:je}
			\end{align}
			where $\mathcal{A}_f\in \mathbb{R}^{2\times 2}$ is a matrix which is used to make the stabilization dimensionally consistent. One possible choice for it, the one assumed here, is given by the Jacobian of the transformation from entropy to conserved variables 
			\begin{equation}
				\mathcal{A}_f:=\frac{\partial \uvec{u}}{\partial \uvec{w}}=\begin{pmatrix}
					\frac{1}{g} & \frac{v}{g}\\
					\frac{v}{g} & H+\frac{v^2}{g}
				\end{pmatrix},
			\end{equation}
			evaluated at the interface $f$ and computed assuming a flat bathymetry in such a way to have a proper bijective map between $\uvec{w}$ and $\uvec{u}$.
			
			
			\item[•] \textbf{Space residual (jr)}
			\begin{align}
				\uvec{ST}_i(\uapp):=\sum_{f\in\faces}\alpha_f  \Bigg\llbracket \uvec{J}  \frac{\partial}{\partial x} \varphi_i \Bigg\rrbracket \mathcal{B}_f \Bigg\llbracket \textcolor{black}{\uvec{J}\frac{\partial}{\partial x} \uvec{u}_h-\boldsymbol{S}^*} \Bigg\rrbracket, \quad \boldsymbol{S}^*:=-\begin{pmatrix}
					0\\
					gH_h \frac{\partial}{\partial x}B_h
				\end{pmatrix},
				\label{eq:jr}
			\end{align}
			where $\uvec{J}:=\frac{\partial \uvec{F}}{\partial \uvec{u}}$ is the Jacobian of the flux \eqref{eq:jacobian} at the interface $f$ and the matrix $\mathcal{B}_f\in \mathbb{R}^{2\times 2}$, just like $\mathcal{A}_f$, is used for consistency purposes. In this work, we assume $\mathcal{B}_f:=\vert \uvec{J} \vert^{-1}$ with $\vert \uvec{J} \vert$ absolute value of the Jacobian, defined as $|\uvec{J}|:=R\vert \Lambda \vert R^{-1}$, where $R$ is the matrix of the right eigenvectors of $\uvec{J}$ and $\vert \Lambda \vert$ a diagonal matrix having as entries the absolute values of the eigenvalues of $\uvec{J}$. More explicitly, for the sake of completeness, $R$ and $\vert \Lambda \vert$ are respectively given by
			\begin{equation}
				R:=\begin{pmatrix}
					1 & 1\\
					v-c & v+c
				\end{pmatrix}, \quad \vert \Lambda \vert:=\begin{pmatrix}
					\vert v-c \vert & 0\\
					0   & \vert v+c \vert
				\end{pmatrix}.
			\end{equation}

			

			
			\item[•] \textbf{Global flux (jg)}
			\begin{align}
				\uvec{ST}_i(\uapp):=\sum_{f\in\faces}\alpha_f   \Bigg\llbracket \uvec{J} \frac{\partial}{\partial x} \varphi_i\Bigg\rrbracket \mathcal{B}_f  \Bigg\llbracket \frac{\partial}{\partial x} \uvec{G}_h \Bigg\rrbracket,
				\label{eq:jg}
			\end{align}
			where again the presence of $\mathcal{B}_f:=\vert \uvec{J} \vert^{-1}$ allows to make the stabilization consistent with the other elements in \eqref{eq:CG}. 
		\end{itemize}
		The abbreviations are based on the objects of the stabilizations.
		Before going to the numerical results, we make some useful remarks.
		
		All the new CIP stabilizations are WB with respect to the lake at rest steady state, as shown in the next proposition. 
		\begin{proposition}
			The stabilization terms \eqref{eq:jt}, \eqref{eq:je}, \eqref{eq:jr} and \eqref{eq:jg} are exactly zero in the context of the lake at rest steady state \eqref{eq:lakeatrest}. 
		\end{proposition}
		\begin{proof}
			The claim is trivial for jt \eqref{eq:jt} and je \eqref{eq:je}: the terms dependent on $q$ and $v$ cancel and we are left with the terms $\overline{\eta}$ and $g\overline{\eta}$, which, in such a case, are constant and, therefore, have zero derivatives and related jump. Clearly, the same holds for jg \eqref{eq:jg} because, as shown in Proposition \ref{prop:G_lake_at_rest}, the global flux $\uvec{G}_h$ has been specifically designed to be constant for a lake at rest steady state. For what concerns instead jr \eqref{eq:jr}, we have that the argument of the second jump, $\uvec{J}\frac{\partial}{\partial x} \uvec{u}_h-\boldsymbol{S}^*$, for a lake at rest steady state reduces to
			\begin{equation}
				\uvec{J}\frac{\partial}{\partial x} \uvec{u}_h-\boldsymbol{S}^*=\begin{pmatrix}
					0 & 1\\
					gH_h & 0
				\end{pmatrix}\begin{pmatrix}
					\frac{\partial}{\partial x}H_h\\
					0
				\end{pmatrix}+\begin{pmatrix}
					0\\
					gH_h \frac{\partial}{\partial x}B_h
				\end{pmatrix}=\begin{pmatrix}
					0\\
					gH_h\frac{\partial}{\partial x}\left( H_h+B_h\right)
				\end{pmatrix},
			\end{equation}
			which is indeed zero because $(H_h+B_h)=\eta_h\equiv \overline{\eta}$ is constant.
			In practice, the definition of $\uvec{S}^*$ is given in such a way to mimic the trick of the first WB space discretization, WB-$\HS$, presented in Section \ref{sec:WBHS}.
		\end{proof}

		The last two stabilizations, \eqref{eq:jr} and \eqref{eq:jg}, are particularly interesting, as in such cases the stabilization is based on discretizations of $\frac{\partial}{\partial x} \uvec{F}-\uvec{S}$, a quantity which is supposed to be zero in the context of a general steady state, not only in the context of the lake at rest.  
		In fact, as we are going to see in the numerical experiments, they have special properties in terms of superconvergence and capturing of small perturbations of general stationary solutions.
		
		\begin{remark}[About non-differential terms]
			One could wonder why, in the context of jr \eqref{eq:jr}, the term $\boldsymbol{S}^*$ seems to take into account only the hydrostatic part of the source $\uvec{S}^{HS}$, defined in \eqref{eq:source_VHS}, and not $\uvec{S}^{V}$. 
			The point is that the remaining velocity part $\uvec{S}^{V}$ has no differential terms and, thus, it would cancel due to the assumption of a continuous representation of the numerical solution.
			Under this point of view, it is worth underlying that the terms $\frac{\partial \uvec{u}}{\partial \uvec{w}}$, $\uvec{J}$ and $\vert\uvec{J}\vert^{-1}$ are non-differential and therefore well-defined at each interface $f$. 
		\end{remark}

		\begin{remark}[Another possible choice for $\mathcal{B}_f$]
			As already pointed out, the matrices $\mathcal{A}_f$ and $\mathcal{B}_f$ are used in order to achieve dimensional consistency. Other choices, with respect to the ones presented, are possible.
			In particular, another valid option for $\mathcal{B}_f$ is given by $\rho_f^{-1}\mathbb{I}$, with $\rho_f$ being the spectral radius of $\uvec{J}$ at the interface $f$ and $\mathbb{I}\in \mathbb{R}^{2\times 2}$ the identity matrix.
			The numerical results got with the two definitions of $\mathcal{B}_f$ are qualitatively similar but slightly better with $\mathcal{B}_f:=\vert \uvec{J}\vert^{-1}$. Therefore, for the sake of compactness, we will only present the ones obtained for such definition.
		\end{remark}

		\begin{remark}
			The corresponding terms $\ST$ for the definitions of the new stabilizations in an RD setting are respectively given by
			\begin{align}
				\ST(\uapp)&:=\sum_{\substack{f\subset\partial{K}\\f\in \faces}}\sum_{r=1}^R \alpha_{f,r}  \frac{\partial^r}{\partial x^r} \restrict{\varphi_i}{K}   \Bigg\llbracket \frac{\partial^r}{\partial x^r} \begin{pmatrix}
					\eta_h\\
					q_h
				\end{pmatrix} \Bigg\rrbracket,\\
				\ST(\uapp)&:=\sum_{\substack{f\subset\partial{K}\\f\in \faces}}\alpha_f   \frac{\partial}{\partial x} \restrict{\varphi_i}{K} \mathcal{A}_f  \Bigg\llbracket \frac{\partial}{\partial x} \uvec{w}_h \Bigg\rrbracket,\\
				\ST(\uapp)&:=\sum_{\substack{f\subset\partial{K}\\f\in \faces}}\alpha_f  \uvec{J}  \frac{\partial}{\partial x} \restrict{\varphi_i}{K} \mathcal{B}_f \Bigg\llbracket \textcolor{black}{\uvec{J}\frac{\partial}{\partial x} \uvec{u}_h-\boldsymbol{S}^*} \Bigg\rrbracket,\\
				\ST(\uapp)&:=\sum_{\substack{f\subset\partial{K}\\f\in \faces}}\alpha_f    \uvec{J} \frac{\partial}{\partial x} \restrict{\varphi_i}{K} \mathcal{B}_f  \Bigg\llbracket \frac{\partial}{\partial x} \uvec{G}_h \Bigg\rrbracket,
			\end{align}
			where the convention on the direction of evaluation of the jump is the one adopted in the context of Proposition \ref{prop:CIP_RD}, $\llbracket z \rrbracket:=z\vert_K-z\vert_{K'}$.
		\end{remark}

		\begin{remark}[Extension to a multidimensional setting]
			Apart from jg \eqref{eq:jg} which requires the notion of global flux and, hence, is strictly related to a \RIcolor{one-dimensional} setting, the other jump stabilizations can be easily generalized to a multidimensional context.
		\end{remark}
		\RIcolor{%
			\begin{remark}[On handling shocks]
				CIP stabilizations are not suitable to handle shocks.
				For this purpose, other elements should be introduced, such as nonlinear limiting or blendings with diffusive low order stabilizations.
				However, this is out of the scope of this work and, hence, tests with shocks have been omitted.		
			\end{remark}
		}

		\section{Deferred Correction}\label{chapDECRD}
		Aiming at an arbitrary high order framework, once the discretization in space has been fixed, we need to select a suitable arbitrary high order time integration technique for the numerical resolution of the CG/RD semidiscretization \eqref{eq:CG}.
		For this purpose, we adopt a DeC time discretization.
		
		Originally introduced in \cite{fox1949some}, the DeC approach has been extensively developed over the years. Several formulations have been proposed \cite{Decoriginal,minion2003semi,Decremi}, with applications to many different fields, for example \cite{micalizzi2023efficient,ciallella2022arbitrary,veiga2024improving,mPDeC,abgrall2021relaxation,han2021dec,abgrall2020high}, 
		ranging from adaptivity to structure preservation.

		In particular, here we consider the bDeCu method, presented in \cite{micalizzi2023new} as an efficient modification, based on ideas introduced by Minion in \cite{minion2003semi}, of the DeC formulation presented in \cite{Decremi}.
		The advantage of such formulation is that it allows to get rid of the burden associated with the big and sparse mass matrix, typical of CG/RD discretizations.
		In particular, its size and sparse structure make its inversion unfeasible in concrete applications, determining, in general, high computational costs related to the resolution, at each time iteration, of several linear systems for standard time integration methods. 
		Moreover, the adoption of some particular stabilization terms $\uvec{ST}_i(\uapp)$, dependent on the time derivative of the approximated solution, may determine contributions to the mass matrix, leading to a new solution-dependent mass matrix, which implies heavy complications.
		In fact, the mass matrix should be recomputed multiple times at each time iteration and, in general, no warranties hold concerning its invertibility (and so concerning the well--posedness of the resulting method), see for example \cite{abgrall2020multidimensional} in a Lagrangian framework.
		All these problems do not exist with the adopted approach.

		We will introduce now the DeC formulation for hyperbolic problems defined in \cite{Decremi} and, afterwards, we will describe the efficient modification introduced in \cite{micalizzi2023new}.
		The idea behind the approach is based on having two operators dependent on a same discretization parameter $\Delta$ and associated to different discretizations of the same problem, $\lopd^1,\lopd^2:X \longrightarrow Y$, with $X$ and $Y$ normed vector spaces.
		The operator $\lopd^2$ corresponds to a high order and implicit discretization and is, hence, difficult to solve, while, $\lopd^1$ corresponds to a low order and explicit discretization and, in particular, we assume that it is easy to solve problems of the type $\lopd^1(\undu)=\undz$ for $\undz\in Y$ given.
		Due to its simplicity, we would prefer solving $\lopd^1$, rather than $\lopd^2$, however, the solution to such operator would not be accurate enough.
		Under some assumptions on the operators, we can consider $\underline{\uvec{u}}^{(p)}$ given by the following iterative procedure
		\begin{equation}
			\label{eq:DeC}
			\lopd^1(\underline{\uvec{u}}^{(p)})=\lopd^1(\underline{\uvec{u}}^{(p-1)})-\lopd^2(\underline{\uvec{u}}^{(p-1)}), \quad p\geq 1,
		\end{equation}
		which is subjected to the following accuracy estimate with respect to the solution $\usol$ to the high order operator $\lopdt^2$
		\begin{equation}
			\label{eq:accuracyestimate}
			\norm{\underline{\uvec{u}}^{(p)}-\usol}_X \leq \left( \Delta \frac{\alpha_2}{\alpha_1} \right)^p\norm{\underline{\uvec{u}}^{(0)}-\usol}_X.
		\end{equation}
		Let us observe that the updating formula \eqref{eq:DeC} is explicit as a result of the assumptions made on the operator $\lopd^1$. Moreover, thanks to the accuracy estimate \eqref{eq:accuracyestimate}, the convergence to the solution $\usol$ of the operator $\lopd^2$ is ensured for $\Delta t$ small enough, independently of the chosen $\undu^{(0)}$, and the number of iterations to achieve a given accuracy with respect to it is controlled.

		In our case, the reference problem is the semidiscrete formulation \eqref{eq:CG}, which can be rephrased more compactly as
		\begin{align}
			\sum_{j=1}^I \left(\int_\Omega \varphi_i \varphi_j dx\right)\frac{d}{d t} \uvec{c}_j(t)+\spacestuff(\uvec{c}(t))=\uvec{0},    \quad \forall i=1,\dots,I,
			\label{eq:CG_compact}
		\end{align}
		with $\spacestuff(\uvec{c}(t))$ containing the terms not related to the time derivative.
		As in the context of a classical one-step method, we assume to know an approximation $\uvec{c}_n\approx\uvec{c}(t_n)$ of the solution to system \eqref{eq:CG} at the generic time $t_n$ and we look for an approximation $\uvec{c}_{n+1}\approx\uvec{c}(t_{n+1})$ at $t_{n+1}:=t_{n}+\Delta t$. 
		
		In order to define the operators $\lopd^1$ and $\lopd^2$, we introduce $M+1$ equispaced subtimenodes in the interval $[t_n,t_{n+1}]$, such that $t_n=t^0<t^1<\dots<t^M=t_{n+1}$. Then, the operator $\lopdt^2:\mathbb{R}^{(I \times Q \times M)}\rightarrow \mathbb{R}^{(I \times Q \times M)}$ is given by
		\begin{align}
			\lopdt^2(\cund)&=\left(\mathcal{L}^2_{\Delta,1}(\cund),\mathcal{L}^2_{\Delta,2}(\cund),\dots,\mathcal{L}^2_{\Delta,I}(\cund)\right) ,
		\end{align}
		\begin{align}
			\lopdi^2(\cund)&=\begin{pmatrix}
				\sum_{j=1}^I \left(\int_\Omega \varphi_i \varphi_j d x\right)\left( \uvec{c}_j^1-\uvec{c}_j^0\right)+\Delta t \sum_{\ell=0}^{M} \theta^1_\ell\spacestuff(\uvec{c}^\ell)\\
				\vdots\\
				\sum_{j=1}^I \left(\int_\Omega \varphi_i \varphi_j d x\right)\left( \uvec{c}_j^M-\uvec{c}_j^0\right)+\Delta t \sum_{\ell=0}^{M} \theta^M_\ell\spacestuff(\uvec{c}^\ell)\\
			\end{pmatrix}, \quad \cund=\begin{pmatrix}
				\uvec{c}^1\\
				\vdots\\
				\uvec{c}^M
			\end{pmatrix},
		\end{align}
		with $\uvec{c}_j^0$ known components of $\uvec{c}^0:=\uvec{c}_n$ and $Q$ number of scalar equations in the original hyperbolic system, $2$ in this case. Its definition is based on replacing the function $\spacestuff(\uvec{c}(t))$ by a high order interpolation in time with the Lagrange polynomials $\psi^m$ associated to the subtimenodes $t^m$.
		The generic coefficient $\theta^m_\ell$ is, in fact, the normalized integral of the function $\psi^m$ over $[t^0,t^m]$. 
		The solution $\underline{\uvec{c}}_\Delta$, such that $\lopd^2(\underline{\uvec{c}}_\Delta)=\uvec{0}$, contains $M$ components $\uvec{c}_\Delta^m$ $m=1,\dots,M$, which are $(M+1)$-th order accurate approximations of the exact solution $\uvec{c}(t^m).$ Trying to solve such operator directly is indeed very complicated as it amounts to solving a huge nonlinear system of algebraic equations.
		
		The low order explicit operator $\lopdt^1:\mathbb{R}^{(I \times Q \times M)}\rightarrow \mathbb{R}^{(I \times Q \times M)}$ is, instead, given by
		\begin{equation}
			\lopdt^1(\cund)=\left(\mathcal{L}^1_{\Delta,1}(\cund),\mathcal{L}^1_{\Delta,2}(\cund),\dots,\mathcal{L}^1_{\Delta,I}(\cund)\right), 
			\label{l1DeC}
		\end{equation}
		\begin{equation}
			\lopdi^1(\cund)=\begin{pmatrix}
				C_i\left( \uvec{c}_i^1-\uvec{c}_i^0\right)+\Delta t \beta^1\spacestuff(\uvec{c}^0)\\
				\vdots\\
				C_i\left( \uvec{c}_i^M-\uvec{c}_i^0\right)+\Delta t \beta^M\spacestuff(\uvec{c}^0)\\
			\end{pmatrix}, \quad \cund=\begin{pmatrix}
				\uvec{c}^1\\
				\vdots\\
				\uvec{c}^M
			\end{pmatrix},
			\label{l1iDeC}
		\end{equation}
		with $\beta^m:=\frac{t^m-t^0}{\Delta t}$ and $C_i:=\int_\Omega \varphi_i dx$. Such definition is based on an Euler approximation in time and a first order mass lumping in space. The components of the solution to $\lopd^1$ are, in fact, first order accurate approximations of the exact solution to \eqref{eq:CG} in the subtimendodes $t^m$ $m=1,\dots,M$.
		
		By a direct computation, the updating formula \eqref{eq:DeC}, in this case reduces to 
		{\small
			\begin{equation}
				\left(
				\begin{array}{ccc}
					\vdots\\
					\uvec{c}_i^{m,(p)}\\
					\vdots\\
				\end{array}
				\right)=\left(
				\begin{array}{ccc}
					\vdots\\
					\uvec{c}_i^{m,(p-1)}\\
					\vdots\\
				\end{array}
				\right)-\frac{1}{C_i}\begin{pmatrix}
					\vdots\\
					\sum_{j=1}^I \left(\int_K \varphi_i \varphi_j d x\right)\left( \uvec{c}_j^{m,(p-1)}-\uvec{c}_j^0\right)+\Delta t \sum_{\ell=0}^{M} \theta^m_\ell\spacestuff(\uvec{c}^{\ell,(p-1)})\\
					\vdots\\
				\end{pmatrix},
				\label{eq:generalupdating}
		\end{equation}}
		for any $i$. 
		
		The vectors $\uvec{c}_i^{m,(p)}$ are the components of $\uvec{c}^{m,(p)}$ which is, itself, a component of $\underline{\uvec{c}}^{(p)}$, the output vector at the $p$-th iteration. 
		For what concerns the initial vector $\underline{\uvec{c}}^{(0)}$, the most reasonable choice is to set for any subtimenode $\underline{\uvec{c}}^{m,(0)}=\uvec{c}^0:=\uvec{c}_n$.
		The updating does not involve the solution at the subtimenode $t^0$, therefore, we also have $\underline{\uvec{c}}^{0,(p)}=\uvec{c}^0:=\uvec{c}_n$ for any $p$. 
		In the end, we set $\uvec{c}_{n+1}:=\uvec{c}^{M,(P)}$, where $P$ is the final number of iterations performed. 
		The optimal number of iterations is given by $P:=M+1$, as the accuracy of the approximation with respect to the solution of $\lopd^2$ increases by one at each iteration but we are not interested in approximating it with accuracy higher than the one of the underlying discretization.

		One can see that the well-posedness of the explicit update \eqref{eq:generalupdating} is strongly related to the fact that the coefficients $C_i$ in \eqref{l1iDeC} are not zero. This is not always the case for any choice of the polynomial basis. 
		A safe option is given by the Bernstein polynomials, for which we always have $C_i>0$ $\forall i$. Another possibility, particularly convenient, is to choose a basis of Lagrange polynomials associated to points defining a quadrature formula sufficiently accurate to guarantee a high order mass lumping, e.g., the GL points, and to adopt such quadrature for the integrals. 
		
		The modification presented in \cite{micalizzi2023new} consists in introducing interpolation processes between the iterations to increase the number of subtimenodes according to the order of accuracy achieved in the specific iterations. In particular, at the iteration $p$, the bDeCu method involves an interpolation in time of $\underline{\uvec{c}}^{(p-1)}$, associated to $p$ subtimenodes $t^{m,(p-1)}$, to get $\underline{\uvec{c}}^{*(p-1)}$, associated to $p+1$ subtimenodes $t^{m,(p)}$. Such vector is then used to perform the iteration via \eqref{eq:generalupdating}. 
		The updating formula stays formally identical, up to the fact that the coefficients $\theta^m_\ell$, at each iteration $p$, are the ones associated to the considered subtimenodes $t^{m,(p)}$. For efficiency reasons, the interpolation is not performed at the first and at the last iteration.
		Useful sketches of the original method and of the modification are shown in Figure \ref{fig:sketchmethods}, in particular, the crosses indicate the location in time of the quantities of interest. 
		For further details, the reader is referred to \cite{micalizzi2023new}.
		
		\begin{figure}
			\begin{subfigure}{0.4\textwidth}
				\begin{tikzpicture}[scale=0.85, every node/.style={transform shape}]
					\tikzset{dot/.style={fill=black,circle}}
					\tikzset{cross/.style={cross out, draw, 
							minimum size=6pt, 
							inner sep=0pt, outer sep=0pt}}
					\newcommand{\YQ}{0};
					\newcommand{\YW}{0.6667};
					\newcommand{\YE}{1.3333};
					\newcommand{\YR}{2};
					\newcommand{\YT}{2.6667};
					\newcommand{\YY}{3.3333};
					\newcommand{\YU}{4};
					\newcommand{\YM}{-1.5};
					\newcommand{\YD}{-0.8};
					\newcommand{\YP}{4.5};
					\newcommand{\YZ}{-0.2};
					\newcommand{\SH}{0.25};
					\newcommand{\OT}{0.6};

					\newcommand\XA{-1.2};
					
					\draw (\XA,\YQ) -- (\XA,\YU);
					\draw (\XA cm-2pt,\YQ) node[anchor=east]{$t_n+\dt$} -- (\XA cm+2pt,\YQ);
					\draw (\XA cm-2pt,\YU) node[anchor=east]{$t_n$} -- (\XA cm+2pt,\YU);

					\renewcommand\XA{-0.5};
					\draw (\XA-\SH,\YD) node[anchor=west] {$\underline{\uvec{c}}^{(0)}$};

					\draw (\XA,\YQ) -- (\XA,\YW);
					\draw[dashed] (\XA,\YW) -- (\XA, \YY);
					\draw (\XA,\YY) -- (\XA,\YU);
					
					\fill (\XA,\YQ) node[cross] {};		
					\fill (\XA,\YU)  node[cross] {};
					\fill (\XA,\YW)  node[cross] {};
					\fill (\XA,\YY)  node[cross] {};
					
					\path (\XA,\YZ) edge[->,>=stealth',bend right] node[anchor=north] {} (0.5,\YZ);
					
					\renewcommand\XA{0.5};
					
					\draw (\XA-\SH,\YD) node[anchor=west] {$\underline{\uvec{c}}^{(1)}$};
					
					\draw (\XA,\YQ) -- (\XA,\YW);
					\draw[dashed] (\XA,\YW) -- (\XA, \YY);
					\draw (\XA,\YY) -- (\XA,\YU);
					\fill (\XA,\YQ)  node[cross] {};		
					\fill (\XA,\YU)  node[cross] {};
					\fill (\XA,\YW)  node[cross] {};
					\fill (\XA,\YY)  node[cross] {};

					\path (\XA,\YZ) edge[->,>=stealth',bend right] node[anchor=north] {} (1.5,\YZ);
					
					\renewcommand\XA{1.5};
					\draw (\XA-\SH,\YD) node[anchor=west] {$\underline{\uvec{c}}^{(2)}$};
					
					\draw (\XA,\YQ) -- (\XA,\YW);
					\draw[dashed] (\XA,\YW) -- (\XA, \YY);
					\draw (\XA,\YY) -- (\XA,\YU);
					
					\fill (\XA,\YQ)  node[cross] {};		
					\fill (\XA,\YU)  node[cross] {};
					\fill (\XA,\YW)  node[cross] {};
					\fill (\XA,\YY)  node[cross] {};
					
					\draw[dotted] (1.75,2) --(2.5,2);
					
					\renewcommand\XA{2.75};
					\draw (\XA-\SH,\YD) node[anchor=west] {$\underline{\uvec{c}}^{(M)}$};
					
					\draw (\XA,\YQ) -- (\XA,\YW);
					\draw[dashed] (\XA,\YW) -- (\XA, \YY);
					\draw (\XA, \YY) -- (\XA,\YU);
					\fill (\XA,\YQ)  node[cross] {};		
					\fill (\XA,\YU)  node[cross] {};
					\fill (\XA,\YW)  node[cross] {};
					\fill (\XA,\YY)  node[cross] {};

					\path (\XA,\YZ) edge[->,>=stealth',bend right] node[anchor=north] {} (3.75,\YZ);
					
					\renewcommand\XA{3.75};
					\draw (\XA-\SH,\YD) node[anchor=west] {$\underline{\uvec{c}}^{(M+1)}$};

					\draw (\XA,\YQ) -- (\XA,\YW);
					\draw[dashed] (\XA,\YW) -- (\XA, \YY);
					\draw (\XA,\YY) -- (\XA,\YU);
					\fill (\XA,\YQ)  node[cross] {};		
					\fill (\XA,\YU)  node[cross] {};
					\fill (\XA,\YW)  node[cross] {};
					\fill (\XA,\YY)  node[cross] {};

				\end{tikzpicture}
			\end{subfigure}
			\begin{subfigure}{0.4\textwidth}
				\begin{tikzpicture}[scale=0.85, every node/.style={transform shape}]
					\tikzset{dot/.style={fill=black,circle}}
					\tikzset{cross/.style={cross out, draw, 
							minimum size=6pt, 
							inner sep=0pt, outer sep=0pt}}
					\newcommand{\YQ}{0};
					\newcommand{\YW}{0.6667};
					\newcommand{\YE}{1.3333};
					\newcommand{\YR}{2};
					\newcommand{\YT}{2.6667};
					\newcommand{\YY}{3.3333};
					\newcommand{\YU}{4};
					\newcommand{\YM}{-1.5};
					\newcommand{\YD}{-0.8};
					\newcommand{\YP}{4.5};
					\newcommand{\YZ}{-0.2};
					\newcommand{\SH}{0.25};
					\newcommand{\OT}{0.6};

					\newcommand\XA{1.3};
					\draw (\XA,\YQ) -- (\XA,\YU);
					\draw (\XA cm-2pt,\YQ) node[anchor=east]{$t_n+\dt$} -- (\XA cm+2pt,\YQ);
					\draw (\XA cm-2pt,\YU) node[anchor=east]{$t_n$} -- (\XA cm+2pt,\YU);

					\renewcommand\XA{2};
					\draw (\XA-\SH,\YD) node[anchor=west] {$\underline{\uvec{c}}^{(0)}$};
					
					\draw (\XA,\YQ) -- (\XA,\YU);
					\fill (\XA,\YQ)  node[cross] {};		
					\fill (\XA,\YU)  node[cross] {};

					\path (\XA,\YZ) edge[->,>=stealth',bend right] node[anchor=north] {} (3,\YZ);
					
					\renewcommand\XA{3};
					
					\draw (\XA-\SH,\YD) node[anchor=west] {$\underline{\uvec{c}}^{(1)}$};
					
					\draw (\XA,\YQ) -- (\XA,\YU);
					\fill (\XA,\YQ)  node[cross] {};		
					\fill (\XA,\YU)  node[cross] {};

					\path (\XA,\YZ) edge[->,>=stealth',bend right] node[anchor=north] {} (4,\YZ);
					
					\renewcommand\XA{4};
					\draw (\XA-\SH,\YD) node[anchor=west] {$\underline{\uvec{c}}^{*(1)}$};
					\draw (\XA,\YQ) -- (\XA,\YU);
					\draw (\XA,\YQ) node[cross] {};
					\draw (\XA,\YQ) node[anchor=south west] {};
					\draw (\XA,\YR) node[cross] {};
					\draw (\XA,\YR) node[anchor=south west] {};
					\draw (\XA,\YU) node[cross] {};
					\draw (\XA,\YU) node[anchor=south west] {};
					
					\path (\XA,\YZ) edge[->,>=stealth',bend right] node[anchor=north] {} (5,\YZ);

					\renewcommand\XA{5};
					\draw (\XA-\SH,\YD) node[anchor=west] {$\underline{\uvec{c}}^{(2)}$};
					
					\draw (\XA,\YQ) -- (\XA,\YU);
					\fill (\XA,\YQ) node[cross] {};
					\fill (\XA,\YR) node[cross] {};
					\fill (\XA,\YU) node[cross] {};
					
					\path (\XA,\YZ) edge[->,>=stealth',bend right] node[anchor=north] {} (6,\YZ);
					\renewcommand\XA{6};
					
					\draw (\XA-\SH,\YD) node[anchor=west] {$\underline{\uvec{c}}^{*(2)}$};
					
					\draw (\XA,\YQ) -- (\XA,\YU);
					\draw (\XA,\YQ) node[cross] {};
					\draw (\XA,\YQ) node[anchor=south west] {};
					\draw (\XA,\YE) node[cross] {};
					\draw (\XA,\YE) node[anchor=south west] {};
					\draw (\XA,\YT) node[cross] {};
					\draw (\XA,\YT) node[anchor=south west] {};
					\draw (\XA,\YU) node[cross] {};
					\draw (\XA,\YU) node[anchor=south west] {};

					\path (\XA,\YZ) edge[->,>=stealth',bend right] node[anchor=north] {} (7,\YZ);
					\renewcommand\XA{7};
					
					\draw (\XA-\SH,\YD) node[anchor=west] {$\underline{\uvec{c}}^{(3)}$};
					
					\draw (\XA,\YQ) -- (\XA,\YU);
					\draw (\XA,\YQ) node[cross] {};
					\draw (\XA,\YQ) node[anchor=south west] {};
					\draw (\XA,\YE) node[cross] {};
					\draw (\XA,\YE) node[anchor=south west] {};
					\draw (\XA,\YT) node[cross] {};
					\draw (\XA,\YT) node[anchor=south west] {};
					\draw (\XA,\YU) node[cross] {};
					\draw (\XA,\YU) node[anchor=south west] {};

					\draw[dotted] (7.25,2) --(8,2);
					
					\renewcommand\XA{8.25};
					\draw (\XA-\SH,\YD) node[anchor=west] {$\underline{\uvec{c}}^{(M)}$};
					
					\draw (\XA,\YQ) -- (\XA,\YW);
					\draw[dashed] (\XA,\YW) -- (\XA, \YY);
					\draw (\XA, \YY) -- (\XA,\YU);
					\fill (\XA,\YQ)  node[cross] {};		
					\fill (\XA,\YU)  node[cross] {};
					\fill (\XA,\YW)  node[cross] {};
					\fill (\XA,\YY)  node[cross] {};

					\path (\XA,\YZ) edge[->,>=stealth',bend right] node[anchor=north] {} (9.25,\YZ);
					
					\renewcommand\XA{9.25};
					\draw (\XA-\SH,\YD) node[anchor=west] {$\underline{\uvec{c}}^{(M+1)}$};
					
					\draw (\XA,\YQ) -- (\XA,\YW);
					\draw[dashed] (\XA,\YW) -- (\XA, \YY);
					\draw (\XA,\YY) -- (\XA,\YU);
					\fill (\XA,\YQ)  node[cross] {};		
					\fill (\XA,\YU)  node[cross] {};
					\fill (\XA,\YW)  node[cross] {};
					\fill (\XA,\YY)  node[cross] {};

				\end{tikzpicture}
			\end{subfigure}
			\caption{Sketch of the original DeC method on the left and of the modified one on the right}
			\label{fig:sketchmethods}
		\end{figure}
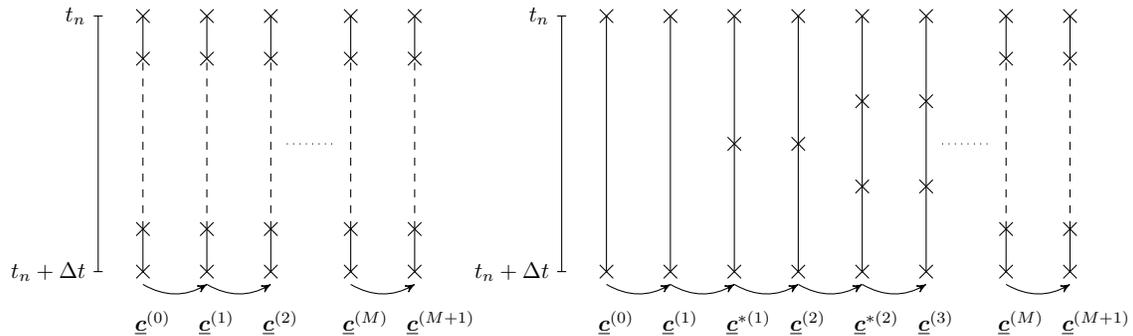

		\section{Numerical results}\label{chapNum}
		In this section, we will numerically investigate the elements previously introduced.
		The biggest part of the numerical investigation will concern the \RIcolor{one-dimensional} setting, for which some general information is reported in the following.
		Three different polynomial bases will be considered: the Bernstein polynomials, B$n$ $n=1,2,3,4$; the Lagrange polynomials associated to equispaced nodes, P$n$ $n=1,2,3$; the Lagrange polynomials associated to GL nodes, PGL$n$ $n=1,2,3,4$.
		More precisely, P1 and B1 coincide and P4 is not present because unstable. 
		Moreover, for each basis PGL$n$, we adopt the associated GL quadrature formula in order to achieve a natural high order mass lumping.

		In all the \RIcolor{one-dimensional} tests, the domain is $\Omega:=(0,25)$ and we assume no friction ($n_M=0$), unless differently specified. For the convergence analyses, we will consider the following $C^{\infty}$ bathymetry
		\begin{equation}
			B(x):=\begin{cases}
				0.2\exp{\left(1-\frac{1}{1-\left(\frac{x-10}{5}\right)^2}\right)}, & \text{if}~5<x<15,\\
				0, & \text{otherwise},
			\end{cases}
			\label{eq:smooth_bathymetry}
		\end{equation}
		while, in all the other cases, we will consider the $C^{0}$ bathymetry
		\begin{equation}
			B(x):=\begin{cases}
				0.2-0.05(x-10)^2, &\text{if}~8<x<12,\\
				0, & \text{otherwise}.
			\end{cases}
			\label{eq:c0_bathymetry}
		\end{equation}
		
		Concerning the parameters $\alpha_{f,r}:=\delta_r \bar{\rho}_f h_f^{2r}$ in \eqref{eq:jc} and \eqref{eq:jt} and $\alpha_{f}:=\delta \bar{\rho}_f h_f^{2}$ in \eqref{eq:je}, \eqref{eq:jr} and \eqref{eq:jg}, we set $\bar{\rho}_f$ to be the spectral radius of the Jacobian $\uvec{J}:=\frac{\partial \uvec{F}}{\partial \uvec{u}}$ of the flux at the interface $f$ and 
		\begin{equation}
			h_f:=\left(\frac{1}{2} \sum\limits_{x_i \in \left(\cup_{K \in K_f} K \right)} \abs{ \Bigg\llbracket \frac{\partial}{\partial x} \varphi_i \Bigg\rrbracket } \right)^{-1},
		\end{equation}
		where $K_f$ denotes the set of the two elements containing the interface $f$.
		For the parameters $\delta_r$ and $\delta$, we adopt the values reported in Table \ref{tab:coeff_CIP}. Let us notice that, in the context of jc \eqref{eq:jc} and jt \eqref{eq:jt}, we consider here the stabilization on the jump of the first and of the second derivatives.

		\begin{table}
			\centering
			\begin{tabular}{|c|c|c|c|c|}\hline
				&	P1=B1,PGL1 & B2,P2,PGL2 & B3,P3,PGL3 & B4,PGL4\\\hline
				$\delta_1=\delta$  &  0.05  &  0.3  &  0.15 &  0.5   \\ \hline		
				$\delta_2$         &  0.5   &  0.2  &  0.2  &  0.01  \\ \hline		
			\end{tabular}
			\caption{Coefficients adopted for $\delta_r$ and $\delta$ in the CIP stabilizations for the different basis functions in the \RIcolor{one-dimensional} tests \label{tab:coeff_CIP}}
		\end{table}
		
		In all the tests, we set CFL:=0.1, except for the ones involving B4 and PGL4, in which we adopt CFL:=0.05.
		In the context of unsteady tests, we will report results obtained with the basis functions PGL$n$ only, since, as underlined in \cite{Decremi,abgrall2019high,michel2021spectral,michel2022spectral,micalizzi2023new}, the DeC methods for CG involving the low order mass lumping require more iterations than what expected from theory for discretizations from order 4 on, for unsteady simulations, in order to attain the formal order of accuracy.
		
		\begin{remark}[On the choice of the parameters]
			One must be very careful in tuning the coefficients $\delta_r$ and $\delta$. A wrong choice can lead to unstable schemes, characterized by lower orders of convergence, with respect to the ones expected from theory, or even blow-ups. For further details, the reader is referred to the study of the linear stability presented in \cite{michel2021spectral}, where a collection of optimal settings in terms of CFL and stabilization parameters is reported. Nevertheless, the analysis behind the definition of such optimal settings does not directly apply to the context of this work for several reasons: the model problem was the linear advection equation, while, here we deal with the nonlinear SW equations; the DeC time integration method considered was the original formulation of the DeC presented in \cite{Decremi} and not the novel modification introduced in \cite{micalizzi2023new}; only the jump of the first derivative was taken into account, instead, here we consider the stabilization on the jump of the second derivative in the context of two stabilizations, i.e., jc \eqref{eq:jc} and jt \eqref{eq:jt}.
		\end{remark}

		%
		%
		%
		%
		%

		\begin{remark}[On the coupling of jr and jg with a WB space discretization]\label{rmk:jrjg}
			As already pointed out, the main focus of the stabilizations jr \eqref{eq:jr} and jg \eqref{eq:jg} is a consistent approximation of the same quantity: $\frac{\partial}{\partial x} \uvec{F}-\uvec{S}$.
			Generally speaking, there is no reason to use the stabilization based on the jump of the residual when one has the global flux at hand, e.g., in the context of the space discretization WB-$\GF$. 
			Similarly, when the global flux is not available, as in the context of the space discretization WB-$\HS$, it is not reasonable to compute it and use it just for the stabilization. 
			To sum up, as a general rule, we will couple the stabilization jr with the space discretization WB-$\HS$ and the stabilization jg with the space discretization WB-$\GF$.
		\end{remark}

		Due to the huge amount of possible combinations bewteen the considered elements (in the context of every test, apart from the reference non-WB approach, we have two WB space discretizations, four WB CIP stabilizations, three types of basis functions and, for each of them, different degrees), we will not systematically present all the results, but rather some representatives, to allow a meaningful comprehension.
		
		After a deep investigation of the \RIcolor{one-dimensional} setting, we will report some results obtained for multidimensional tests. 
		In particular, in this case, we will focus on high order convergence and well--balancing toward lake at rest for the reference non-WB framework and WB-HS coupled with jt.
		The main goal of this paper is to achieve well--balancing with respect to generic steady states not known in closed-form, which is successfully achieved in the \RIcolor{one-dimensional} tests;
		concerning the multidimensional setting, instead, the deep investigation of this aspect is left for future works.
		The results reported here for multidimensional tests have the only purpose to show the possibility to easily apply some of the proposed discretizations to an unstructured framework.

		The numerical results are organized as follows.
		We will start by testing, in Section \ref{sub:exwblatr}, the exact well-balancing with respect to the lake at rest steady state. We will collect the results, for the basis functions of highest degree, B4, PGL4 and P3, in tables.
		We will continue, in Section \ref{sub:AHO}, with some convergence analyses to check the arbitary high order accuracy of all the elements introduced on smooth steady states: a supercritical flow, a subcritical flow and a transcritical flow. In particular, in order to provide results for all the different types of polynomial bases, the results of the convergence analysis on the three steady states will be shown respectively for B4, PGL4 and P3. Nevertheless, several extra comparisons will be reported, concerning the settings with the best performances.
		In Section~\ref{sub:pert_lake_at_rest}, we will report the results of simulations involving the evolution of small perturbations of the lake at rest steady state.
		In Section~\ref{sub:pert_moving}, we will focus on the evolution of small perturbations of general steady states, whose analytical expression is not available in closed-form, with and without friction.
		Finally, in Section~\ref{sub:2d_tests}, we will report the results of some multidimensional tests on unstructured meshes.

		\subsection{Exact well-balancing for lake at rest}\label{sub:exwblatr}
		The simulations in this section are meant to test the WB feature with respect to the lake at rest steady state. We assume, in this context, the $C^0$ bathymetry \eqref{eq:c0_bathymetry}, as all the WB elements that we introduced do not require any smoothness assumption. Let us consider the lake at rest steady state given by 
		\begin{equation}
			\eta=H+B\equiv \overline{\eta}, \quad v\equiv 0,
		\end{equation}
		with $\overline{\eta}:=0.5$ and a final time $T_f:=10$ with strong boundary conditions. 
		The results got for $B4$, $PGL4$ and $P3$ and $100$ elements are respectively reported in Tables \ref{tab:LatRB4}, \ref{tab:LatRPGL4} and \ref{tab:LatRP3}.
		As expected from theory, the reference non-WB approach gives an error which is far from machine precision. The same holds for schemes obtained by coupling a WB space discretization with the orginal non-WB stabilization jc. In all the remaining cases, we have combinations of WB elements and, in fact, the related errors are around machine precision.

		\begin{center}
			\begin{table}
				\centering
				\begin{tabular}{cc|c|c|c}
					\cline{3-4}
					& & $L^1$ error $H$  & $L^1$ error $q$ \\ \cline{1-4}
					\multicolumn{2}{ |c| }{\multirow{1}{*}{Reference non-WB} } & 1.577E-002 &  1.169E-003 &      \\ \cline{1-4}
					\multicolumn{1}{ |c  }{\multirow{4}{*}{WB-$\HS$} } &
					\multicolumn{1}{ |c| }{jc (non-WB)} & 9.568E-004 &  2.060E-003 &      \\ \cline{2-4}
					\multicolumn{1}{ |c  }{}                        &
					\multicolumn{1}{ |c| }{jt} & 3.786E-015  & 1.084E-013 &      \\ \cline{2-4}
					\multicolumn{1}{ |c  }{}                        &
					\multicolumn{1}{ |c| }{je} & 9.028E-015 &  8.644E-014 &      \\ \cline{2-4}
					\multicolumn{1}{ |c  }{}                        &
					\multicolumn{1}{ |c| }{jr} & 2.515E-015  & 9.935E-014 &      \\ \cline{1-4}
					\multicolumn{1}{ |c  }{\multirow{4}{*}{WB-$\GF$} } &
					\multicolumn{1}{ |c| }{jc (non-WB)} & 9.608E-004 &  2.080E-003 &      \\ \cline{2-4}
					\multicolumn{1}{ |c  }{}                        &
					\multicolumn{1}{ |c| }{jt} & 5.215E-015  & 1.441E-014 &      \\ \cline{2-4}
					\multicolumn{1}{ |c  }{}                        &
					\multicolumn{1}{ |c| }{je} & 4.510E-015  & 9.251E-015 &      \\ \cline{2-4}
					\multicolumn{1}{ |c  }{}                        &
					\multicolumn{1}{ |c| }{jg} & 4.393E-015 &  1.303E-014 &      \\ \cline{1-4}
				\end{tabular} 
				\caption{Lake at rest, B4}
				\label{tab:LatRB4}
			\end{table}
			\begin{table}
				\centering
				\begin{tabular}{cc|c|c|c}
					\cline{3-4}
					& & $L^1$ error $H$  & $L^1$ error $q$ \\ \cline{1-4}
					\multicolumn{2}{ |c| }{\multirow{1}{*}{Reference non-WB} } & 1.028E-002 &  1.879E-003 &      \\ \cline{1-4}
					\multicolumn{1}{ |c  }{\multirow{4}{*}{WB-$\HS$} } &
					\multicolumn{1}{ |c| }{jc (non-WB)} & 3.007E-004 &  5.963E-004 &      \\ \cline{2-4}
					\multicolumn{1}{ |c  }{}                        &
					\multicolumn{1}{ |c| }{jt} & 9.403E-013 &  4.418E-012 &      \\ \cline{2-4}
					\multicolumn{1}{ |c  }{}                        &
					\multicolumn{1}{ |c| }{je} & 9.396E-013 &  4.415E-012 &      \\ \cline{2-4}
					\multicolumn{1}{ |c  }{}                        &
					\multicolumn{1}{ |c| }{jr} & 9.409E-013 &  4.415E-012 &      \\ \cline{1-4}
					\multicolumn{1}{ |c  }{\multirow{4}{*}{WB-$\GF$} } &
					\multicolumn{1}{ |c| }{jc (non-WB)} & 3.016E-004 &  6.168E-004 &      \\ \cline{2-4}
					\multicolumn{1}{ |c  }{}                        &
					\multicolumn{1}{ |c| }{jt} & 6.431E-013 &  2.659E-012 &      \\ \cline{2-4}
					\multicolumn{1}{ |c  }{}                        &
					\multicolumn{1}{ |c| }{je} & 6.423E-013 &  2.659E-012 &      \\ \cline{2-4}
					\multicolumn{1}{ |c  }{}                        &
					\multicolumn{1}{ |c| }{jg} & 6.431E-013 &  2.651E-012 &      \\ \cline{1-4}
				\end{tabular} 
				\caption{Lake at rest, PGL4}
				\label{tab:LatRPGL4}
			\end{table}
			\begin{table}
				\centering
				\begin{tabular}{cc|c|c|c}
					\cline{3-4}
					& & $L^1$ error $H$  & $L^1$ error $q$ \\ \cline{1-4}
					\multicolumn{2}{ |c| }{\multirow{1}{*}{Reference non-WB} } & 2.551E-002 &  2.931E-003 &      \\ \cline{1-4}
					\multicolumn{1}{ |c  }{\multirow{4}{*}{WB-$\HS$} } &
					\multicolumn{1}{ |c| }{jc (non-WB)} & 4.185E-004 &  7.355E-004 &      \\ \cline{2-4}
					\multicolumn{1}{ |c  }{}                        &
					\multicolumn{1}{ |c| }{jt} & 7.342E-015 &  1.170E-014 &      \\ \cline{2-4}
					\multicolumn{1}{ |c  }{}                        &
					\multicolumn{1}{ |c| }{je} & 8.673E-015 &  1.451E-014 &      \\ \cline{2-4}
					\multicolumn{1}{ |c  }{}                        &
					\multicolumn{1}{ |c| }{jr} & 8.004E-015 &  1.382E-014 &      \\ \cline{1-4}
					\multicolumn{1}{ |c  }{\multirow{4}{*}{WB-$\GF$} } &
					\multicolumn{1}{ |c| }{jc (non-WB)} & 4.206E-004 &  7.420E-004 &      \\ \cline{2-4}
					\multicolumn{1}{ |c  }{}                        &
					\multicolumn{1}{ |c| }{jt} & 6.213E-014 &  2.522E-013 &      \\ \cline{2-4}
					\multicolumn{1}{ |c  }{}                        &
					\multicolumn{1}{ |c| }{je} & 6.153E-014 &  2.502E-013 &      \\ \cline{2-4}
					\multicolumn{1}{ |c  }{}                        &
					\multicolumn{1}{ |c| }{jg} & 6.104E-014 &  2.502E-013 &      \\ \cline{1-4}
				\end{tabular} 
				\caption{Lake at rest, P3}
				\label{tab:LatRP3}
			\end{table}
		\end{center}
		
		\subsection{Arbitrary high order accuracy}\label{sub:AHO}
		In this section, we aim at numerically confirming the arbitrary high order accuracy of the considered space discretizations and of the novel jump stabilizations on smooth solutions. Therefore, for the tests presented here, we assume the $C^{\infty}$ bathymetry \eqref{eq:smooth_bathymetry}.
		Let us consider the three frictionless isoenergetic smooth steady states \cite{delestre2013swashes} satisfying \eqref{eq:steady_no_friction} with the following boundary conditions
		
		\begin{itemize}
			\item[•] \textbf{Supercritical}
			\begin{equation}
				q_L:= 24, \quad H_L:=2,
				\label{eq:superBC}
			\end{equation}
			\item[•] \textbf{Subcritical}
			\begin{equation}
				q_L:= 4.42, \quad H_R:=2,
				\label{eq:subBC}
			\end{equation}
			\item[•] \textbf{Transcritical}
			\begin{equation}
				q_L:= 1.53,
				\label{eq:transBC}
			\end{equation}
		\end{itemize}
		where, due to the fact that the momentum of the flow must be constant, the value at the boundary prescribes also the value in the interior of the domain: $q\equiv q_L$. Since $B$ is given, the total water height in each point, for the three steady states, can be (exactly) computed by solving \eqref{eq:steady_no_friction} with respect to $H$ and is depicted in Figure \ref{fig:smooth_steady_states}.
		We consider a final time $T_f:=100.$
		
		\begin{figure}
			\centering
			\begin{subfigure}[b]{0.30\textwidth}
				\centering
				\includegraphics[width=\textwidth]{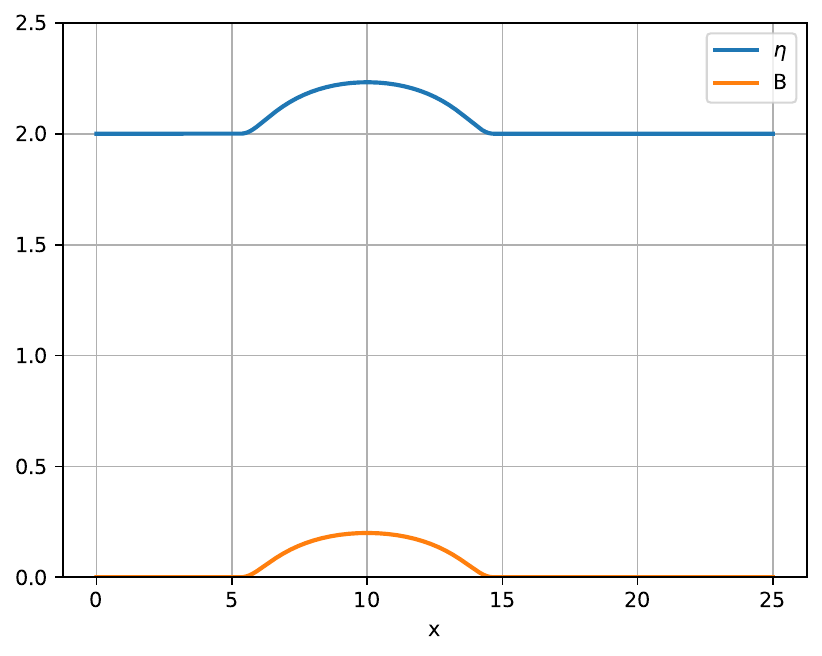}
				\caption{Supercritical}
				\label{convergence_comp_jumps_super}
			\end{subfigure}
			\begin{subfigure}[b]{0.30\textwidth}
				\centering
				\includegraphics[width=\textwidth]{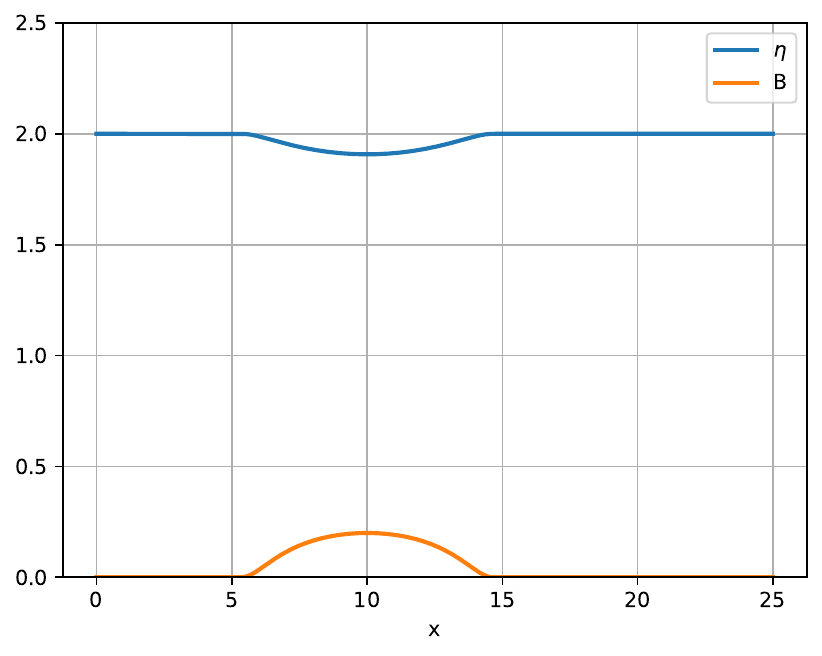}
				\caption{Subcritical}
				\label{convergence_comp_jumps_sub}
			\end{subfigure}
			\begin{subfigure}[b]{0.30\textwidth}
				\centering
				\includegraphics[width=\textwidth]{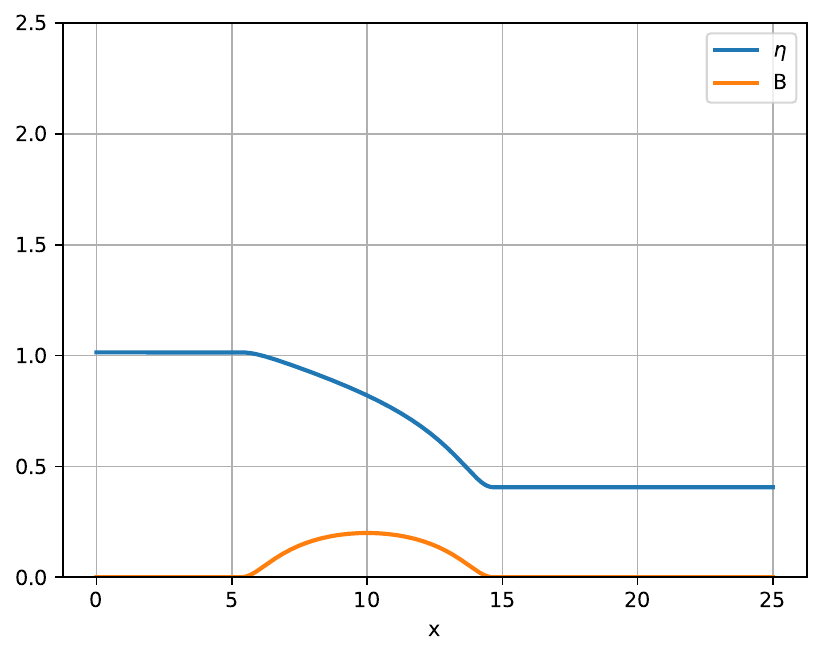}
				\caption{Transcritical}
				\label{convergence_comp_jumps_trans}
			\end{subfigure}
			\caption{Smooth steady states}
			\label{fig:smooth_steady_states}
		\end{figure}
		
		The results of the convergence analysis for the three steady states, respectively with B4, PGL4 and P3, are reported in Figure \ref{fig:convergence}. We can see how the formal order of accuracy is always recovered, with very evident superconvergences for the stabilizations involving the jump of the residual, jr  \eqref{eq:jr}, and of the derivative of the global flux, jg \eqref{eq:jg}. The errors obtained with such stabilizations are always much smaller than the ones obtained with the other schemes: roughly speaking, the difference is at least one order of magnitude, even more in the supercritical case. Further, the two stabilizations are characterized by steeper convergence slopes, with respect to the ones expected from theory, and a strong propensity to capture the constant momentum up to machine precision, see for example the supercritical tests.

		\begin{figure}
			\begin{subfigure}{0.31\textwidth}
				\includegraphics[width=\textwidth]{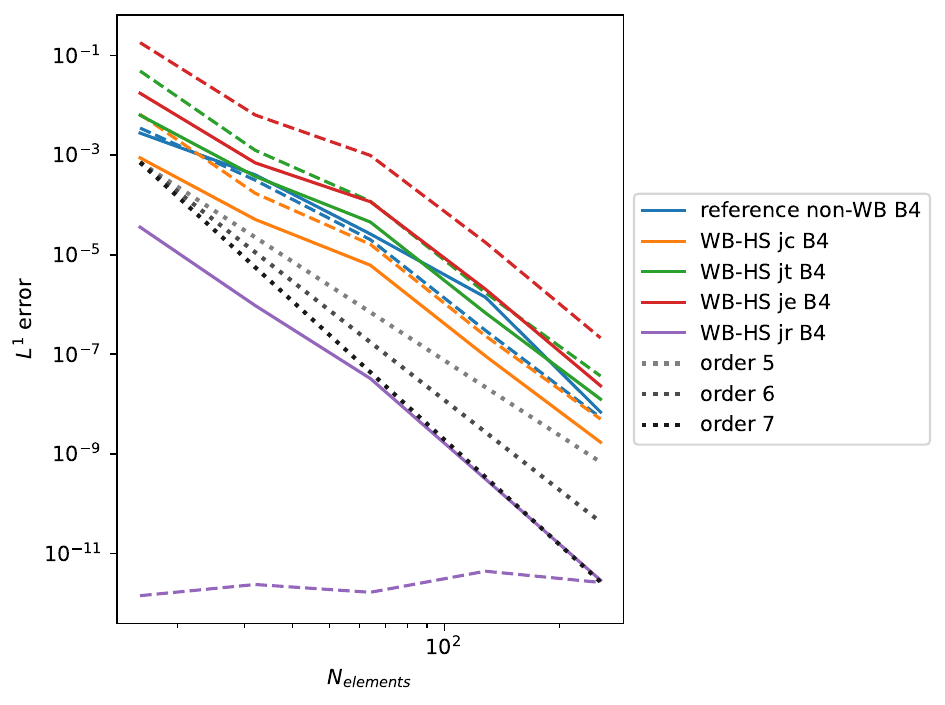}\caption{Supercritical, WB-$\HS$, B4}
			\end{subfigure}
			\begin{subfigure}{0.31\textwidth}
				\includegraphics[width=\textwidth]{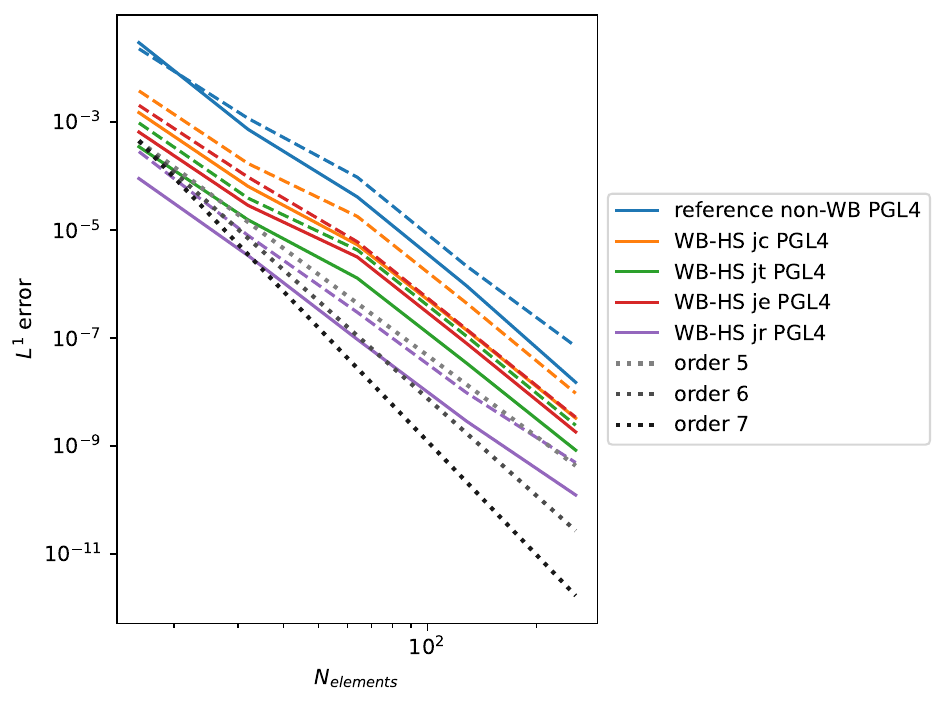}\caption{Subcritical, WB-$\HS$, PGL4}
			\end{subfigure}
			\begin{subfigure}{0.31\textwidth}
				\includegraphics[width=\textwidth]{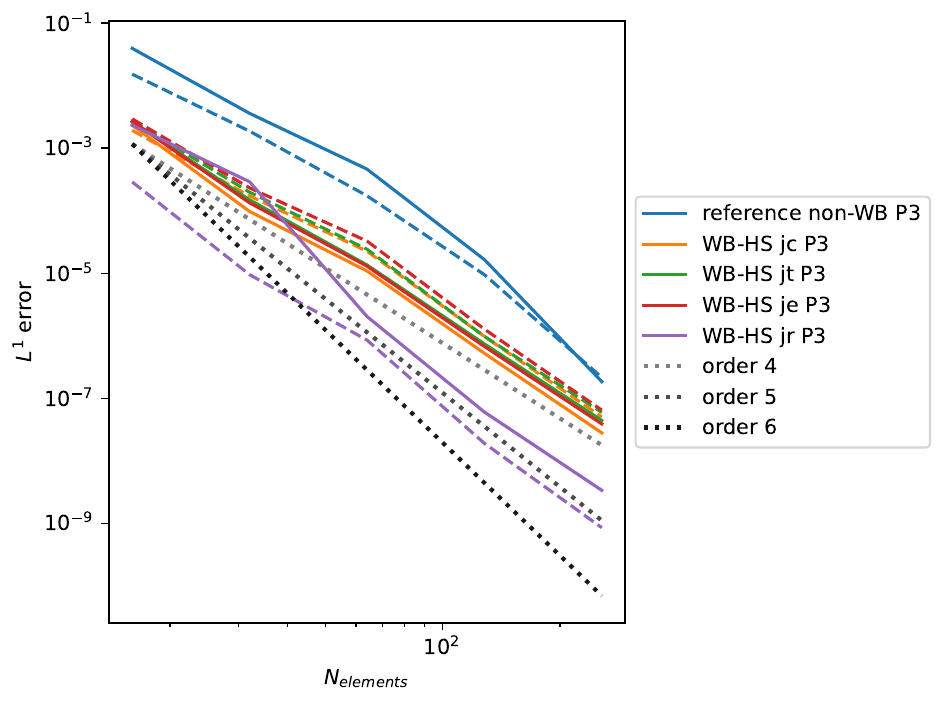}\caption{Transcritical, WB-$\HS$, P3}
			\end{subfigure}\\
			\begin{subfigure}{0.31\textwidth}
				\includegraphics[width=\textwidth]{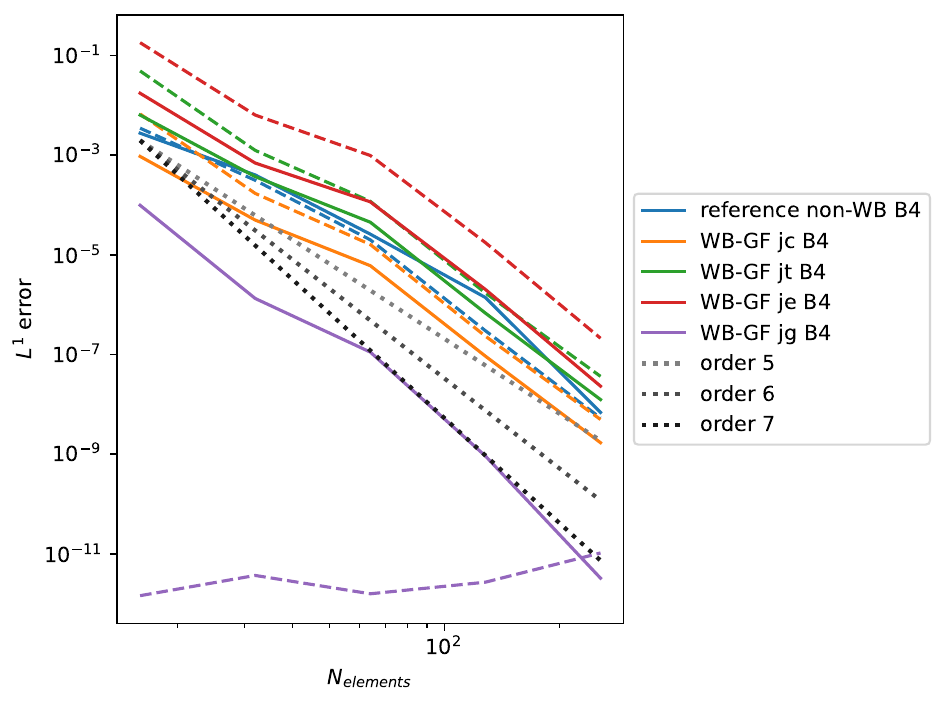}\caption{Supercritical, WB-$\GF$, B4}
			\end{subfigure}
			\begin{subfigure}{0.31\textwidth}
				\includegraphics[width=\textwidth]{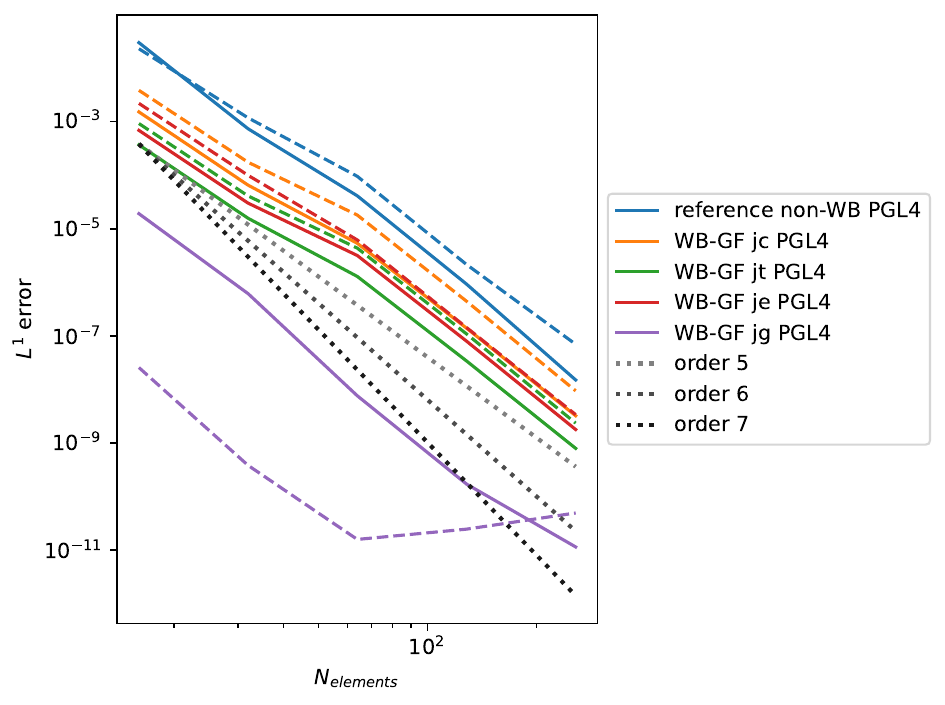}\caption{Subcritical, WB-$\GF$, PGL4}
			\end{subfigure}
			\begin{subfigure}{0.31\textwidth}
				\includegraphics[width=\textwidth]{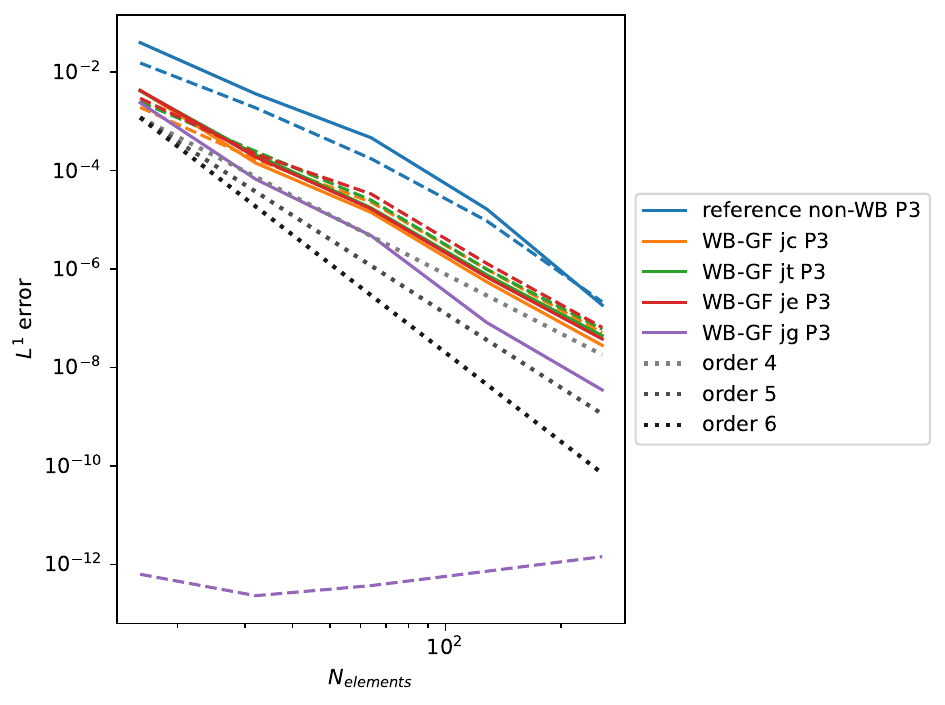}\caption{Transcritical, WB-$\GF$, P3}
			\end{subfigure}
			\caption{Convergence analysis: supercritical with B4, subcritical with PGL4 and transcritical with P3. $L^1$ error on $H$ in continuous line, on $q$ in dashed line}\label{fig:convergence}
		\end{figure}
		
		We will focus now on the two best performing jump stabilizations, neglecting the other ones for the sake of compactness. In Figure \ref{fig:convergencebases}, we display the results on the same tests for basis functions of different degrees. We can see that the superconvergences are not strictly related to the basis functions of highest degree. Apart from B1 (equivalent to P1) and PGL1, whose results confirm the expected second order accuracy, in almost all the other cases we experience convergence slopes steeper than the ones expected from theory:
		\begin{itemize}
			\item[•] in the supercritical case, B2 converges with order 4 rather than 3, B4 with 7 rather than 5; further, only for jg, B3 converges with order 6 rather than 4;
			\item[•] in the subcritical case, PGL3 converges with order 5 rather than 4; further, only for jg, PGL2 converges with order 4 rather than 3 and PGL4 with order 6 rather than 5;
			\item[•] in the transcritical case, P3 converges with order 5 rather than 4 and, only for jg, P2 converges with order 4 rather than 3.
		\end{itemize}
		Moreover, also in this case, the ability of capturing exactly the constant momentum is very remarkable, see the supercritical case or the transcritical case with jg.

		\begin{figure}
			\begin{subfigure}{0.31\textwidth}
				\includegraphics[width=\textwidth]{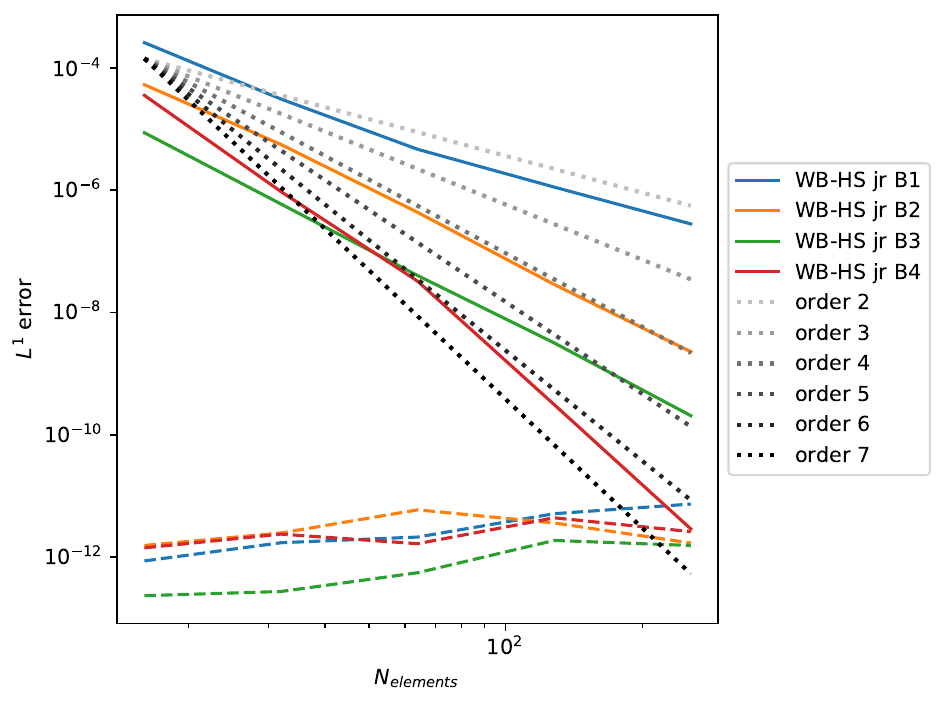}\caption{Supercritical, WB-$\HS$, jr}
			\end{subfigure}
			\begin{subfigure}{0.31\textwidth}
				\includegraphics[width=\textwidth]{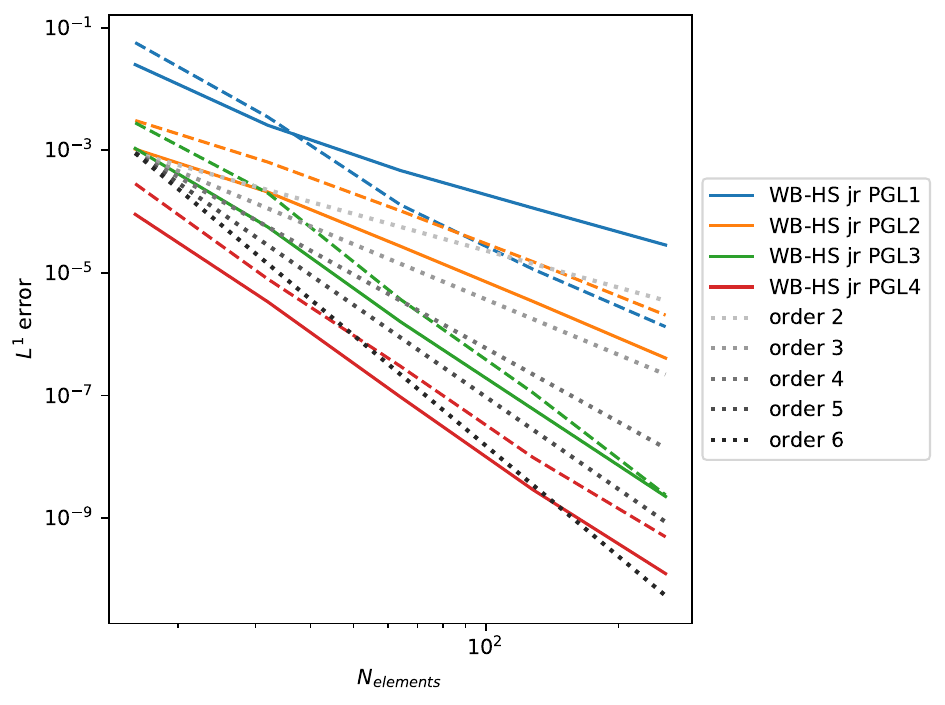}\caption{Subcritical, WB-$\HS$, jr}
			\end{subfigure}
			\begin{subfigure}{0.31\textwidth}
				\includegraphics[width=\textwidth]{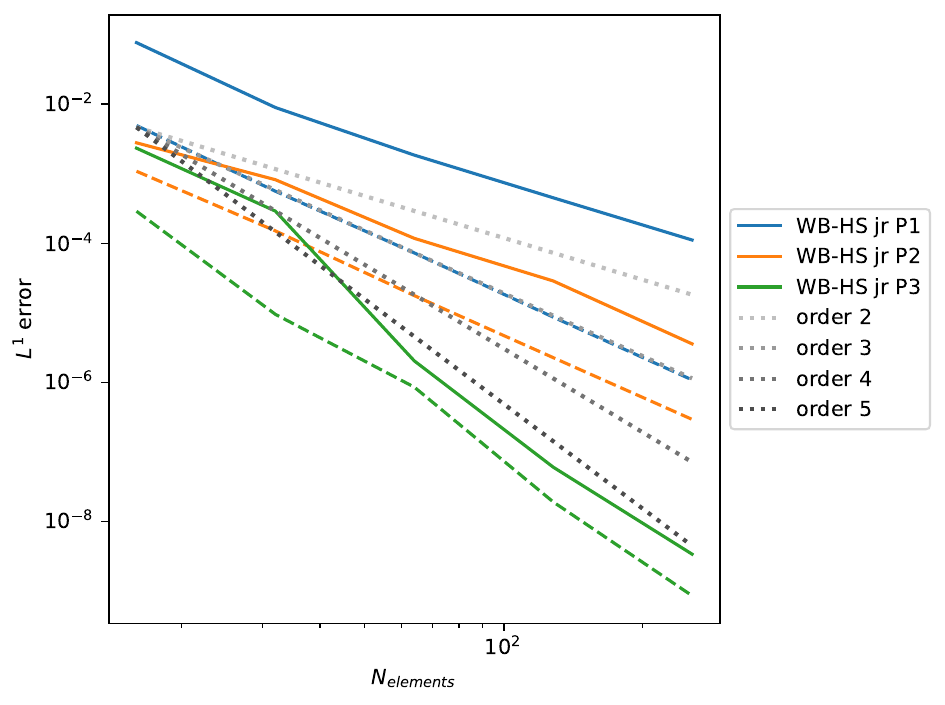}\caption{Transcritical, WB-$\HS$, jr}
			\end{subfigure}\\
			\begin{subfigure}{0.31\textwidth}
				\includegraphics[width=\textwidth]{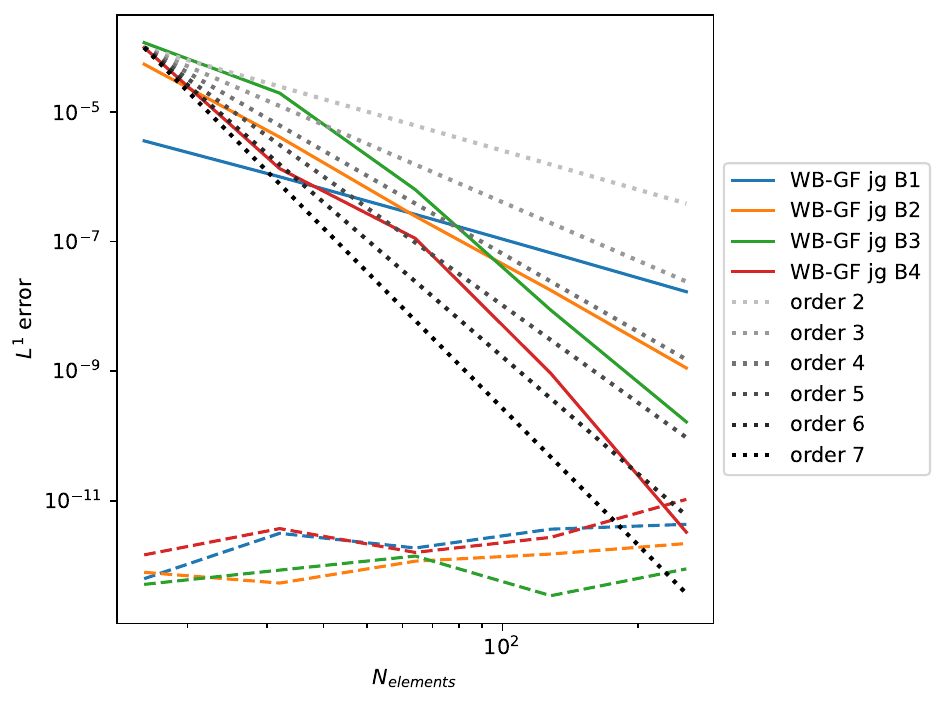}\caption{Supercritical, WB-$\GF$, jg}
			\end{subfigure}
			\begin{subfigure}{0.31\textwidth}
				\includegraphics[width=\textwidth]{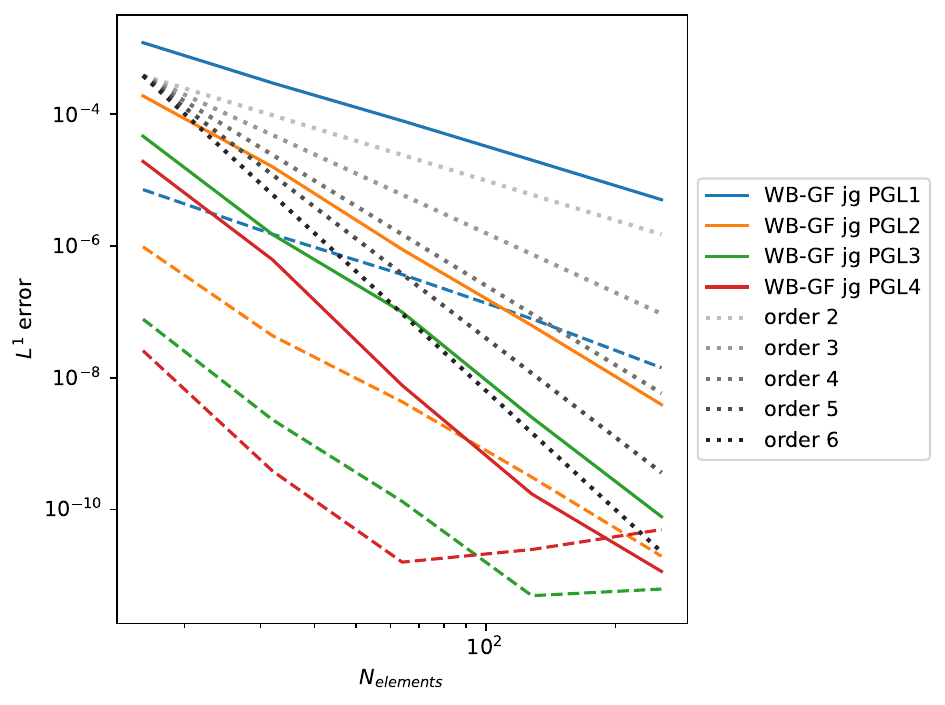}\caption{Subcritical, WB-$\GF$, jg}
			\end{subfigure}
			\begin{subfigure}{0.31\textwidth}
				\includegraphics[width=\textwidth]{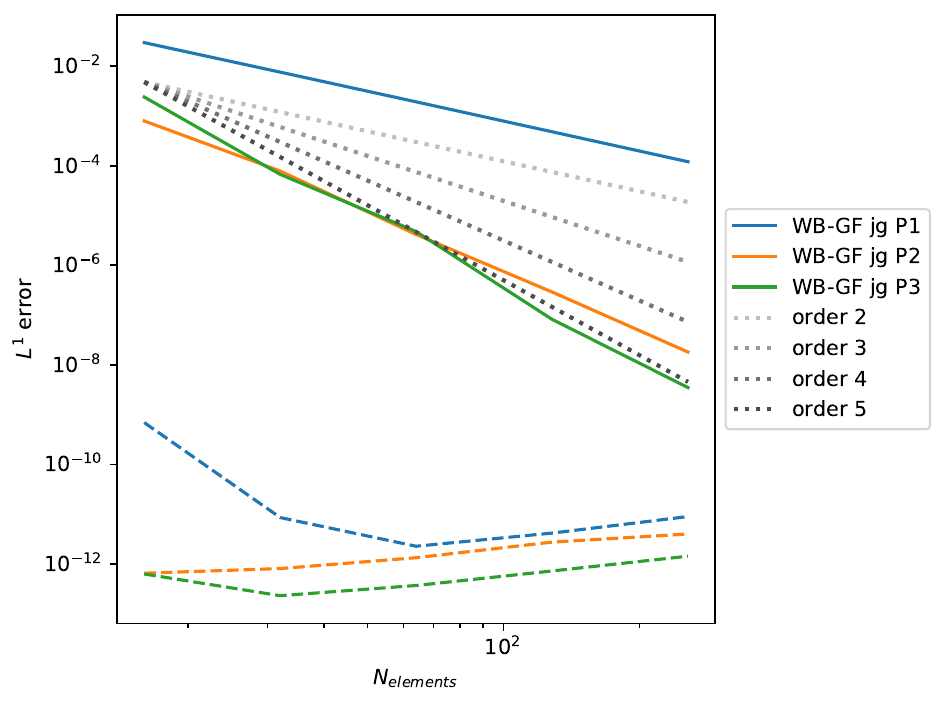}\caption{Transcritical, WB-$\GF$, jg}
			\end{subfigure}
			\caption{Convergence analysis: best performing settings for basis functions with different degrees. Supercritical with B$n$, subcritical with PGL$n$ and transcritical with P$n$. $L^1$ error on $H$ in continuous line, on $q$ in dashed line}\label{fig:convergencebases}
		\end{figure}
		
		A comparison between B3, PGL3 and P3 on the three steady states is reported in Figure \ref{fig:convergencebasescompare}. We can see that in the context of the setting WB-$\HS$-jr, P3 performs better than B3 and PGL3; this does not hold for  WB-$\GF$-jg, for which PGL3 is the best performing basis among the ones considered and the results of B3 and P3 are very similar.

		\begin{figure}
			\begin{subfigure}{0.31\textwidth}
				\includegraphics[width=\textwidth]{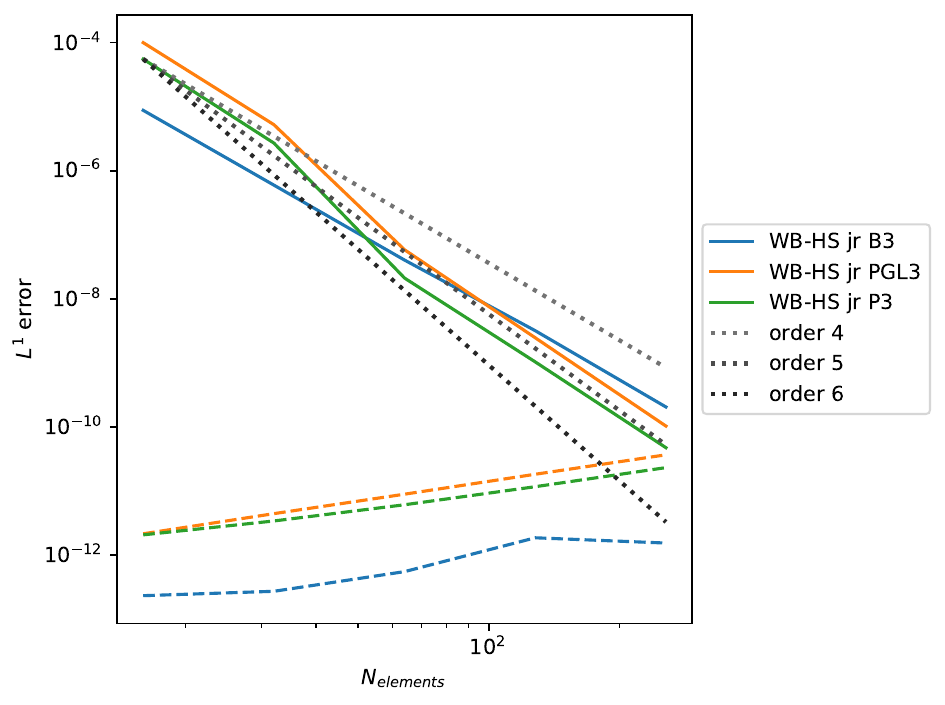}\caption{Supercritical, WB-$\HS$, jr}
			\end{subfigure}
			\begin{subfigure}{0.31\textwidth}
				\includegraphics[width=\textwidth]{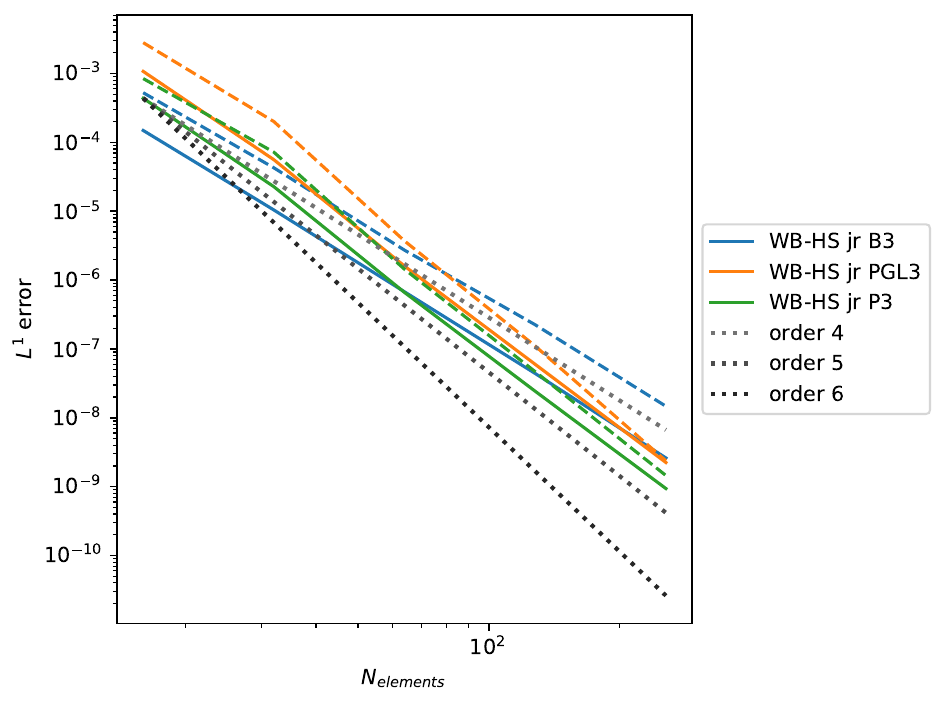}\caption{Subcritical, WB-$\HS$, jr}
			\end{subfigure}	
			\begin{subfigure}{0.31\textwidth}
				\includegraphics[width=\textwidth]{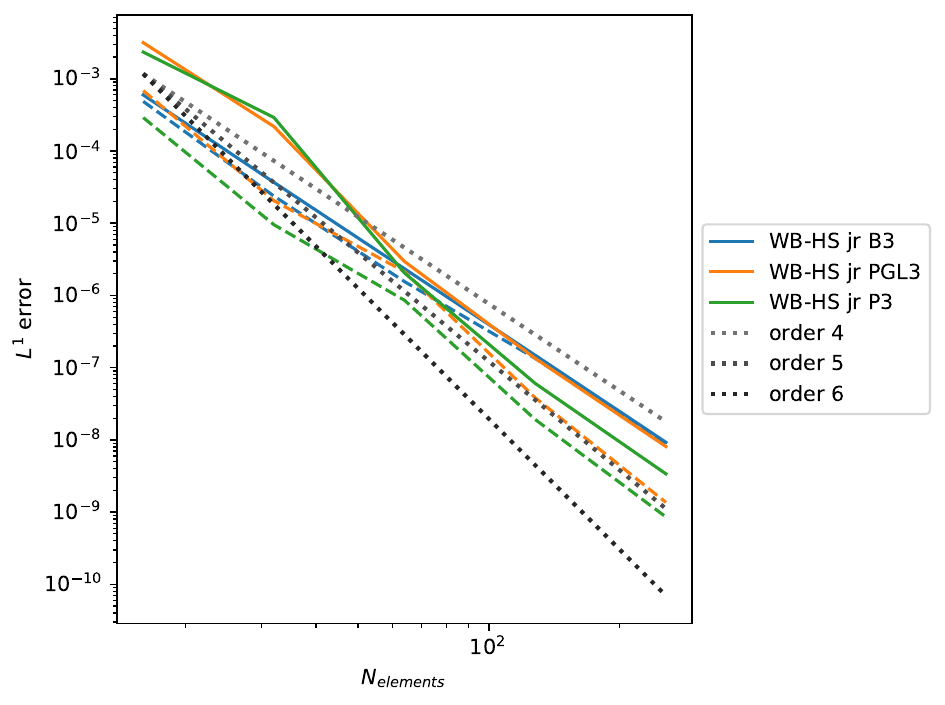}\caption{Transcritical, WB-$\HS$, jr}
			\end{subfigure}\\
			\begin{subfigure}{0.31\textwidth}
				\includegraphics[width=\textwidth]{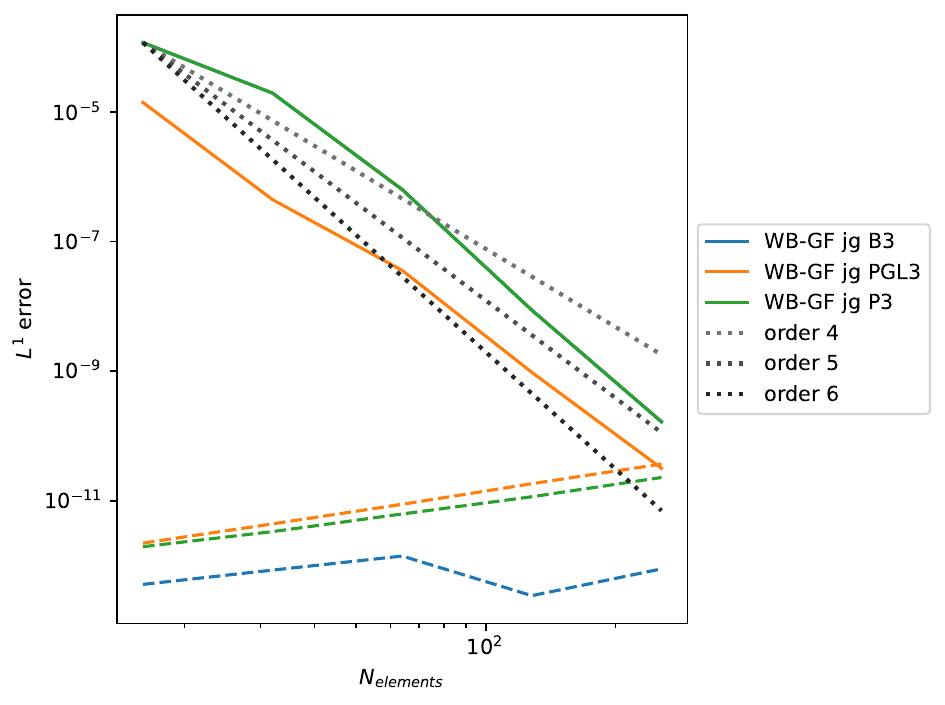}\caption{Supercritical, WB-$\GF$, jg}
			\end{subfigure}
			\begin{subfigure}{0.31\textwidth}
				\includegraphics[width=\textwidth]{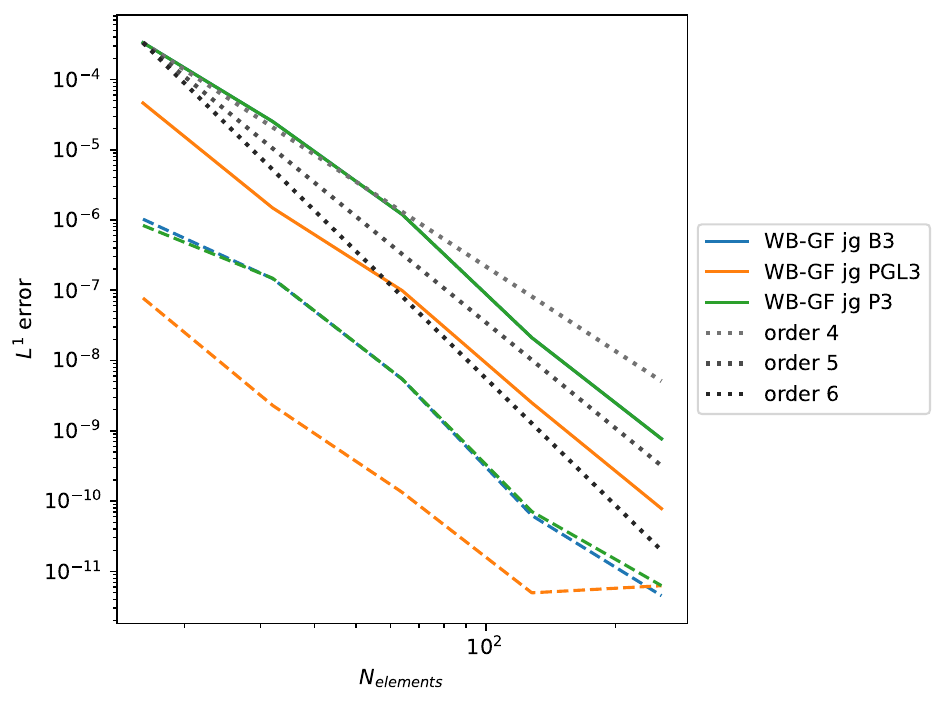}\caption{Subcritical, WB-$\GF$, jg}
			\end{subfigure}
			\begin{subfigure}{0.31\textwidth}
				\includegraphics[width=\textwidth]{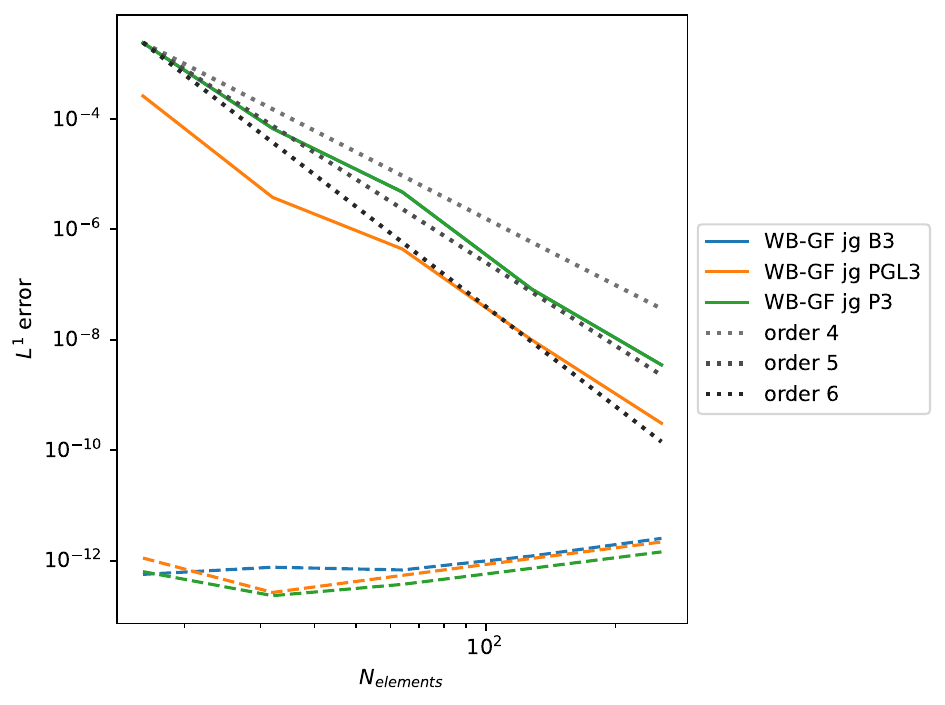}\caption{Transcritical, WB-$\GF$, jg}
			\end{subfigure}
			\caption{Convergence analysis: comparison between B3, PGL3 and P3 with the best performing settings. $L^1$ error on $H$ in continuous line, on $q$ in dashed line}\label{fig:convergencebasescompare}
		\end{figure}
		
		Finally, we present a comparison between the best performing settings for the basis functions of highest degree B4, PGL4 and P3 in Figure \ref{fig:bestperforming}. We already underlined, in Remark \ref{rmk:jrjg}, that WB-$\HS$-jr and WB-$\GF$-jg represent the most natural couplings. Nevertheless, for the sake of curiosity, we will consider also the other two possible combinations. For the supercritical flow, WB-$\HS$-jg is the best performing setting followed by WB-$\HS$-jr; for the subcritical and the transcritical flows, we can see how WB-$\GF$-jg is by far the best combination in terms of capturing of the constant momentum. In the context of the subcritical flow, such setting is also characterized by smaller errors on the water height, while, in the context of the transcritical flow the performance of all the settings under this point of view is not significatively different.
		
		\begin{figure}
			\centering
			\begin{subfigure}{0.31\textwidth}
				\includegraphics[width=\textwidth]{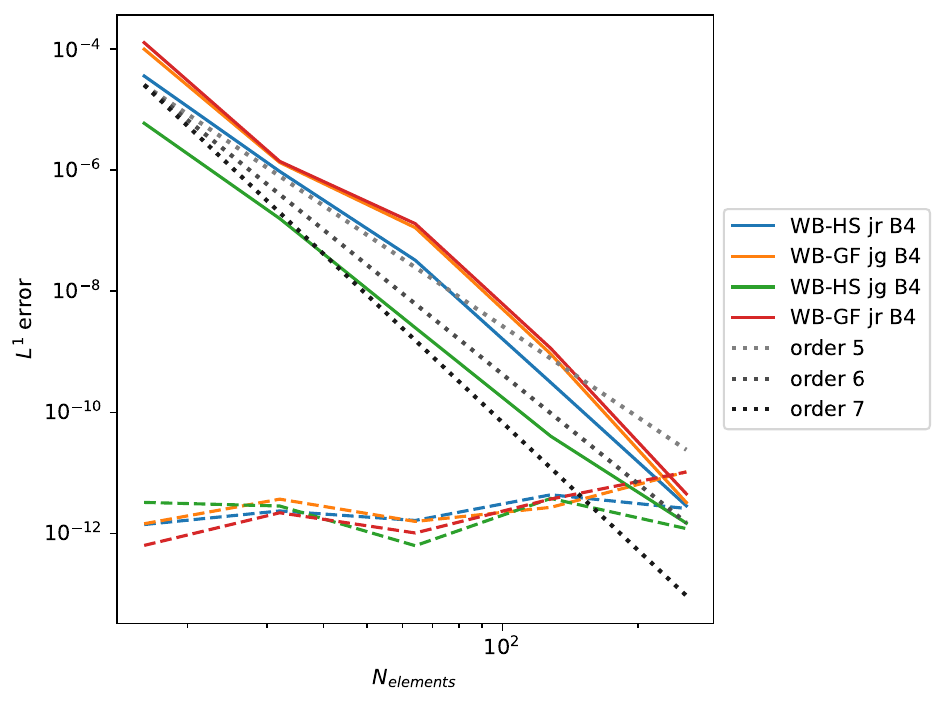}\caption{Supercritical}
			\end{subfigure}
			\begin{subfigure}{0.31\textwidth}
				\includegraphics[width=\textwidth]{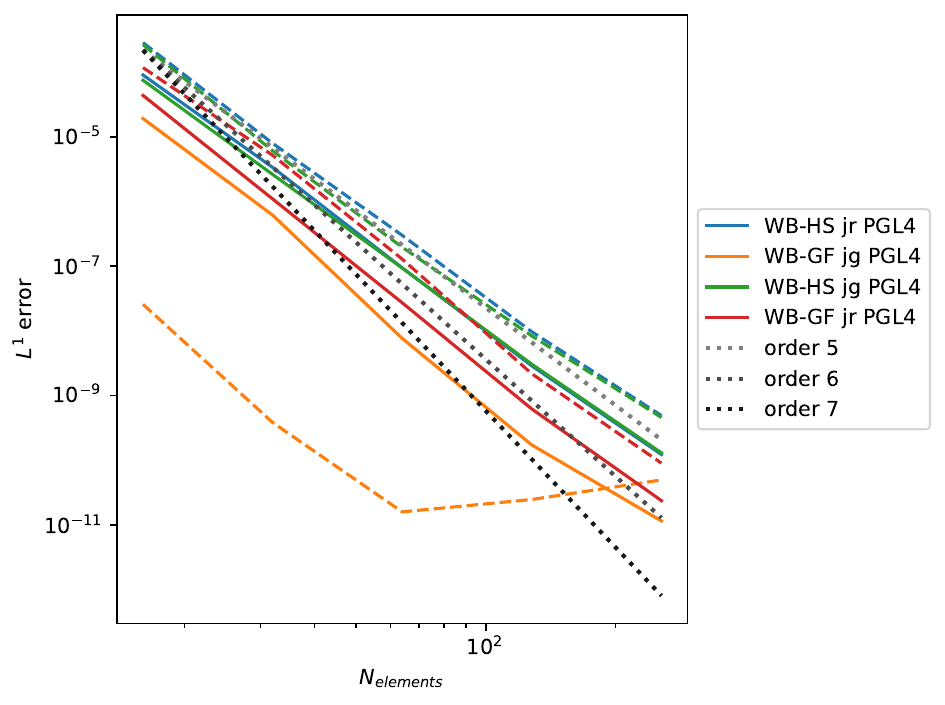}\caption{Subcritical}
			\end{subfigure}
			\begin{subfigure}{0.31\textwidth}
				\includegraphics[width=\textwidth]{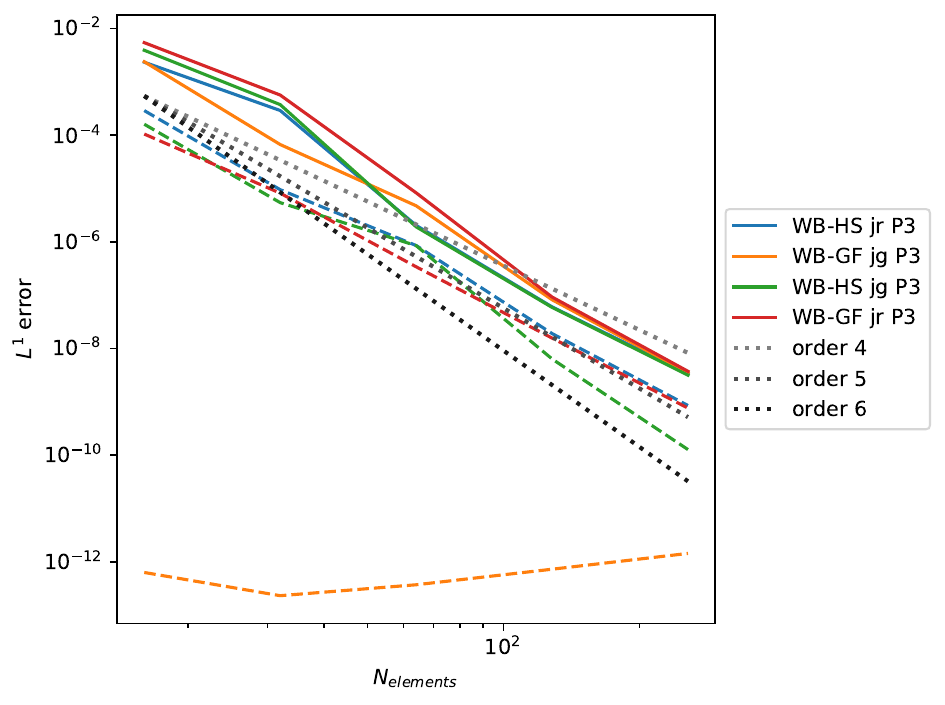}\caption{Transcritical}
			\end{subfigure}
			\caption{Convergence analysis: comparison between the best performing settings. $L^1$ error on $H$ in continuous line, on $q$ in dashed line}\label{fig:bestperforming}
		\end{figure}
		
		As already remarked, reporting the results for all the possible combinations of basis functions, space discretizations and jump stabilizations for any test would have been rather chaotic. For this reason, only the most significative ones have been selected. Nevertheless, we have tried, through several comparisons, to provide a wide variety of results for all the settings, focusing more on the best performing ones. Summarizing, the results seen in this subsection confirm the advantages in adopting the stabilizations jr \eqref{eq:jr} and jg \eqref{eq:jg} in the context of smooth steady states.

		\subsection{Evolution of small perturbations of lake at rest}\label{sub:pert_lake_at_rest}
		In this section, we test the ability of the WB space discretizations and stabilizations to capture the evolution of small perturbations of the lake at rest steady state.
		
		We consider again the reference test in Section \ref{sub:exwblatr} but we introduce the following small perturbation
		\begin{equation}
			\eta(x):=\begin{cases}
				\eta_{s}+A\exp{\left(1-\frac{1}{1-\left(\frac{x-6}{0.5}\right)^2}\right)}, &\text{if}~5.5<x<6.5,\\
				\eta_{s}, & \text{otherwise},
			\end{cases}
			\label{eq:perturbation_lake_at_rest}
		\end{equation}
		with $A:=5\cdot 10^{-5}$, where $\eta_{s}\equiv \overline{\eta}$ represents the total water height of the steady state.

		The initial condition and the evolution of the pertubation at the time $T_f:=1.5$, obtained with PGL4 and adopting non-WB settings, are depicted in Figure \ref{fig:latrnonwb}. 
		For each setting, two results are plotted, one obtained with a coarse mesh with $30$ elements, the other one obtained with a refined mesh with $128$ elements. 
		One can see that there is indeed an advantage in adopting a WB space discretization: in the context of the reference non-WB framework, the discretization error completely overwhelms the perturbation, while, in the other two cases, one gets spurious oscillations which are much smaller. 
		Nevertheless, the presence of such oscillations testifies the non-WB character of the schemes obtained by coupling a WB space discretization, WB-$\HS$ or WB-$\GF$, with a non-WB stabilization, jc.

		\begin{figure}
			\centering
			\begin{subfigure}{0.40\textwidth}
				\includegraphics[width=\textwidth]{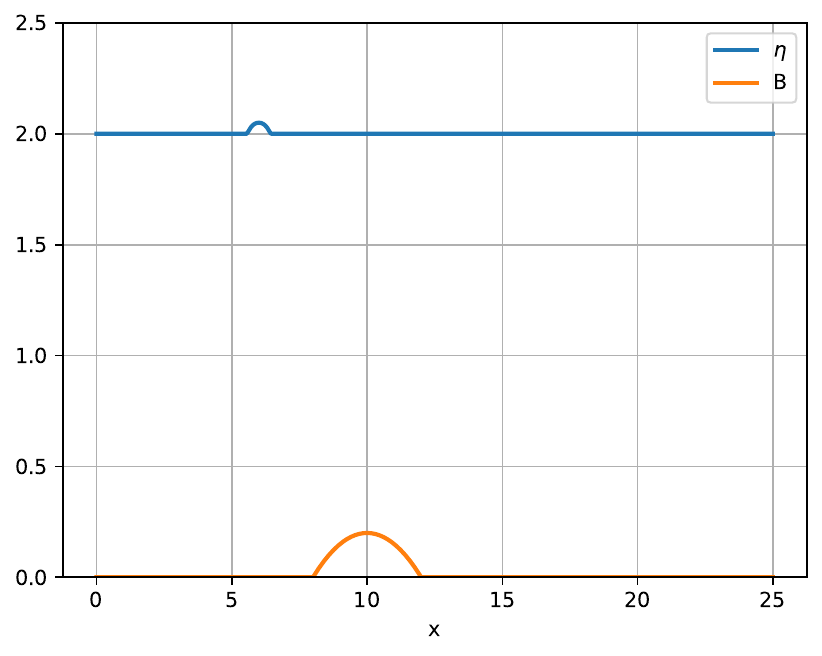}\caption{Initial total height and bathymetry. The perturbation is amplified by a factor $10^3$ in order to make it visible}
			\end{subfigure}\\
			\begin{subfigure}{0.31\textwidth}
				\includegraphics[width=\textwidth]{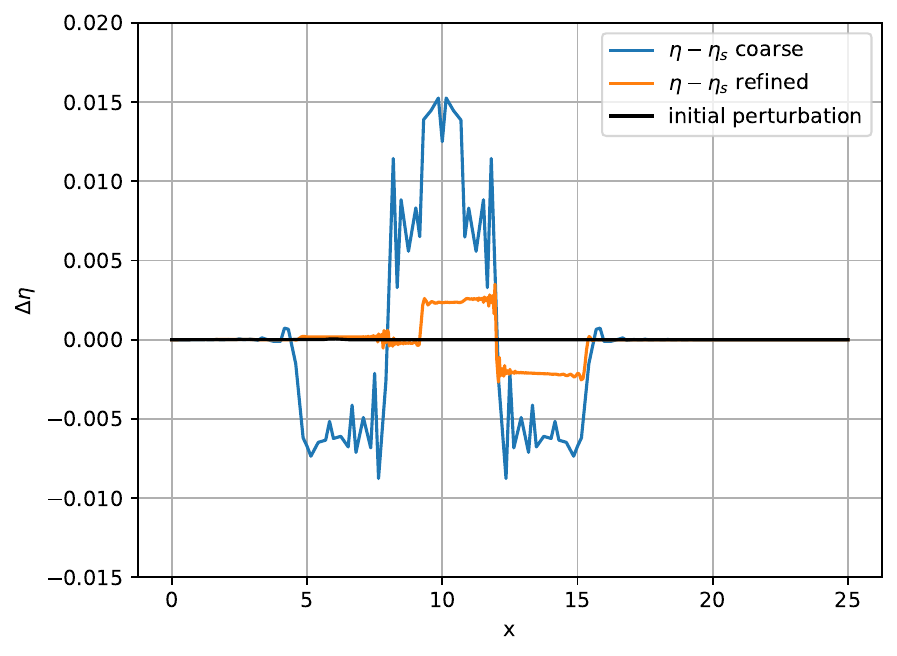}\caption{Reference non-WB setting}
			\end{subfigure}
			\begin{subfigure}{0.31\textwidth}
				\includegraphics[width=\textwidth]{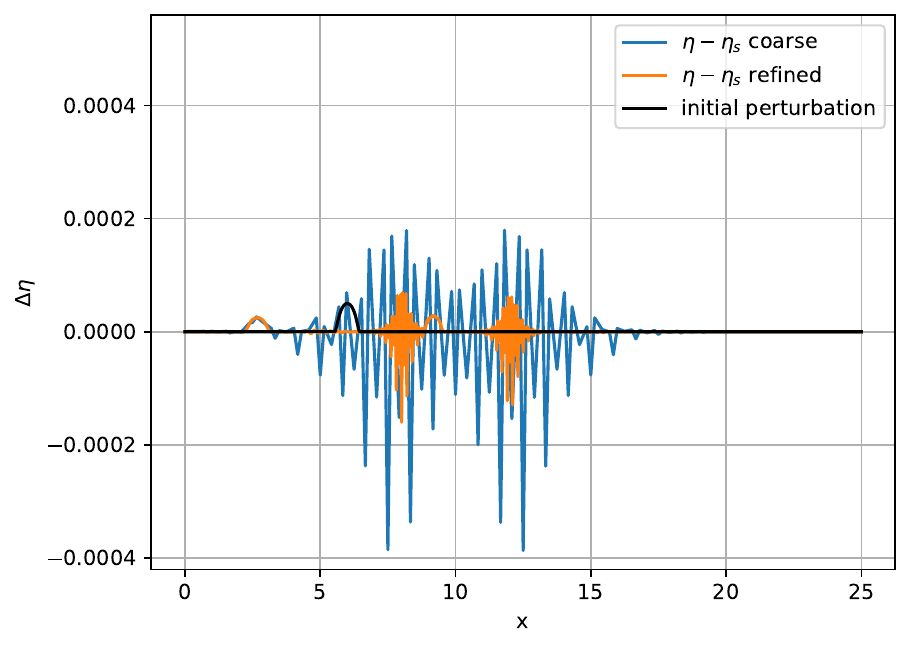}\caption{WB-$\HS$ with jc (non-WB)}
			\end{subfigure}
			\begin{subfigure}{0.31\textwidth}
				\includegraphics[width=\textwidth]{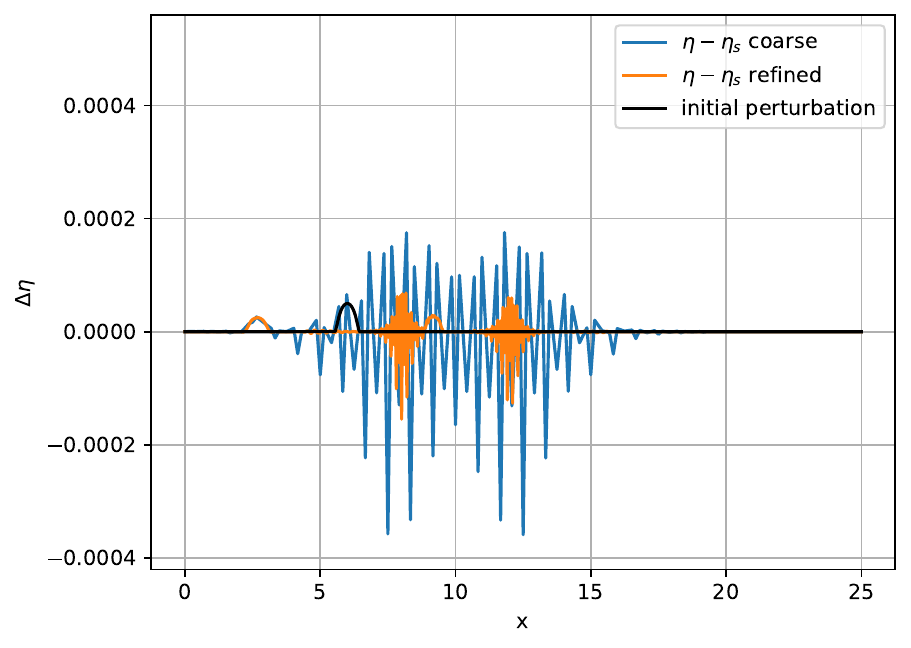}\caption{WB-$\GF$ with jc (non-WB)}
			\end{subfigure}
			\caption{Perturbation of lake at rest: initial condition and results obtained with non-WB settings. Results referred to PGL4 with $30$ elements for the coarse mesh and $128$ elements for the refined mesh. A different scale has been used for the reference non-WB setting}\label{fig:latrnonwb}
		\end{figure}
		
		In Figure \ref{fig:latrwb}, instead, one can see how the adoption of a fully WB scheme, i.e., for which also the stabilization is WB, is able to completely remove the spurious oscillations. 
		
		\begin{figure}
			\centering
			\begin{subfigure}{0.31\textwidth}
				\includegraphics[width=\textwidth]{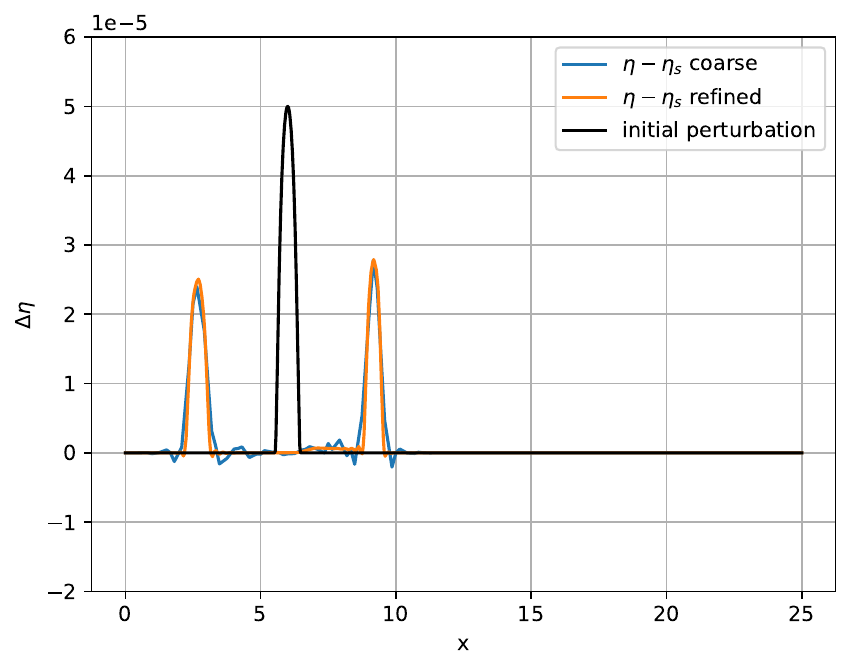}\caption{WB-$\HS$ with jt}
			\end{subfigure}
			\begin{subfigure}{0.31\textwidth}
				\includegraphics[width=\textwidth]{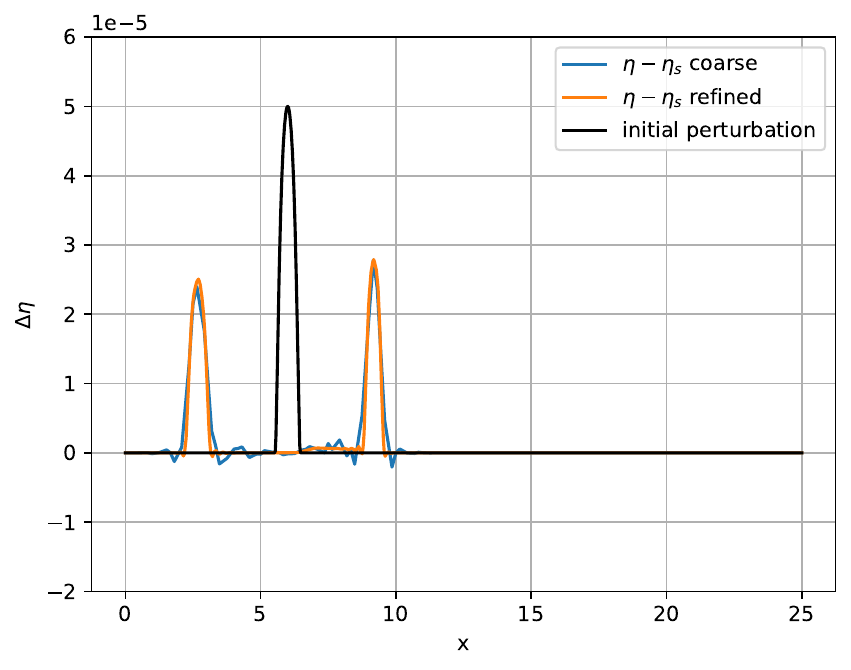}\caption{WB-$\HS$ with je}
			\end{subfigure}
			\begin{subfigure}{0.31\textwidth}
				\includegraphics[width=\textwidth]{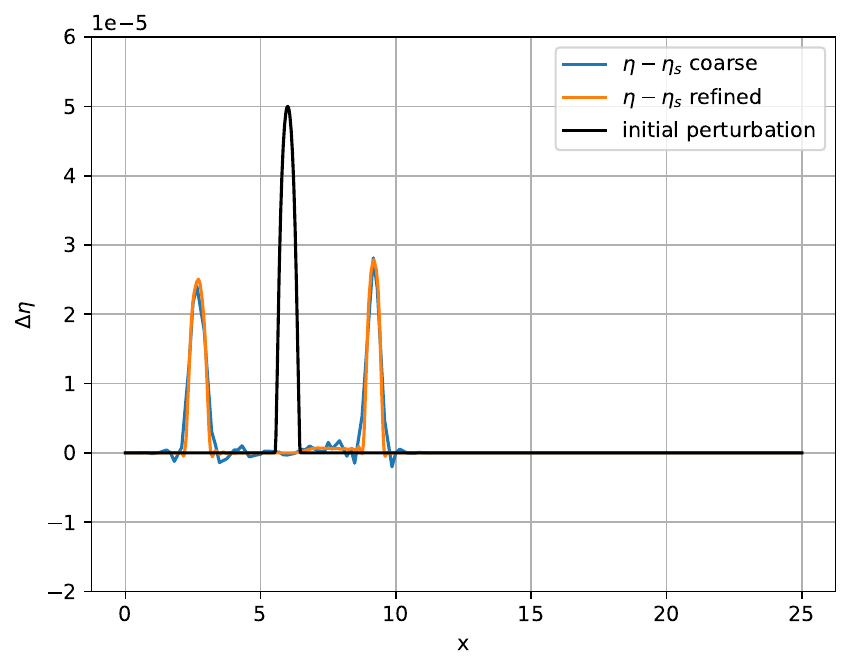}\caption{WB-$\HS$ with jr}
			\end{subfigure}\\
			\begin{subfigure}{0.31\textwidth}
				\includegraphics[width=\textwidth]{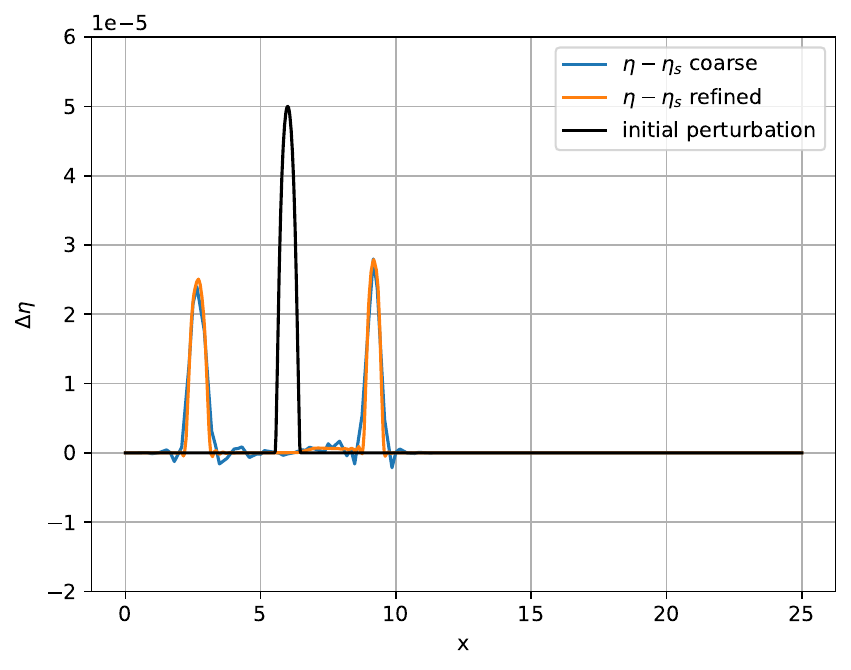}\caption{WB-$\GF$ with jt}
			\end{subfigure}
			\begin{subfigure}{0.31\textwidth}
				\includegraphics[width=\textwidth]{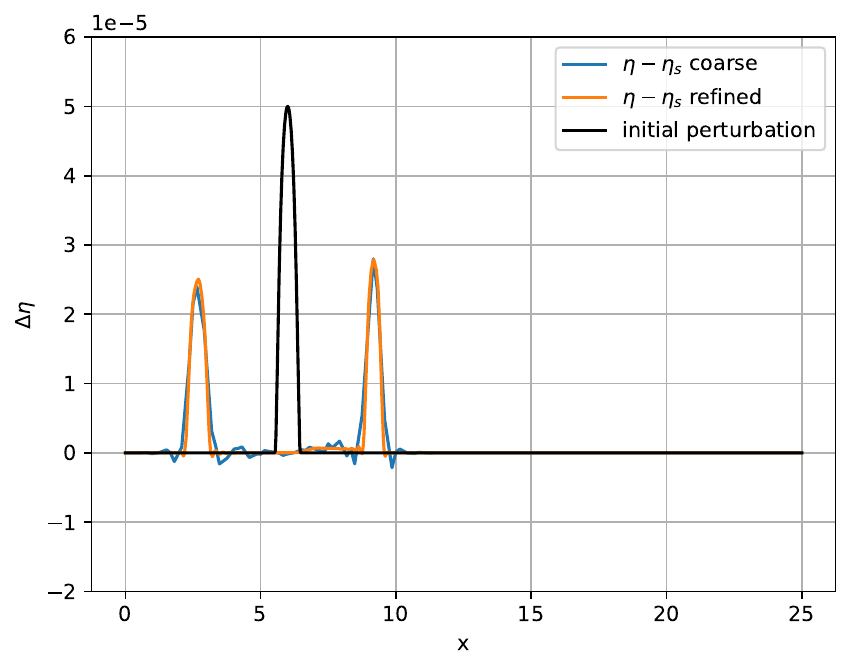}\caption{WB-$\GF$ with je}
			\end{subfigure}
			\begin{subfigure}{0.31\textwidth}
				\includegraphics[width=\textwidth]{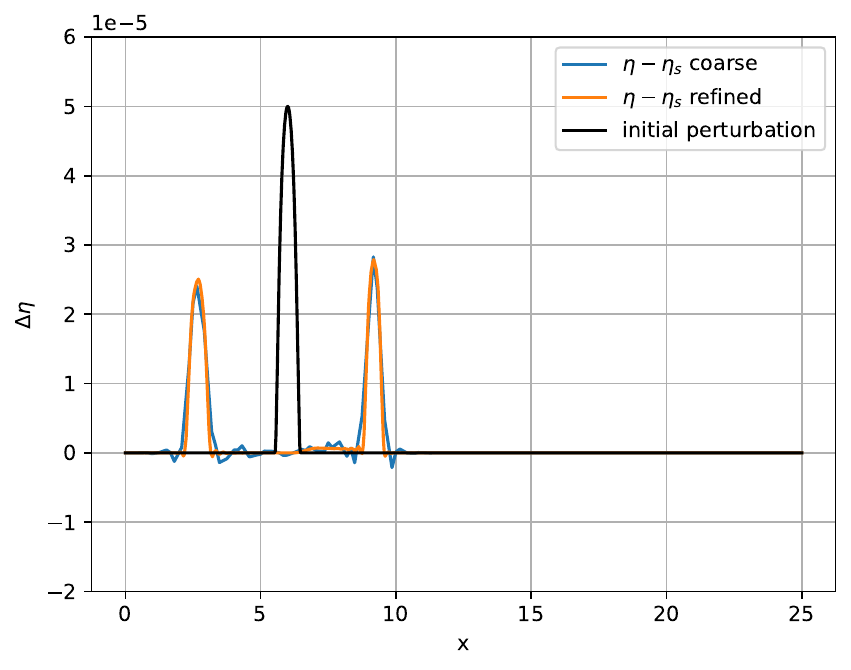}\caption{WB-$\GF$ with jg}
			\end{subfigure}
			\caption{
				Perturbation of lake at rest: WB settings. Results referred to PGL4 with $30$ elements for the coarse mesh and $128$ elements for the refined mesh
			}\label{fig:latrwb}
		\end{figure}
		
		Analogous results have been got for PGL$n$ with $n=1,2,3$. For what concerns B$n$ and P$n$, instead, the results are similar only up to order 3. For $n\geq 3$, the pathology of the time-stepping method for low order mass lumpings, already mentioned, prevents from recovering the formal order of accuracy without increasing the number of iterations with respect to what theoretically predicted.
		For the sake of compactness such results have been omitted.

		\subsection{Evolution of small perturbations of moving equilibria}\label{sub:pert_moving}
		In this section, we test the WB properties of the introduced elements with respect to general steady states not known in closed-form.
		We remark that the two WB space discretizations here adopted, as well as all the novel CIP stabilizations, have been designed ad hoc to exactly preserve the lake at rest steady state. However, the last two stabilizations, jr \eqref{eq:jr} and jg \eqref{eq:jg}, also address the problem of the preservation of general steady states, being based on discretizations of $\frac{\partial}{\partial x}\uvec{F}-\uvec{S}$. In fact, as already shown in the convergence analyses, such stabilizations are characterized by strong superconvergences towards steady states.
		This section is divided in two parts: in the first one we assume no friction, instead, in the second one we assume a Manning friction coefficient $n_M:=0.03.$

		\subsubsection{Tests without friction}
		We consider here the three non-smooth steady states characterized by the boundary conditions \eqref{eq:superBC}, \eqref{eq:subBC} and \eqref{eq:transBC} but with the $C^0$ bathymetry \eqref{eq:c0_bathymetry}. We will analyze them separately in the following. 
		Again, the water height can be retrieved via the (exact) solution of \eqref{eq:steady_no_friction}. 
		
		\begin{itemize}
			\item[•] \textbf{Supercritical flow}\\
			We consider in this case the same small perturbation \eqref{eq:perturbation_lake_at_rest} adopted for the lake at rest steady state but a different final time $T_f:=1$. Indeed, in this case $\eta_s$ is not constant. The initial condition and the results got with the non-WB (with respect to lake at rest) settings are reported in Figure \ref{fig:supnonwb}. Coherently with the previous case, the results are referred to PGL4 with $30$ and $128$ elements respectively for the coarse and the refined meshes. Again, there is a certain advantage in adopting the WB space discretizations, which seem to be more capable to handle a non-smooth bathymetry even for steady states different from the lake at rest. 
			However, still they are characterized by spurious oscillations and this is not surprising as the elements have not been designed to achieve well-balancing with respect to a general steady state.

			\begin{figure}
				\centering
				\begin{subfigure}{0.40\textwidth}
					\includegraphics[width=\textwidth]{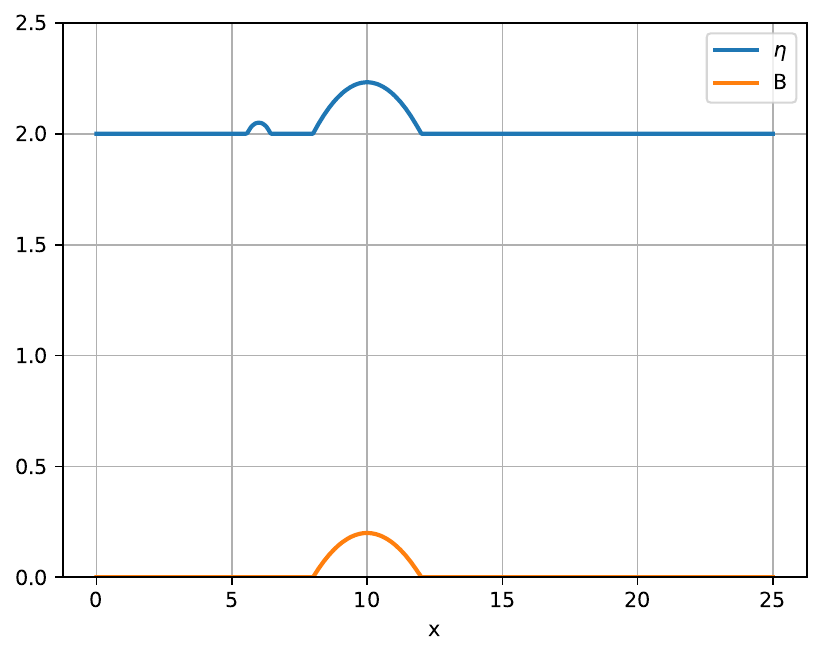}\caption{Initial total height and bathymetry. The perturbation is amplified by a factor $10^3$ in order to make it visible}
				\end{subfigure}\\
				\begin{subfigure}{0.31\textwidth}
					\includegraphics[width=\textwidth]{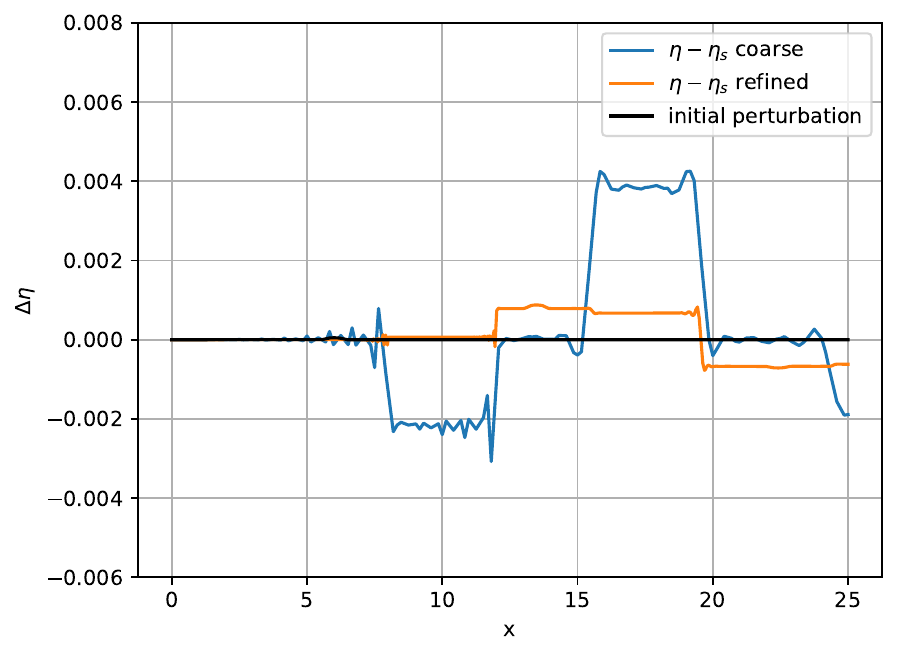}\caption{Reference non-WB setting}
				\end{subfigure}
				\begin{subfigure}{0.31\textwidth}
					\includegraphics[width=\textwidth]{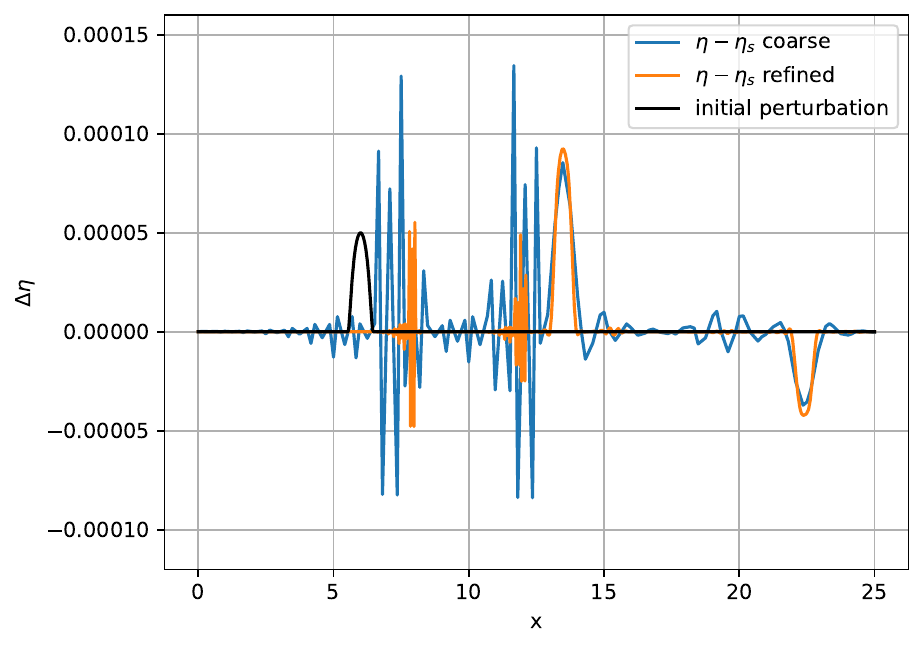}\caption{WB-$\HS$ with jc}
				\end{subfigure}
				\begin{subfigure}{0.31\textwidth}
					\includegraphics[width=\textwidth]{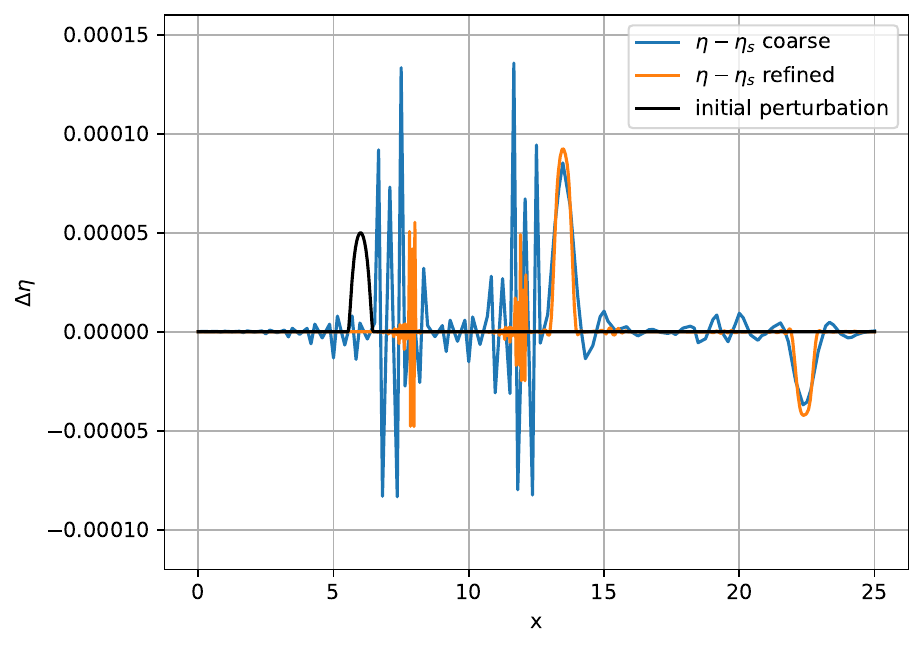}\caption{WB-$\GF$ with jc}
				\end{subfigure}
				\caption{Perturbation of non-smooth frictionless supercritical steady state: initial condition and results obtained with non-WB settings. Results referred to PGL4 with $30$ elements for the coarse mesh and $128$ elements for the refined mesh. A different scale has been used for the reference non-WB setting}\label{fig:supnonwb}
			\end{figure}

			In Figure \ref{fig:supwb}, we see the effect of the different stabilizations. It is immediately noticeable the ability of jr and jg to capture in a polite way the evolution of the perturbation without spurious oscillations. This feature can be somehow expected since, as already remarked, the two stabilizations are designed to stabilize the quantity $\frac{\partial}{\partial x}\uvec{F}-\uvec{S}$. Nevertheless, let us notice that no particular discretization has been adopted to make sure that they are exactly zero with respect to the investigated steady state. One can observe that very little spurious oscillations are present also for jr and jg in the results obtained with the coarse mesh, but this is normal, due to the lack of a limiting strategy. Such oscillations completely disappear in the mesh refinement.
			
			\begin{figure}
				\centering
				\begin{subfigure}{0.31\textwidth}
					\includegraphics[width=\textwidth]{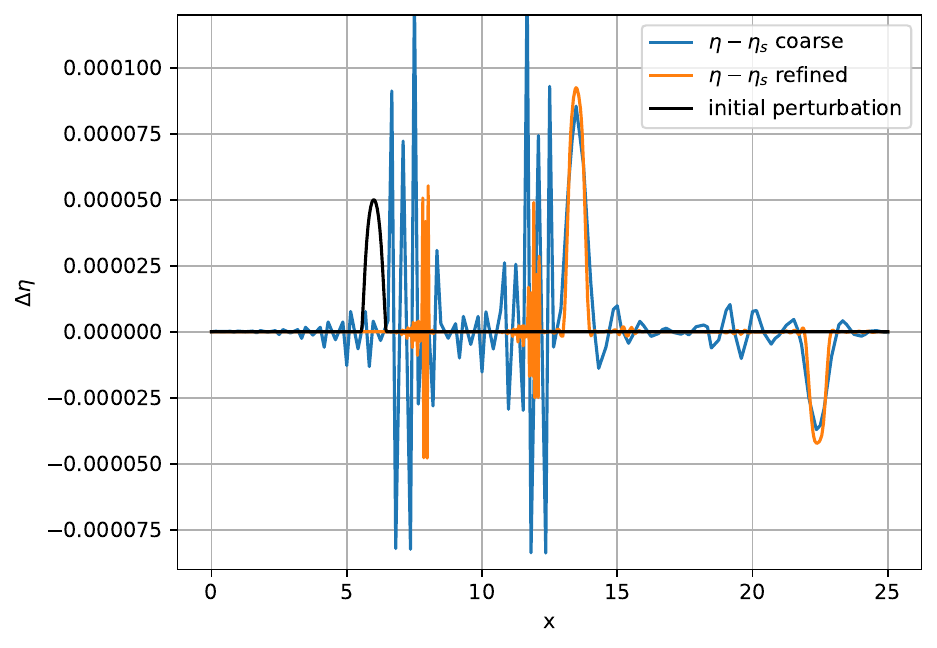}\caption{WB-$\HS$ with jc}
				\end{subfigure}
				\begin{subfigure}{0.31\textwidth}
					\includegraphics[width=\textwidth]{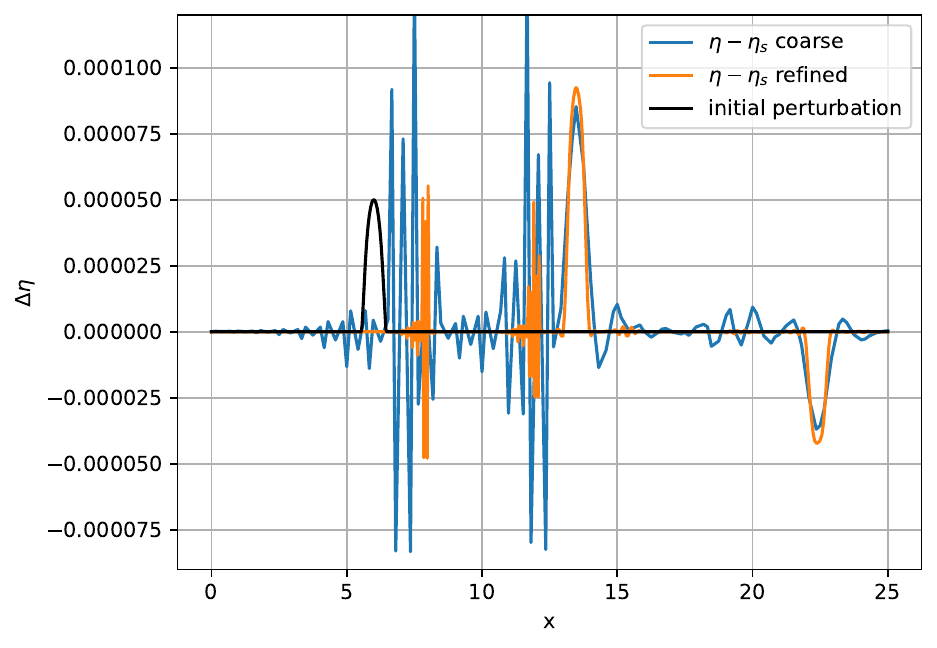}\caption{WB-$\GF$ with jc}
				\end{subfigure}\\
				\begin{subfigure}{0.31\textwidth}
					\includegraphics[width=\textwidth]{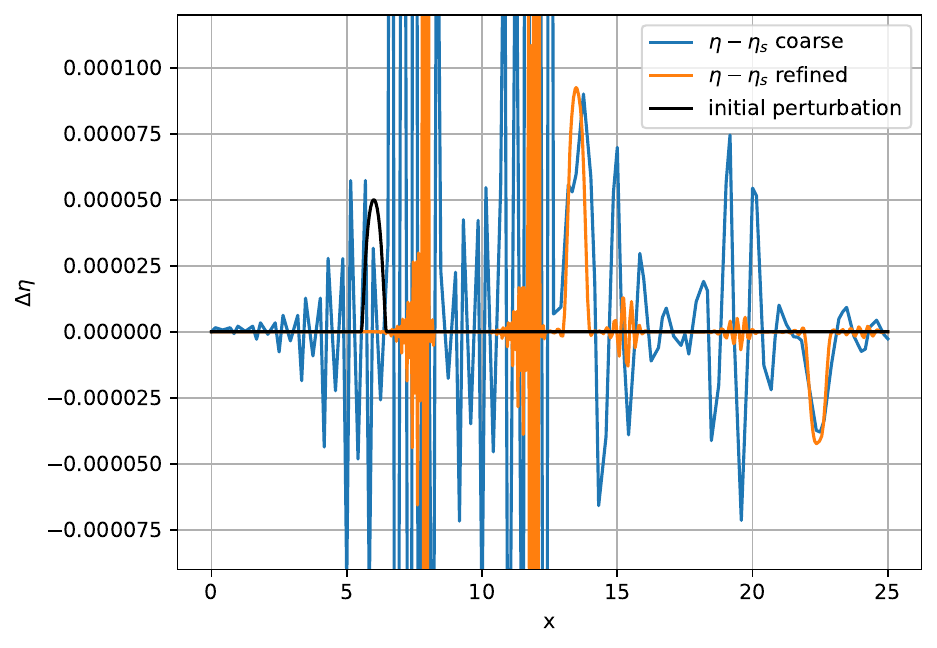}\caption{WB-$\HS$ with jt}
				\end{subfigure}
				\begin{subfigure}{0.31\textwidth}
					\includegraphics[width=\textwidth]{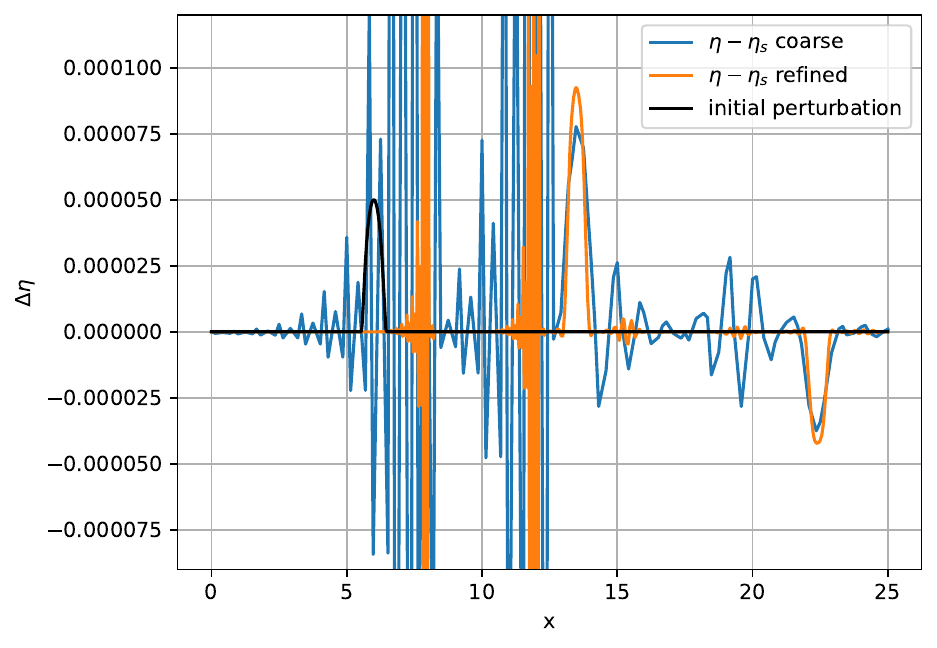}\caption{WB-$\HS$ with je}
				\end{subfigure}
				\begin{subfigure}{0.31\textwidth}
					\includegraphics[width=\textwidth]{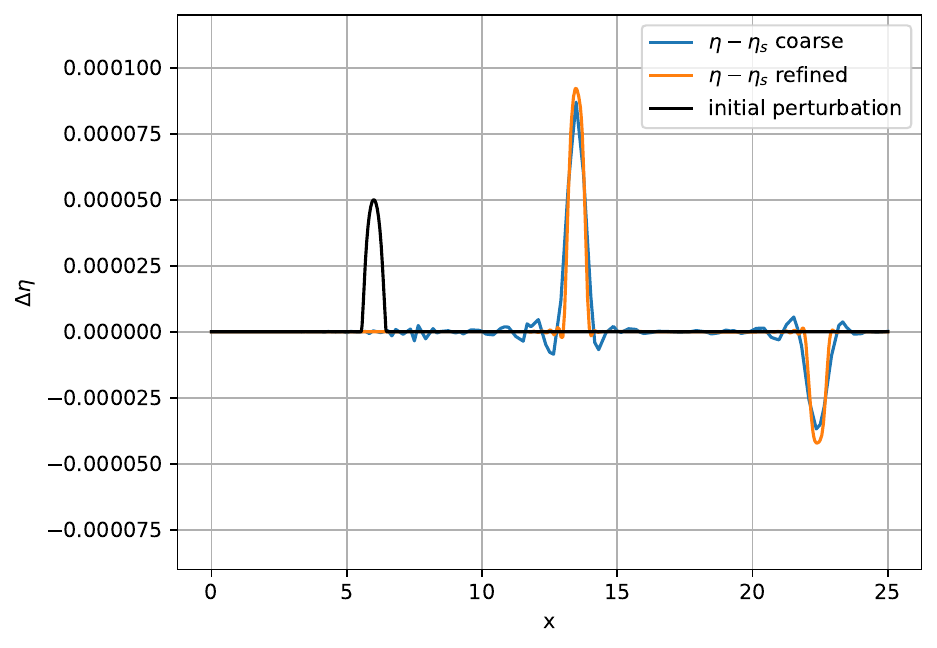}\caption{WB-$\HS$ with jr}
				\end{subfigure}\\
				\begin{subfigure}{0.31\textwidth}
					\includegraphics[width=\textwidth]{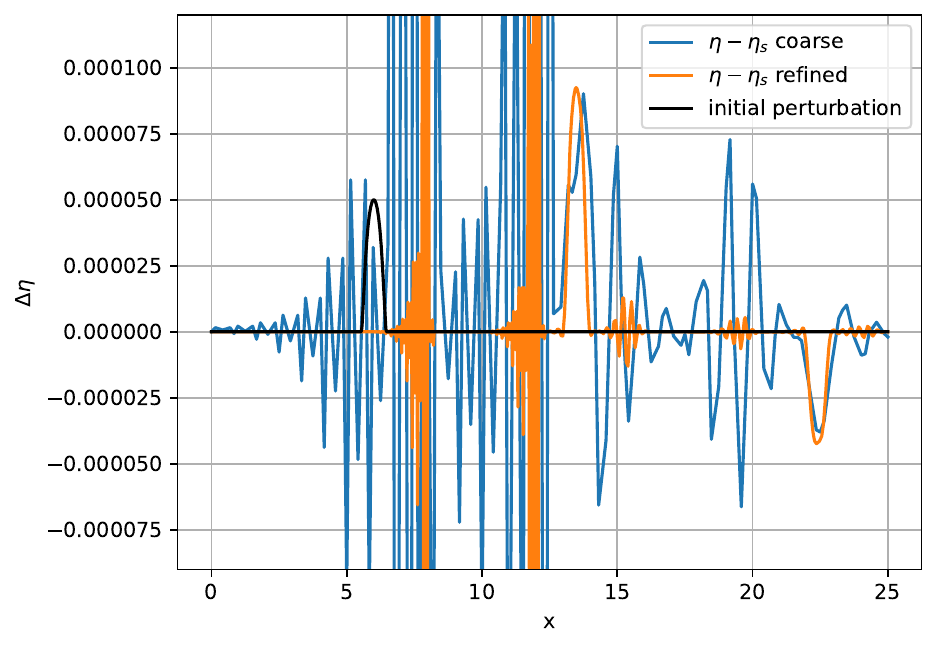}\caption{WB-$\GF$ with jt}
				\end{subfigure}
				\begin{subfigure}{0.31\textwidth}
					\includegraphics[width=\textwidth]{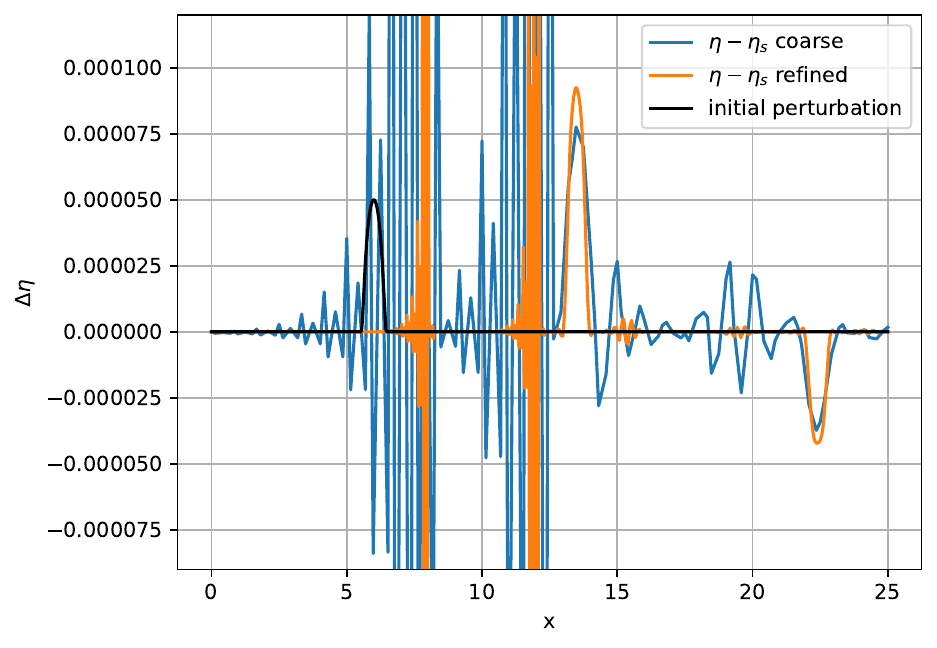}\caption{WB-$\GF$ with je}
				\end{subfigure}
				\begin{subfigure}{0.31\textwidth}
					\includegraphics[width=\textwidth]{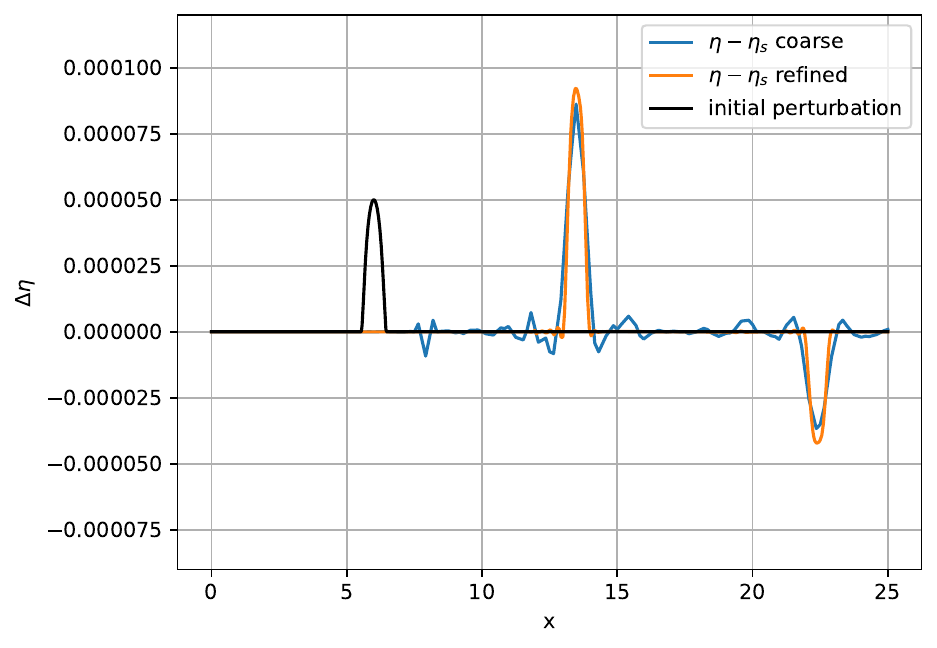}\caption{WB-$\GF$ with jg}
				\end{subfigure}
				\caption{Perturbation of non-smooth frictionless supercritical steady state: comparison between the different stabilizations. Results referred to PGL4 with $30$ elements for the coarse mesh and $128$ elements for the refined mesh}\label{fig:supwb}
			\end{figure}
			
			An interesting ``fair'' comparison between basis functions with different degrees is displayed in Figure \ref{fig:supwb_compare_PGLn}. The number of elements in the coarse meshes has been chosen in such a way that the total number of DoFs is constant. One can clearly see the effect of increasing the order of accuracy in the diminishing of the spurious oscillations both in number and magnitude, as well as in the better capturing of the peaks.
			
			\begin{figure}
				\centering
				\begin{subfigure}{0.31\textwidth}
					\includegraphics[width=\textwidth]{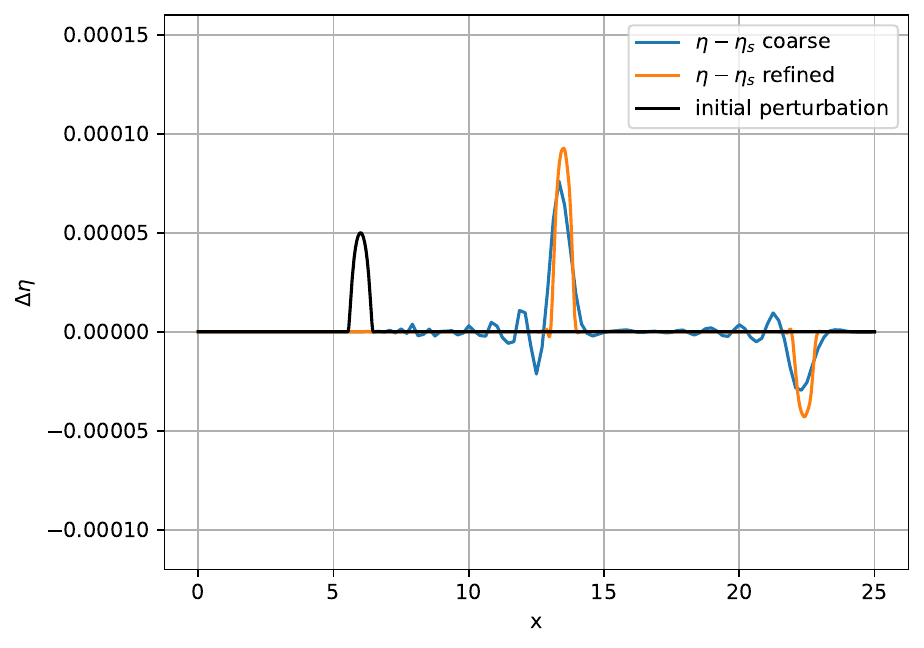}\caption{WB-$\HS$ with jr, PGL2}
				\end{subfigure}
				\begin{subfigure}{0.31\textwidth}
					\includegraphics[width=\textwidth]{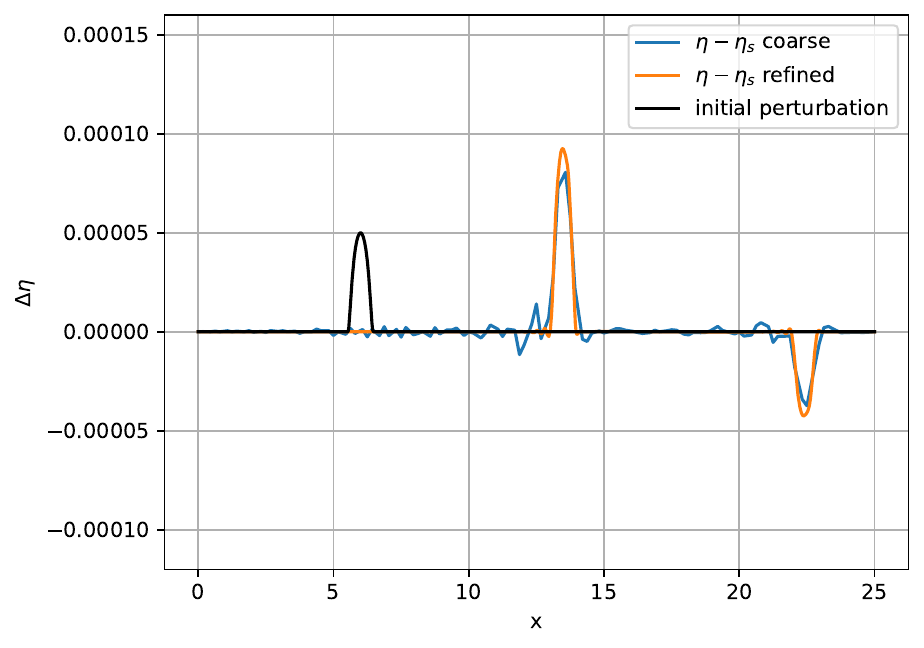}\caption{WB-$\HS$ with jr, PGL3}
				\end{subfigure}
				\begin{subfigure}{0.31\textwidth}
					\includegraphics[width=\textwidth]{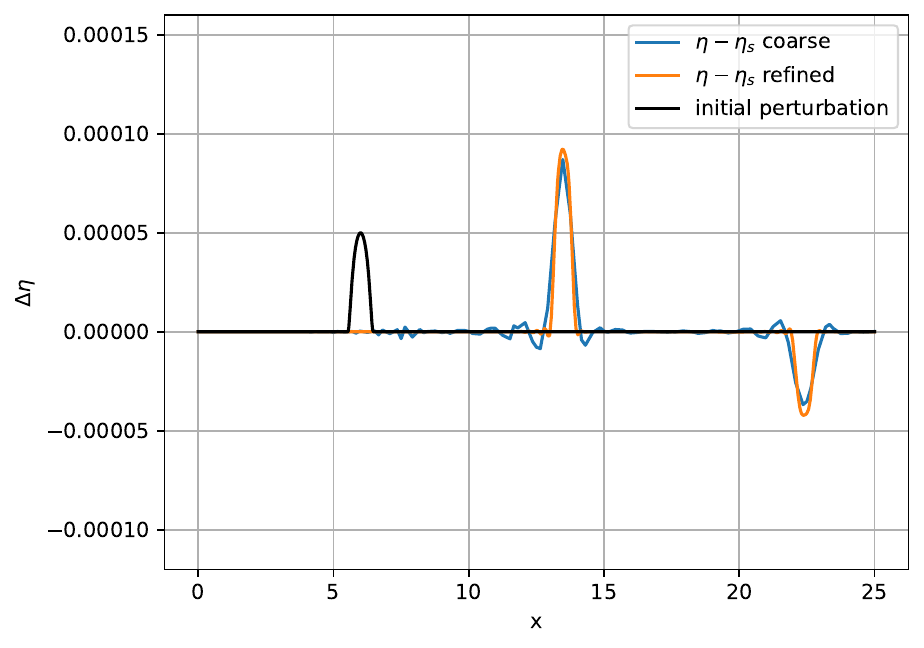}\caption{WB-$\HS$ with jr, PGL4}
				\end{subfigure}\\
				\begin{subfigure}{0.31\textwidth}
					\includegraphics[width=\textwidth]{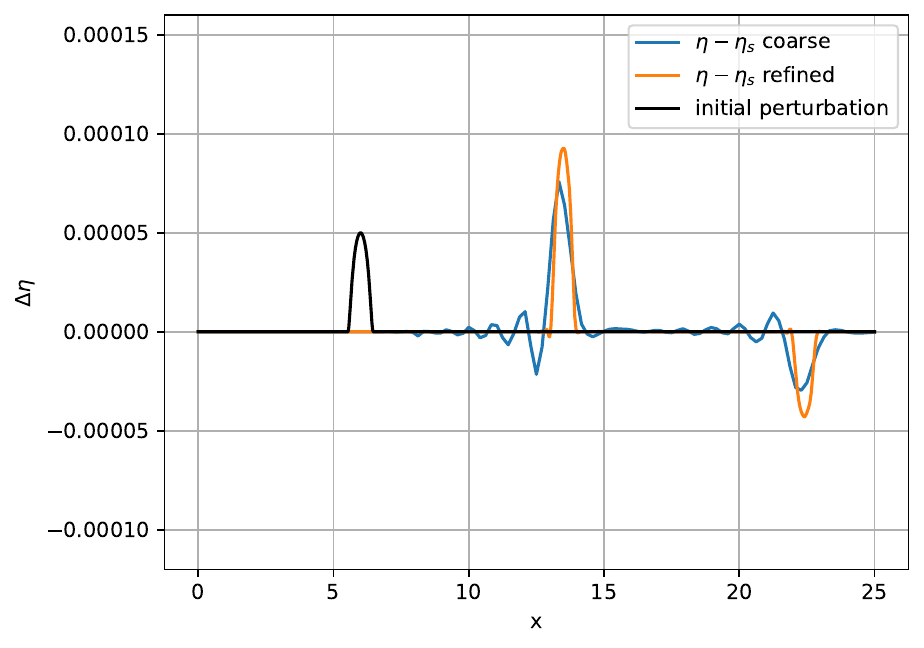}\caption{WB-$\GF$ with jg, PGL2}
				\end{subfigure}
				\begin{subfigure}{0.31\textwidth}
					\includegraphics[width=\textwidth]{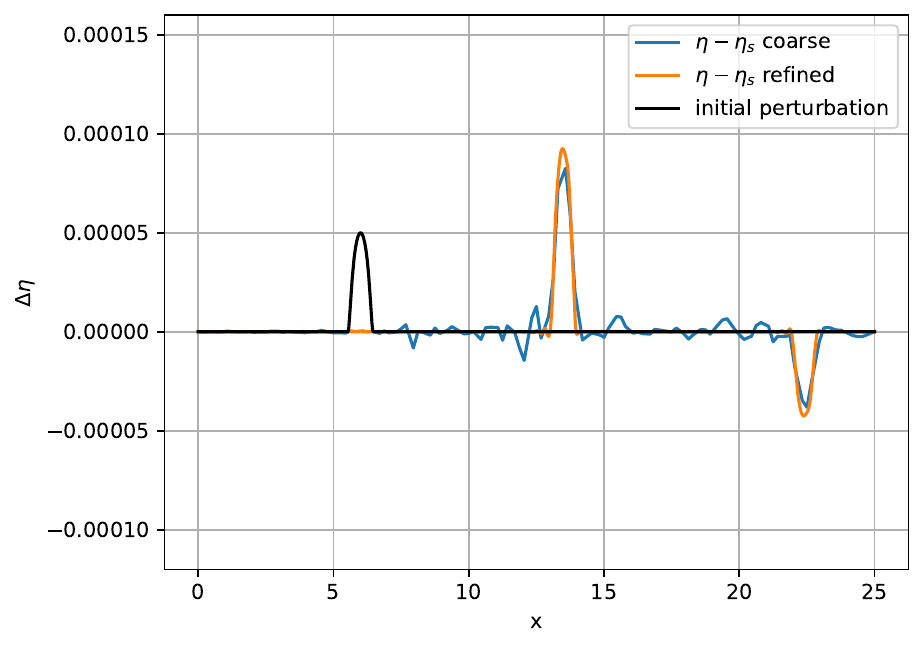}\caption{WB-$\GF$ with jg, PGL3}
				\end{subfigure}
				\begin{subfigure}{0.31\textwidth}
					\includegraphics[width=\textwidth]{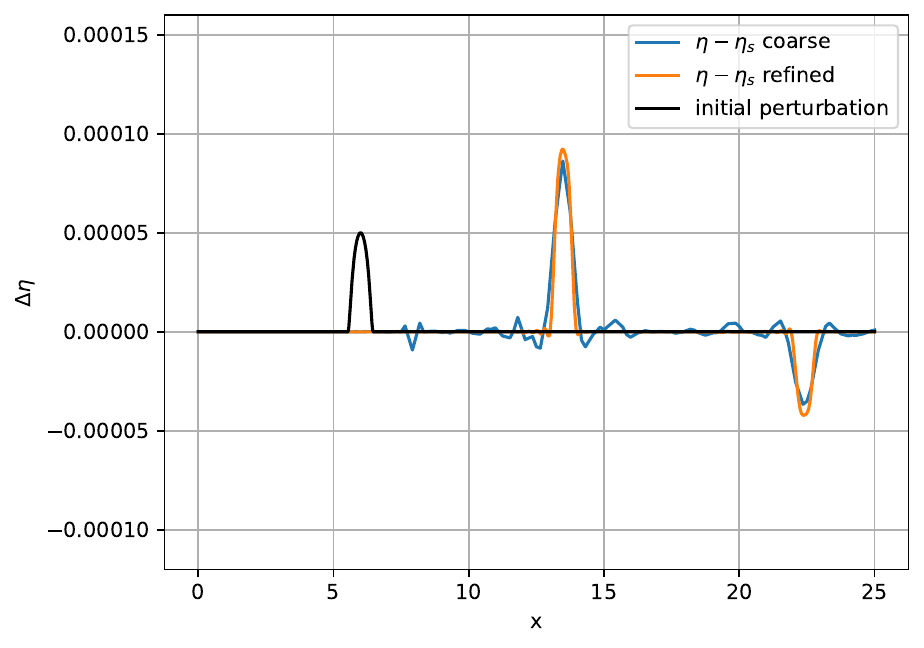}\caption{WB-$\GF$ with jg, PGL4}
				\end{subfigure}
				\caption{Perturbation of non-smooth frictionless supercritical steady state: fair comparison between basis functions of different degree with jr and jg. Respectively $30$, $40$ and $60$ elements for PGL4, PGL3 and PGL2 for the coarse meshes and $128$, $256$ and $512$ elements for the refined ones}\label{fig:supwb_compare_PGLn}
			\end{figure}

			\RIcolor{
				The ability to nicely capture the evolution of the perturbation even with order 3 on a coarse mesh should not be taken for granted: 
				the results obtained for WB-HS and WB-GF with jc, jt and je, for the same order and mesh resolution, are characterized by spurious oscillations, due to the discretization error, which completely overwhelm the dynamics of the perturbation. Such results have been omitted in order to improve the readability of this section.\\
			}

			\item[•] \textbf{Subcritical flow}\\
			The results got in this context are qualitatively similar to the ones obtained in the previous test, up to the fact that in this case we choose the same perturbation \eqref{eq:perturbation_lake_at_rest} but with $A:=5\cdot 10^{-4}$ and a final time $T_f:=1.5$. 
			
			Again, we start by showing, in Figure \ref{fig:subnonwb}, the unsatisfactory results that one gets in the context of the reference non-WB framework and with the two WB space discretizations coupled with the original stabilization jc. However, also in this case we underline how the WB space discretizations are definitely more suitable when one has to deal with a non-smooth bathymetry. Another common feature shared with the supercritical case is the presence of spurious oscillations due to the lack of any particular attention to well-balancing toward a general steady state.

			\begin{figure}
				\centering
				\begin{subfigure}{0.40\textwidth}
					\includegraphics[width=\textwidth]{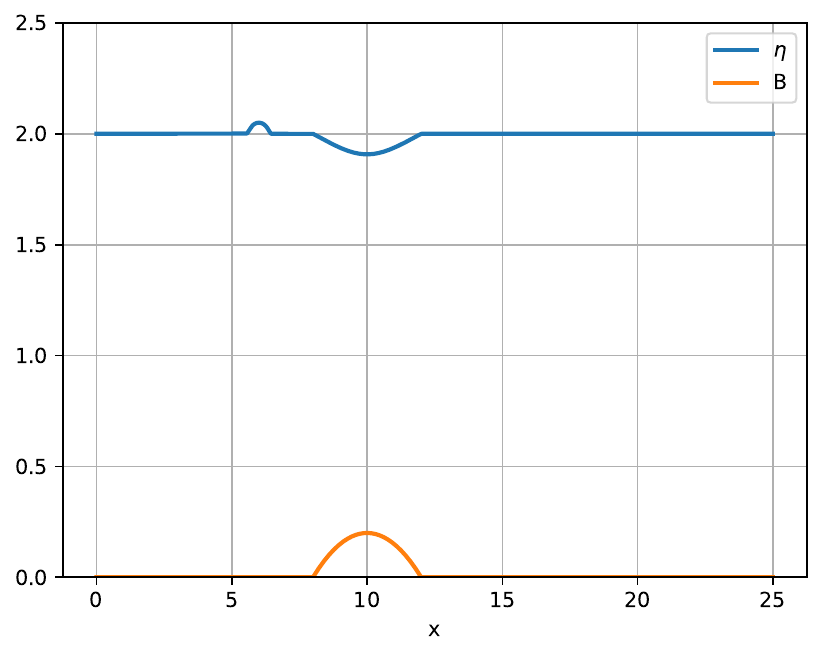}\caption{Initial total height and bathymetry. The perturbation is amplified by a factor $10^2$ in order to make it visible}
				\end{subfigure}\\
				\begin{subfigure}{0.31\textwidth}
					\includegraphics[width=\textwidth]{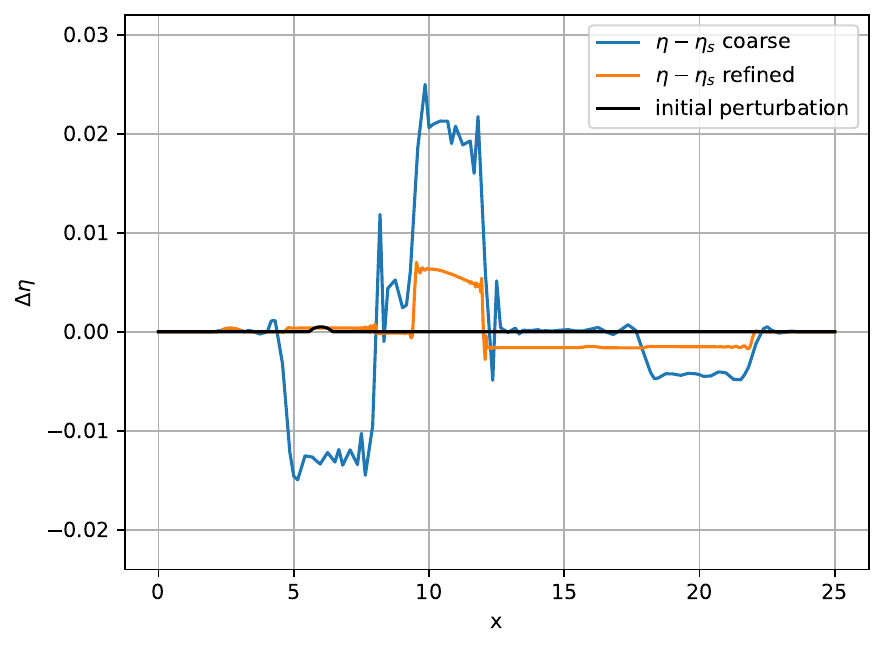}\caption{Reference non-WB setting}
				\end{subfigure}
				\begin{subfigure}{0.31\textwidth}
					\includegraphics[width=\textwidth]{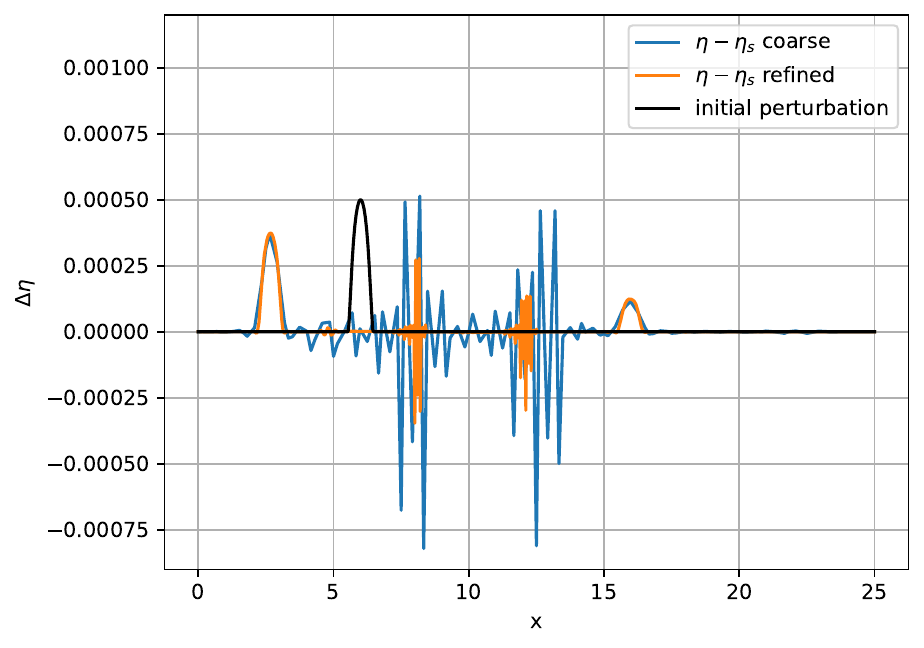}\caption{WB-$\HS$ with jc}
				\end{subfigure}
				\begin{subfigure}{0.31\textwidth}
					\includegraphics[width=\textwidth]{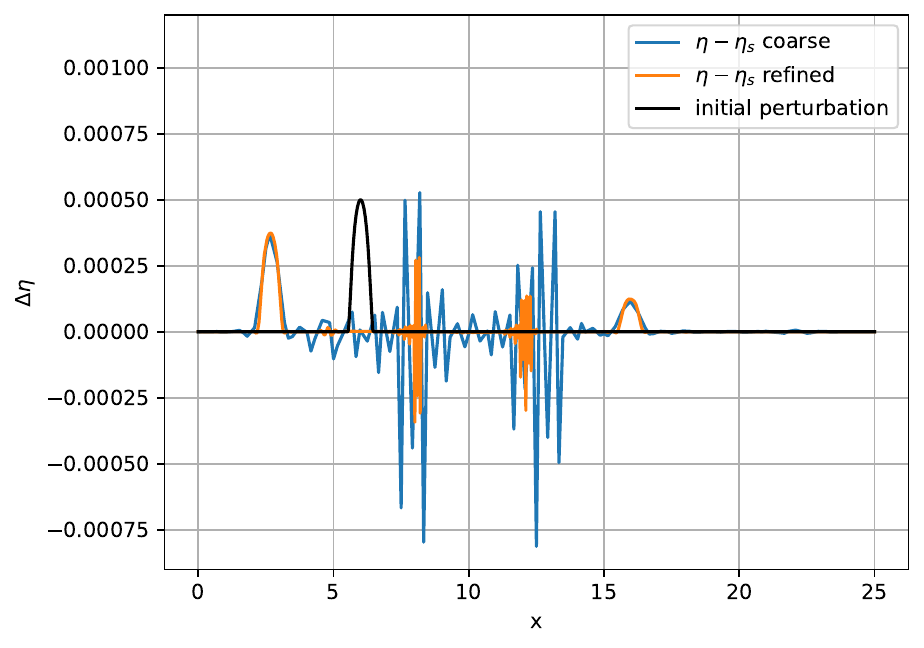}\caption{WB-$\GF$ with jc}
				\end{subfigure}
				\caption{Perturbation of non-smooth frictionless subcritical steady state: initial condition and results obtained with non-WB settings. Results referred to PGL4 with $30$ elements for the coarse mesh and $128$ elements for the refined mesh. A different scale has been used for the reference non-WB setting}\label{fig:subnonwb}
			\end{figure}

			The advantages of adopting an approach oriented towards the preservation of a general steady state can bee seen in Figure \ref{fig:subwb}. Again, the WB stabilizations jr and jg manage to remove almost completely the non-physical oscillations, which totally disappear in the mesh refinement. Also in this case, we remark that no limiting strategy has been adopted and this is the reason for the little fluctuations that one can see in the results associated to the coarse meshes.

			\begin{figure}
				\centering
				\begin{subfigure}{0.31\textwidth}
					\includegraphics[width=\textwidth]{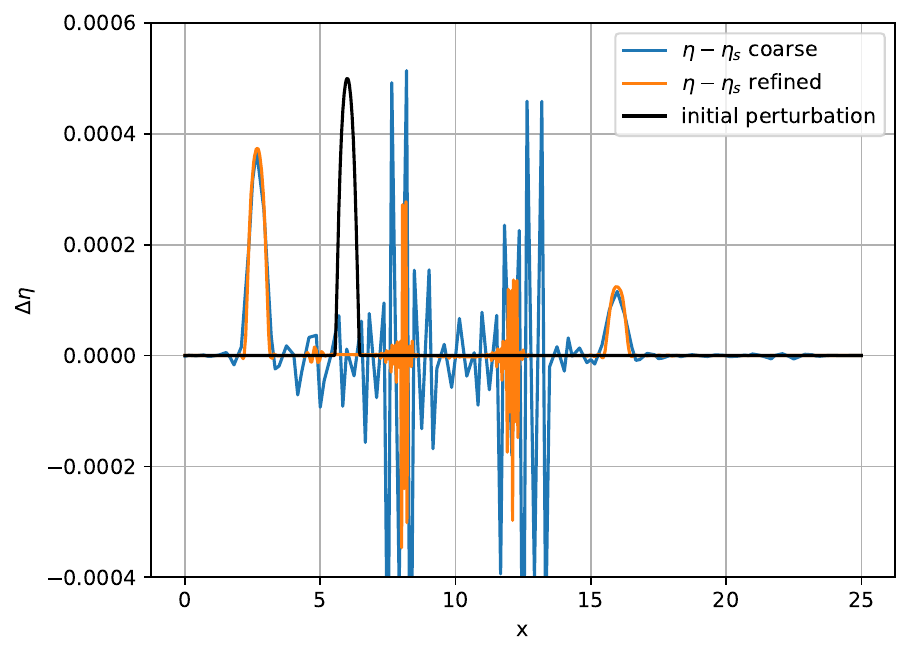}\caption{WB-$\HS$ with jc}
				\end{subfigure}
				\begin{subfigure}{0.31\textwidth}
					\includegraphics[width=\textwidth]{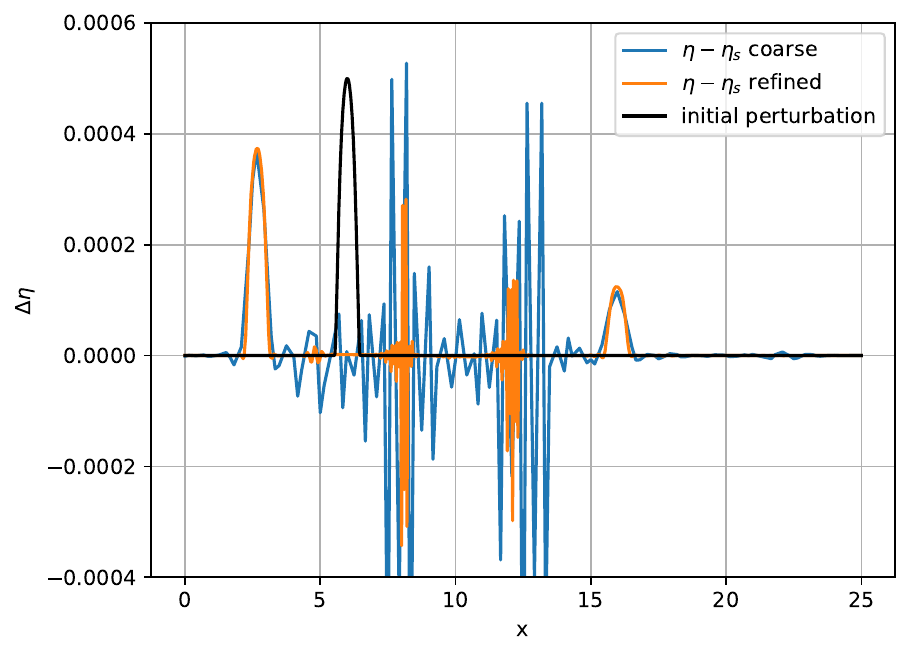}\caption{WB-$\GF$ with jc}
				\end{subfigure}\\
				\begin{subfigure}{0.31\textwidth}
					\includegraphics[width=\textwidth]{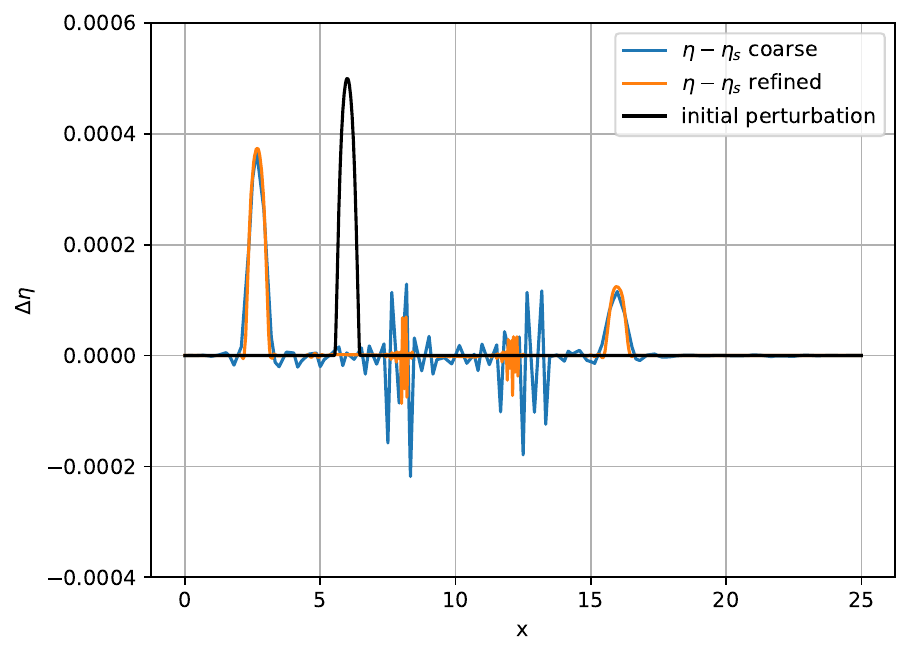}\caption{WB-$\HS$ with jt}
				\end{subfigure}
				\begin{subfigure}{0.31\textwidth}
					\includegraphics[width=\textwidth]{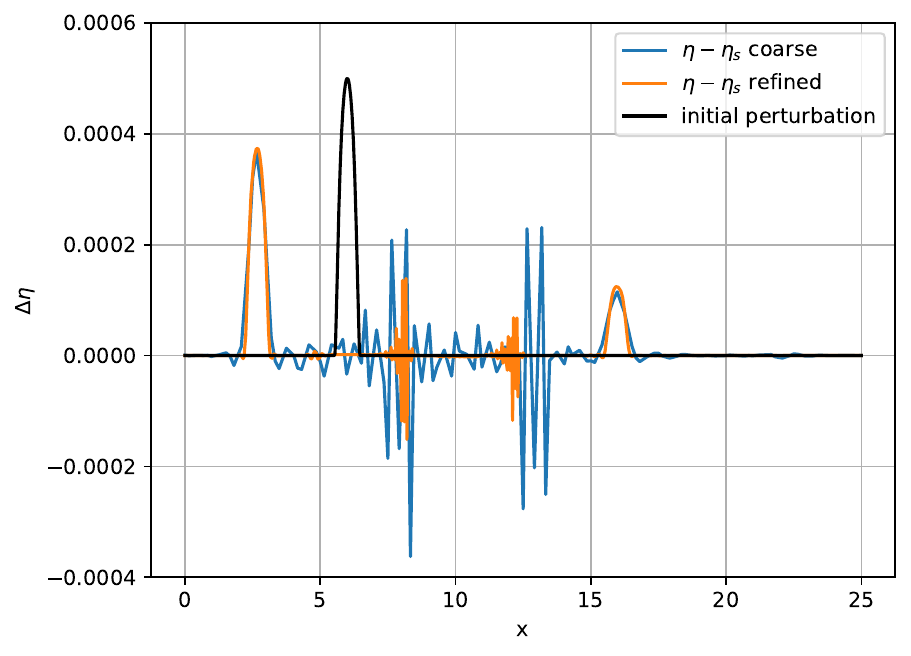}\caption{WB-$\HS$ with je}
				\end{subfigure}
				\begin{subfigure}{0.31\textwidth}
					\includegraphics[width=\textwidth]{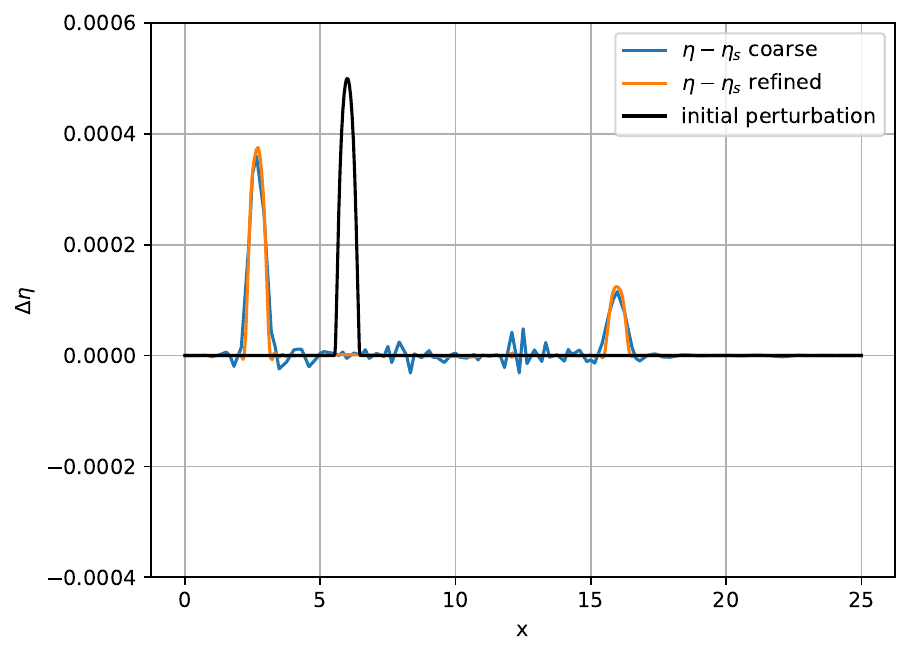}\caption{WB-$\HS$ with jr}
				\end{subfigure}\\
				\begin{subfigure}{0.31\textwidth}
					\includegraphics[width=\textwidth]{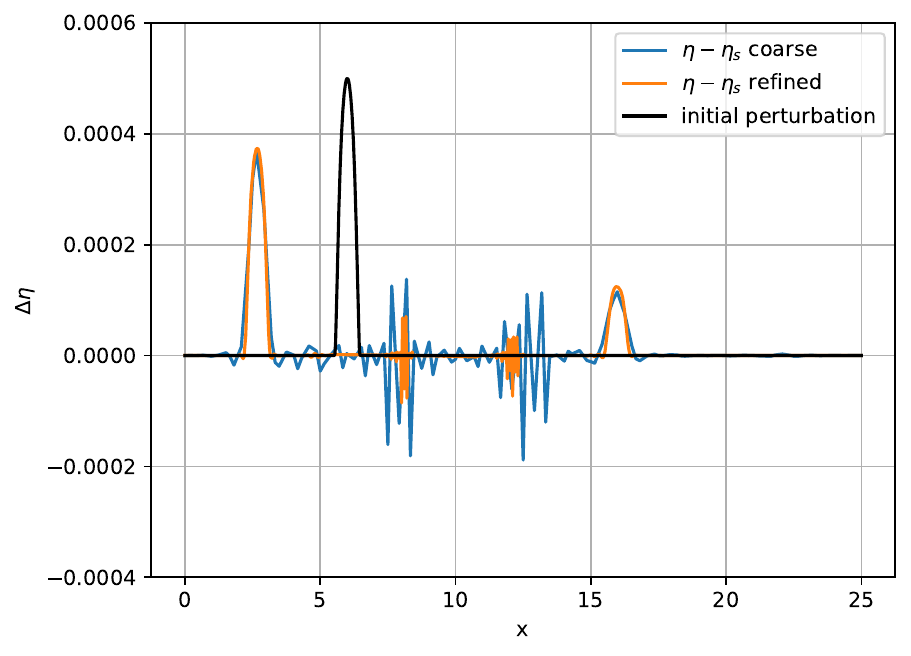}\caption{WB-$\GF$ with jt}
				\end{subfigure}
				\begin{subfigure}{0.31\textwidth}
					\includegraphics[width=\textwidth]{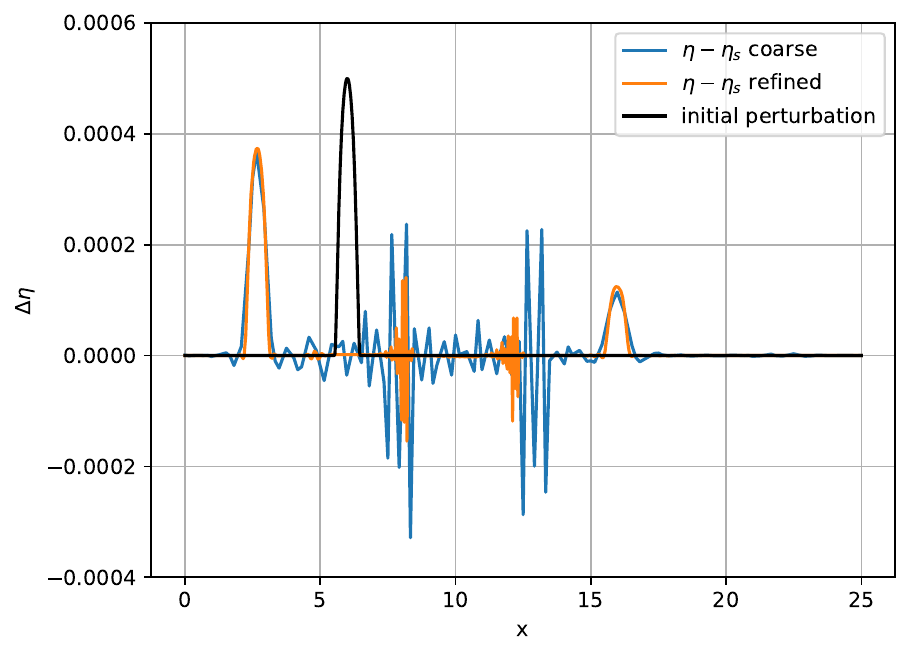}\caption{WB-$\GF$ with je}
				\end{subfigure}
				\begin{subfigure}{0.31\textwidth}
					\includegraphics[width=\textwidth]{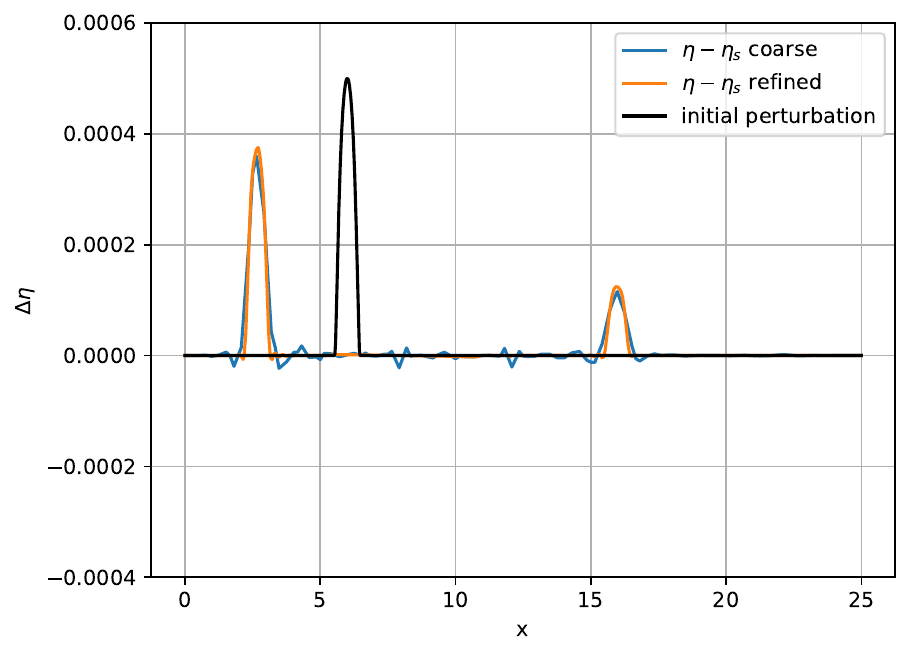}\caption{WB-$\GF$ with jg}
				\end{subfigure}
				\caption{Perturbation of non-smooth frictionless subcritical steady state: comparison between the different stabilizations. Results referred to PGL4 with $30$ elements for the coarse mesh and $128$ elements for the refined mesh}\label{fig:subwb}
			\end{figure}

			For the sake of compactness, we omit here other results but analogous considerations to the ones reported at the end of the tests concerning the perturbation of the supercritical case hold.\\
			
			\item[•] \textbf{Transcritical flow}\\
			The perturbation assumed in this context is identical to the one assumed in the subcritical case, i.e., \eqref{eq:perturbation_lake_at_rest} with $A:=5\cdot 10^{-4}$. We consider the same final time $T_f:=1.5$.
			The initial condition is displayed in Figure \ref{fig:transnonwb}.
			
			\begin{figure}
				\centering
				\includegraphics[width=0.4\textwidth]{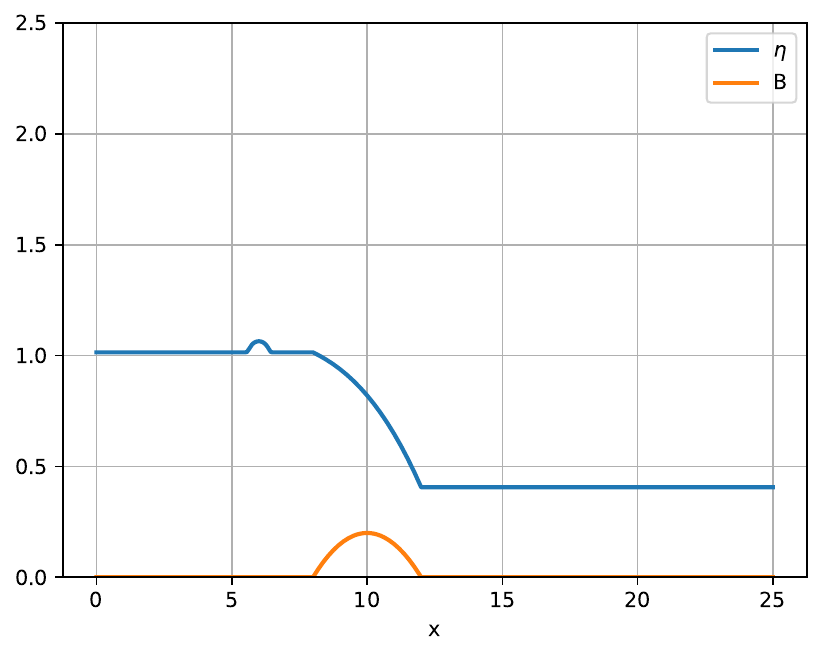}\caption{Perturbation of non-smooth frictionless transcritical steady state: initial total height and bathymetry. The perturbation is amplified by a factor $10^2$ in order to make it visible}
				\label{fig:transnonwb}
			\end{figure}
			
			\RIcolor{
				The results got for this steady state are absolutely analogous to the ones obtained in the previous cases and, therefore, for the sake of compactness, we directly focus on the best performing settings. 
				Figure~\ref{fig:transwb} shows the good results obtained for WB-HS-jr and WB-GF-jg, along with the oscillatory results obtained coupling the two space discretizations with jt in order to have a comparison with a stabilization not designed for the preservation of generic steady states.
				The results obtained with the other stabilizations, jc and je, are similar to the ones obtained with jt and, hence, they have been omitted.
			}%
			%
			The advantages of basing the stabilization on $\frac{\partial}{\partial x}\uvec{F}-\uvec{S}$ are pretty evident. The spurious oscillations obtained with jr and jg are much smaller and controlled in terms of number and magnitude. The little ones still present in such cases, due to a lack of limiting, fade away in the mesh refinement.

			\begin{figure}
				\centering
				\begin{subfigure}{0.31\textwidth}
					\includegraphics[width=\textwidth]{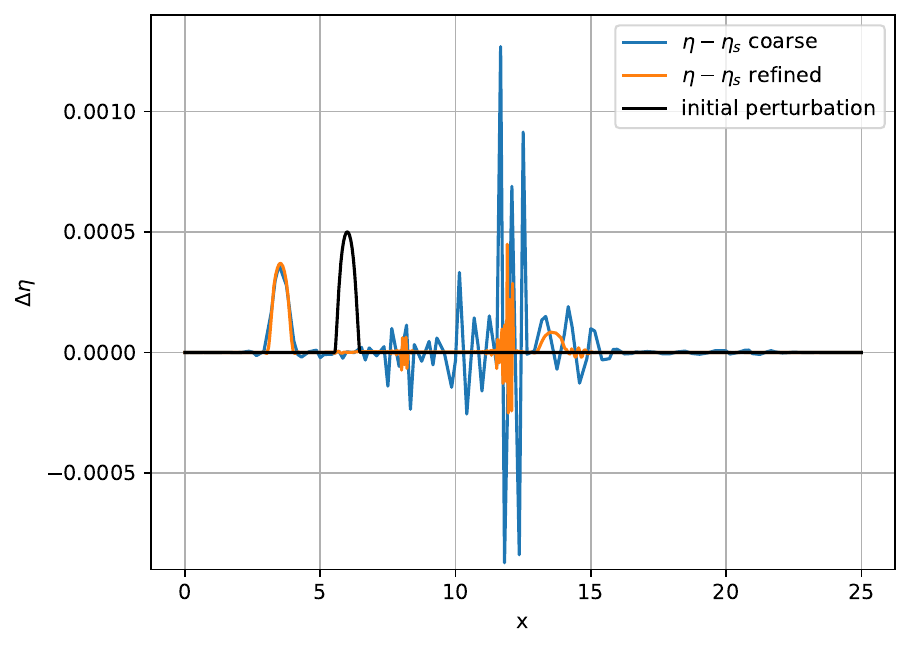}\caption{WB-$\HS$ with jt}
				\end{subfigure}
				\begin{subfigure}{0.31\textwidth}
					\includegraphics[width=\textwidth]{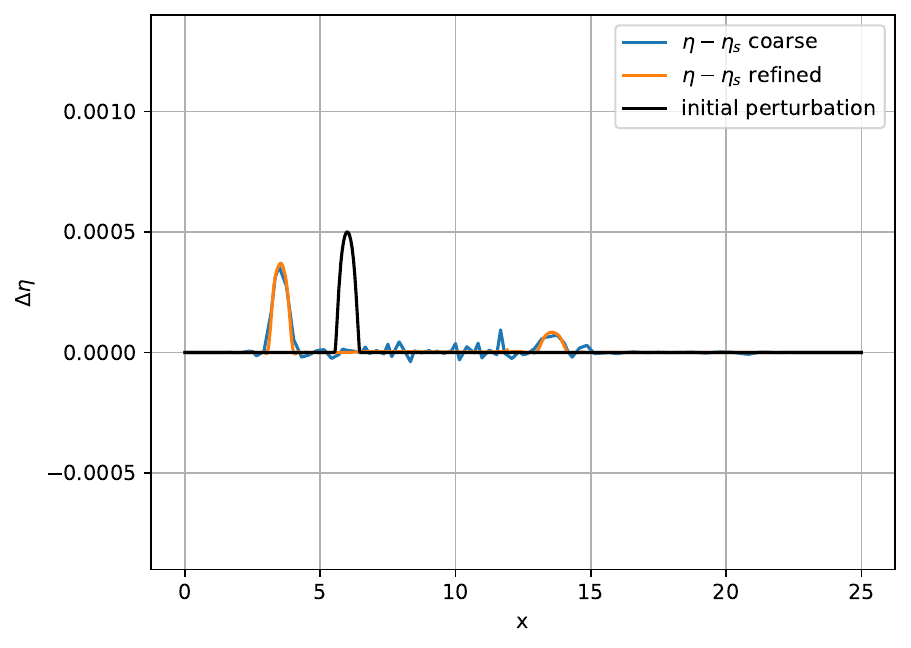}\caption{WB-$\HS$ with jr}
				\end{subfigure}\\
				\begin{subfigure}{0.31\textwidth}
					\includegraphics[width=\textwidth]{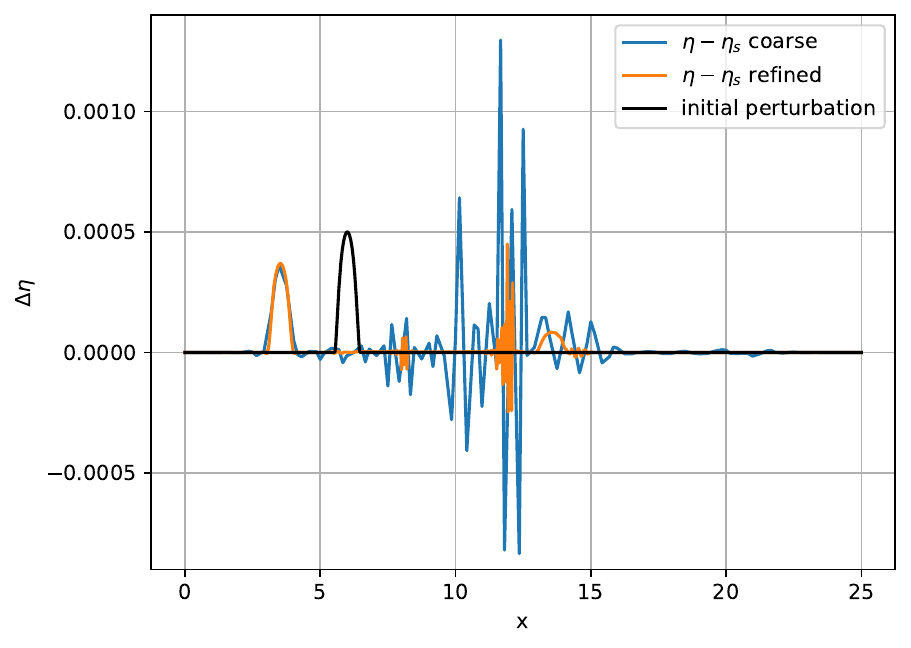}\caption{WB-$\GF$ with jt}
				\end{subfigure}
				\begin{subfigure}{0.31\textwidth}
					\includegraphics[width=\textwidth]{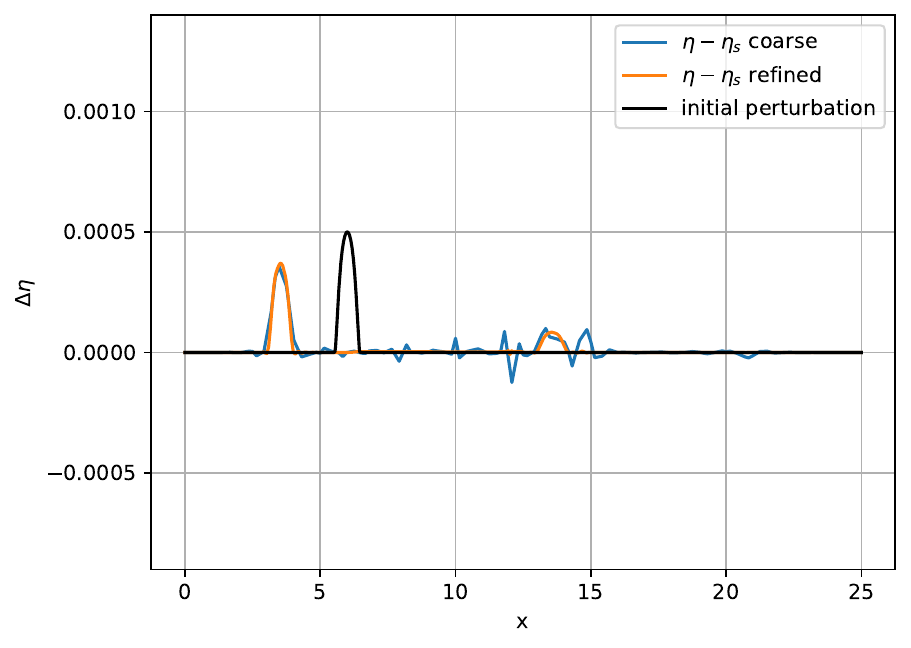}\caption{WB-$\GF$ with jg}
				\end{subfigure}
				\caption{Perturbation of non-smooth frictionless transcritical steady state: comparison between different stabilizations. Results referred to PGL4 with $30$ elements for the coarse mesh and $128$ elements for the refined mesh}\label{fig:transwb}
			\end{figure}
			
			For the sake of compactness, we omit other results but we remark that there are no significative differences with respect to the other non-smooth frictionless steady states.

		\end{itemize}

		\subsubsection{Tests with friction}
		In this section, we will only focus on the supercritical and on the subcritical flows respectively characterized by the boundary conditions \eqref{eq:superBC} and \eqref{eq:subBC}. In particular, we assume the usual $C^0$ bathymetry \eqref{eq:c0_bathymetry} and $n_M:=0.03$. 
		
		Just like in the frictionless tests of the previous section, the steady states are not available in closed-form. Moreover, in this context, \eqref{eq:steady_no_friction} does not hold and there is no way to exactly compute the water height. Therefore, the steady states have been obtained by running simulations with very refined meshes, with $2048$ elements and P1 basis functions, for time long enough and, finally, transferred to the meshes used for the tests through interpolation of $\eta$ and $q$. In this context, as initial conditions, we adopted the frictionless steady states of the previous section. Further, for coherence, for each simulation involving the evolution of the perturbation of a steady state with a specific setting, the steady state obtained through the same setting has been adopted as reference.
		
		The perturbations and the final times assumed here are the same as the ones assumed in the frictionless case in the analogous tests of the previous section.
		
		\begin{itemize}
			\item[•] \textbf{Supercritical flow}\\
			We start by showing, in Figure \ref{fig:sup_fric_steady}, the numerical steady states obtained with different settings. We can see how all the results provided are consistent: the friction causes a speed decrease in the direction of the flow, which, due to the constant momentum, is responsible for the general increase of the total height, from left to right, not present in the frictionless case.
			Concerning $\eta$, we cannot appreciate any macroscopical difference between the approximations provided by the different schemes. For what concerns $q$, instead, the reader is invited to notice the different scales used for the different settings: jr and jg are the only stabilizations able to capture the constant momentum up to machine precision. In all the other cases, the oscillations in correspondence of the discontinuities of the first derivative of the bathymetry are of the order of $10^{-3}\sim 10^{-4}.$

			\begin{figure}
				\centering
				\begin{subfigure}{0.31\textwidth}
					\includegraphics[width=\textwidth]{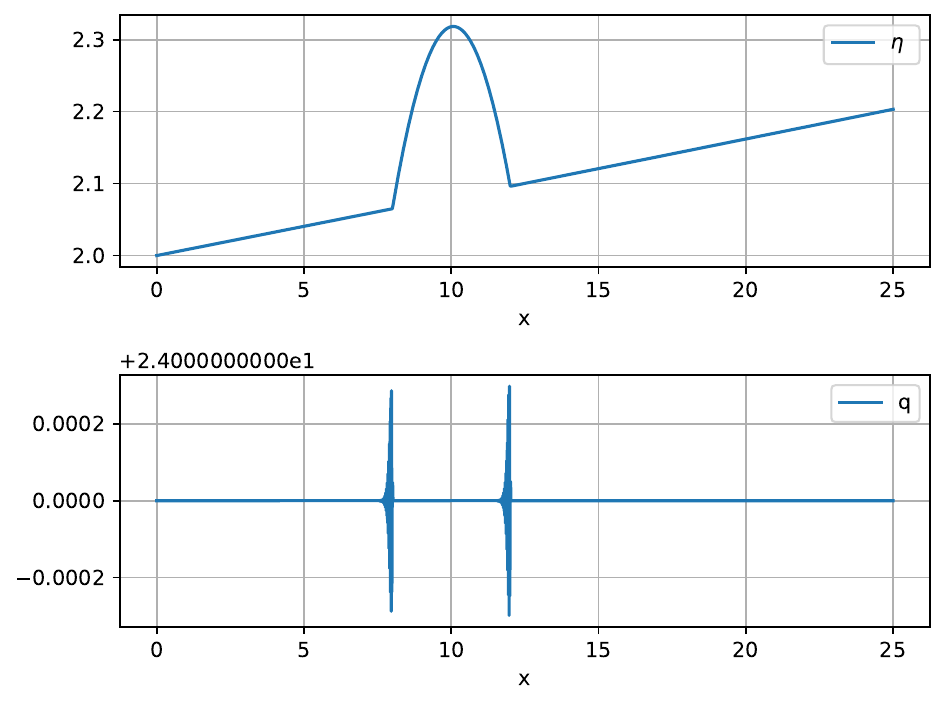}\caption{Reference non-WB}
				\end{subfigure}
				\begin{subfigure}{0.31\textwidth}
					\includegraphics[width=\textwidth]{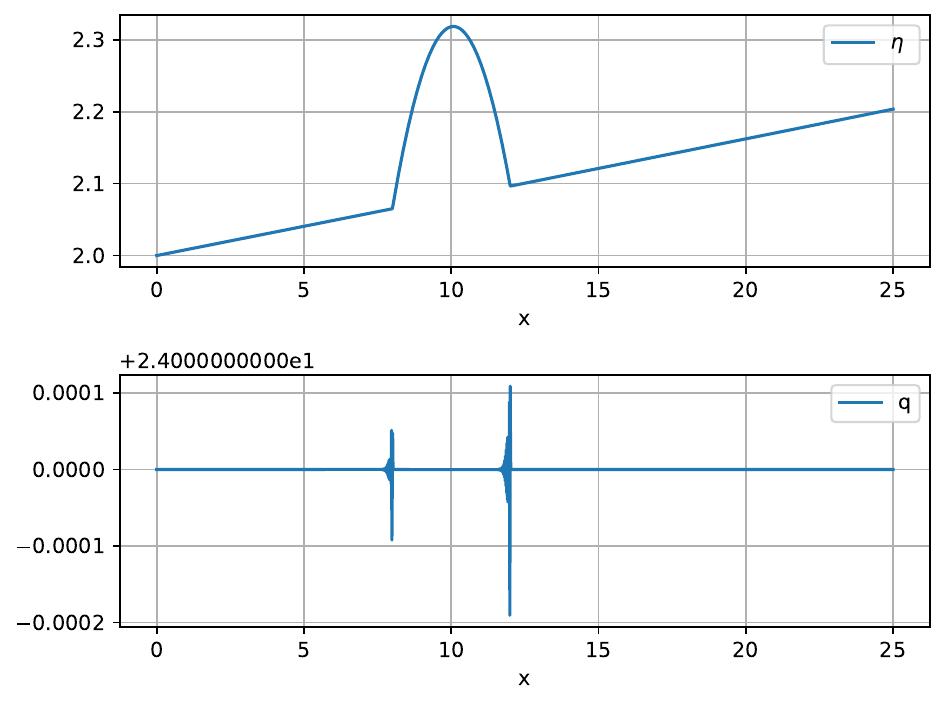}\caption{WB-$\HS$ with jc}
				\end{subfigure}
				\begin{subfigure}{0.31\textwidth}
					\includegraphics[width=\textwidth]{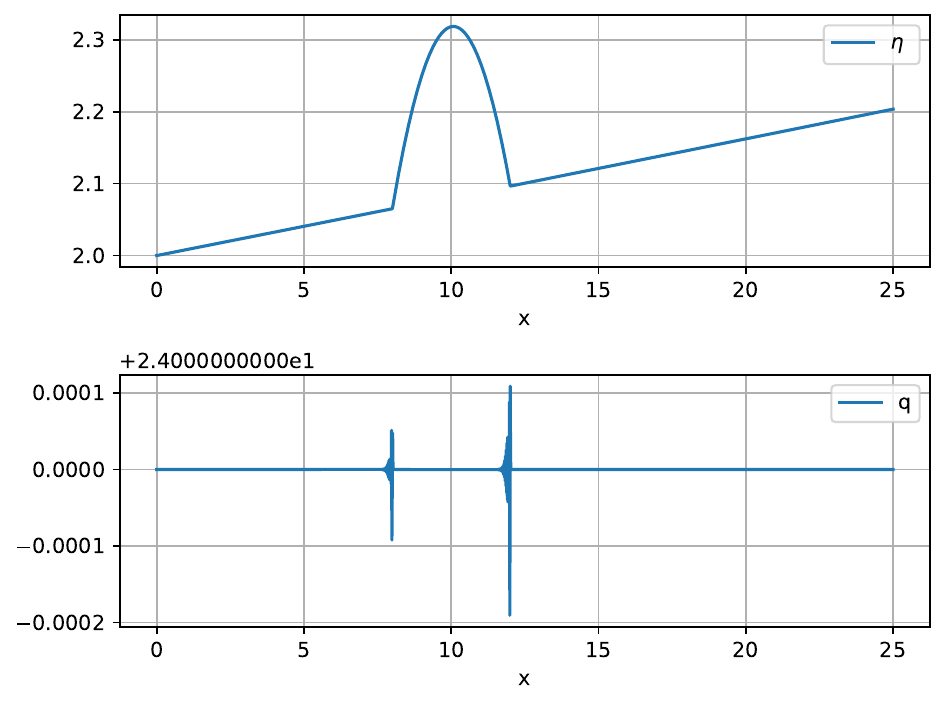}\caption{WB-$\GF$ with jc}
				\end{subfigure}\\
				\begin{subfigure}{0.31\textwidth}
					\includegraphics[width=\textwidth]{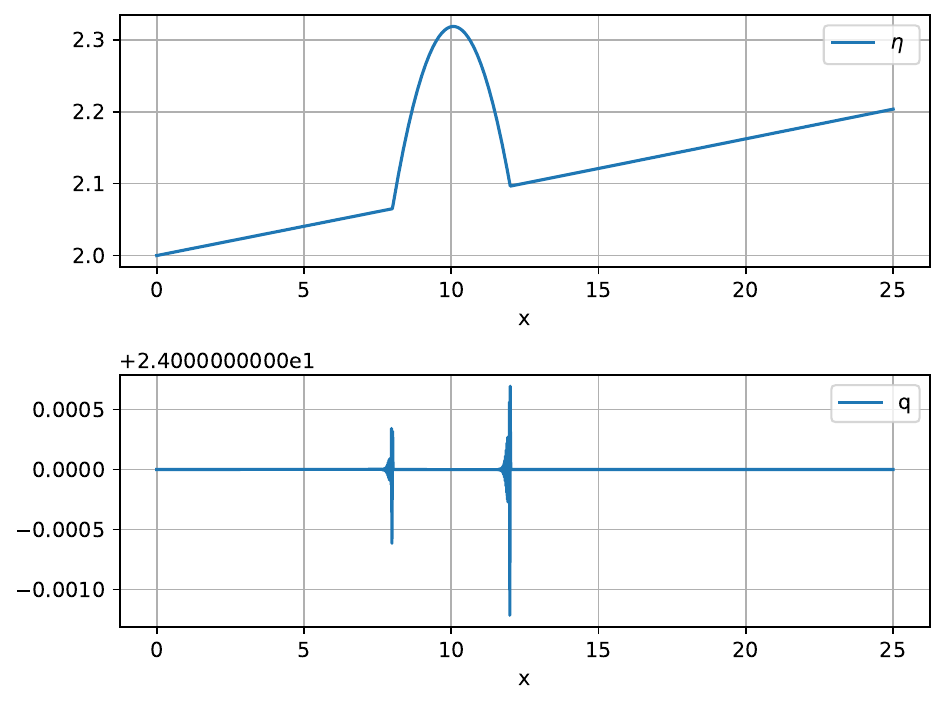}\caption{WB-$\HS$ with jt}
				\end{subfigure}
				\begin{subfigure}{0.31\textwidth}
					\includegraphics[width=\textwidth]{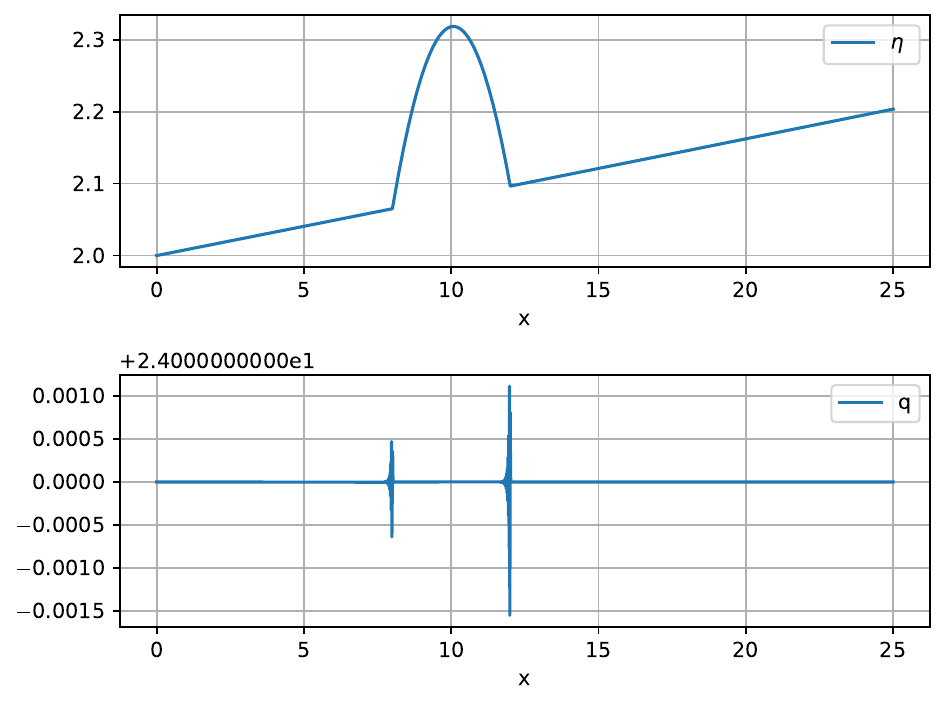}\caption{WB-$\HS$ with je}
				\end{subfigure}
				\begin{subfigure}{0.31\textwidth}
					\includegraphics[width=\textwidth]{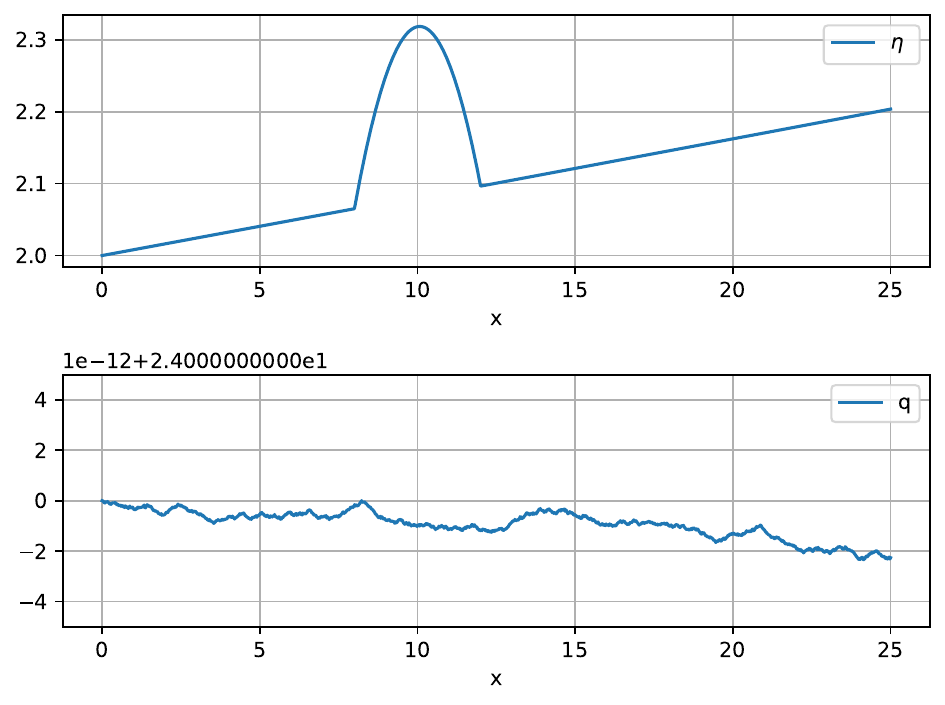}\caption{WB-$\HS$ with jr}
				\end{subfigure}\\
				\begin{subfigure}{0.31\textwidth}
					\includegraphics[width=\textwidth]{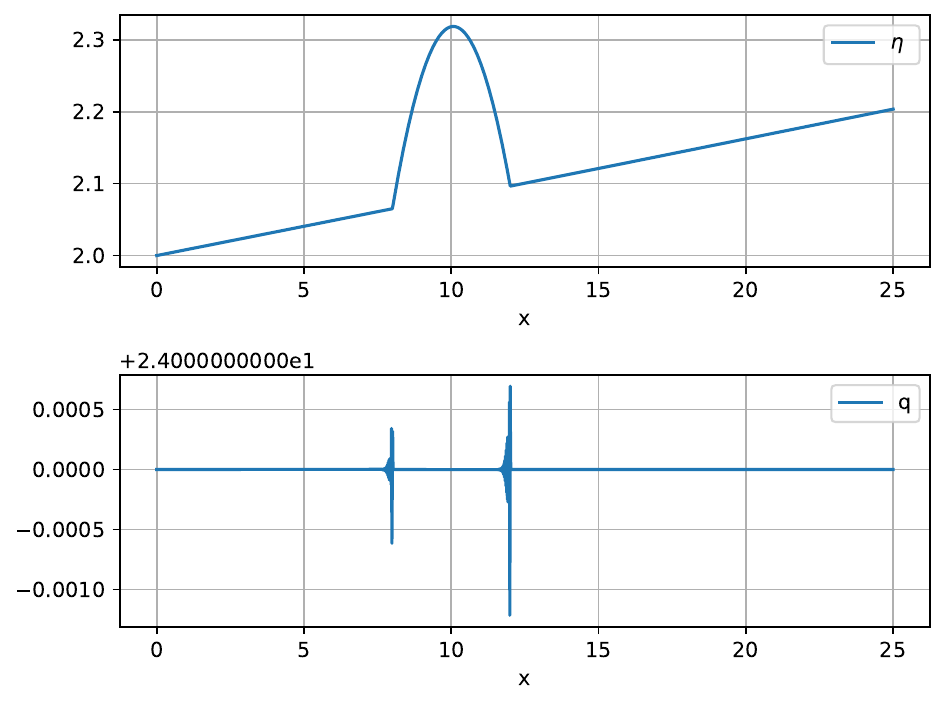}\caption{WB-$\GF$ with jt}
				\end{subfigure}
				\begin{subfigure}{0.31\textwidth}
					\includegraphics[width=\textwidth]{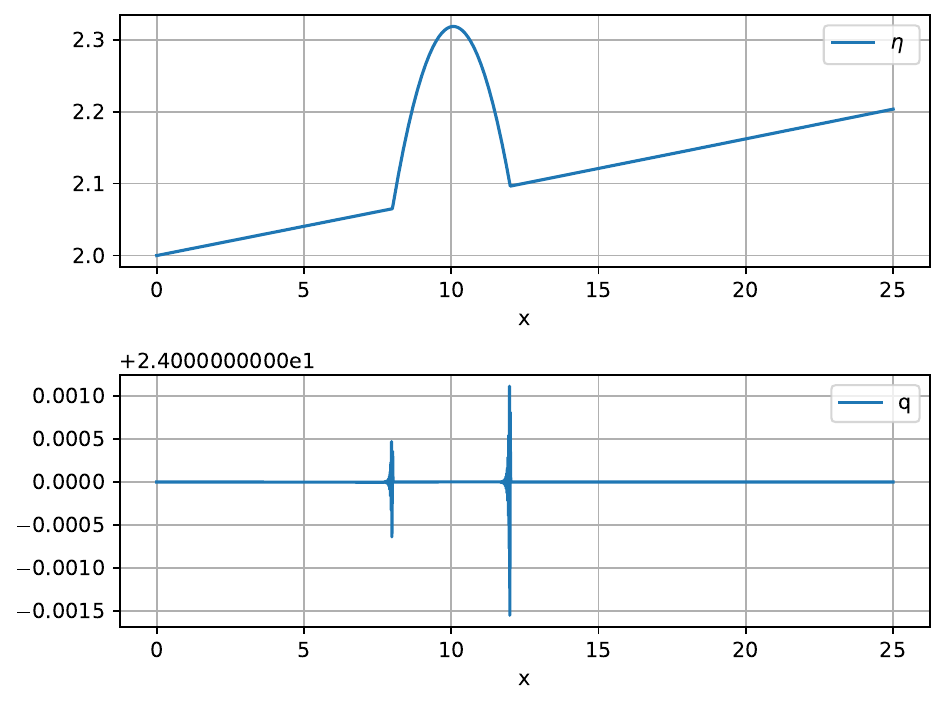}\caption{WB-$\GF$ with je}
				\end{subfigure}
				\begin{subfigure}{0.31\textwidth}
					\includegraphics[width=\textwidth]{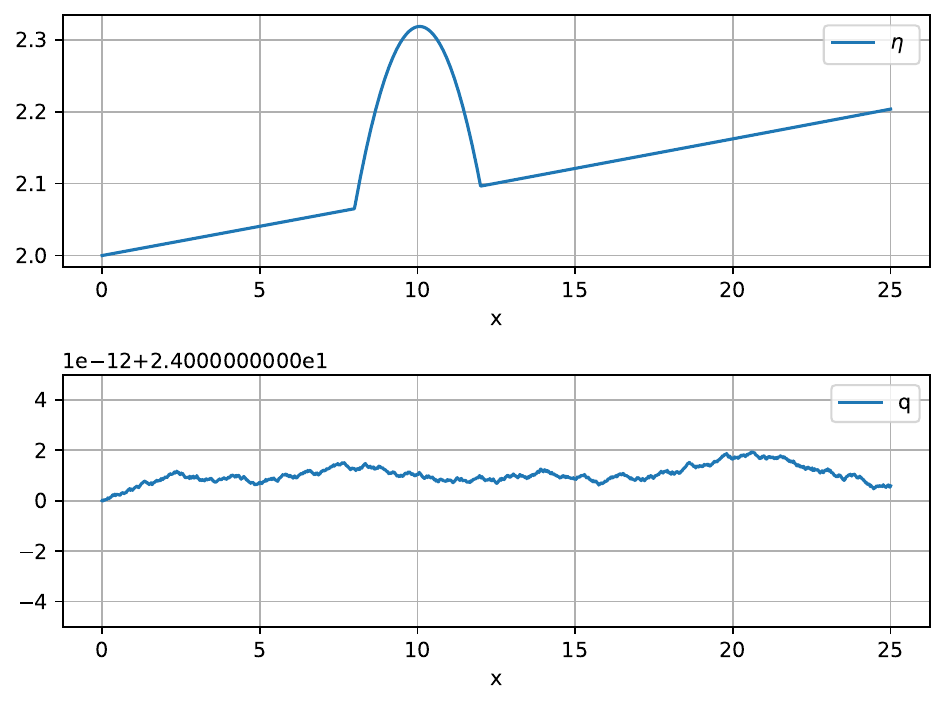}\caption{WB-$\GF$ with jg}
				\end{subfigure}
				\caption{Non-smooth supercritical steady state with friction: numerical steady state obtained with different settings. Results referred to P1 with $2048$ elements. Different scales have been used for $q$}\label{fig:sup_fric_steady}
			\end{figure}
			
			We continue now with the perturbation analysis.
			\RIcolor{%
				The results are similar to the ones retrieved in the analogous frictionless test.
				The evolution of the perturbation obtained with different settings is reported in Figure~\ref{fig:supwb_fric}.
				As usual, the results obtained for jr and jg are much better than the ones obtained with the other settings and analogous considerations hold with respect to the frictionless case. 
				We remark that the (omitted) results got with jc and je coupled with WB-HS and WB-GF were analogous to the ones got with jt.\\
			}
			%

			
			\begin{figure}
				\centering
				\begin{subfigure}{0.31\textwidth}
					\includegraphics[width=\textwidth]{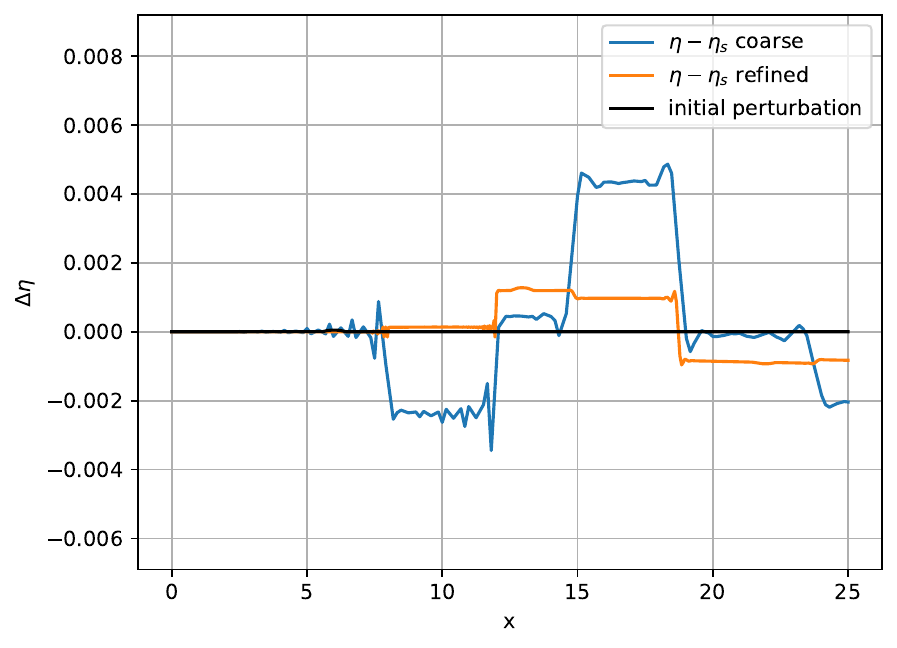}\caption{Reference non-WB setting}
				\end{subfigure}
				\begin{subfigure}{0.31\textwidth}
					\includegraphics[width=\textwidth]{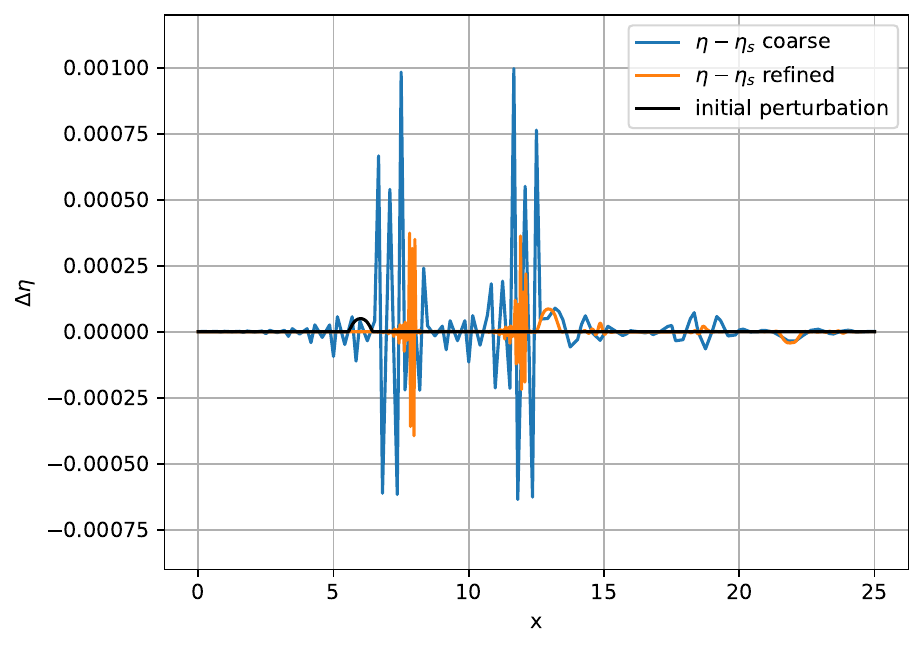}\caption{WB-$\HS$ with jt}
				\end{subfigure}
				\begin{subfigure}{0.31\textwidth}
					\includegraphics[width=\textwidth]{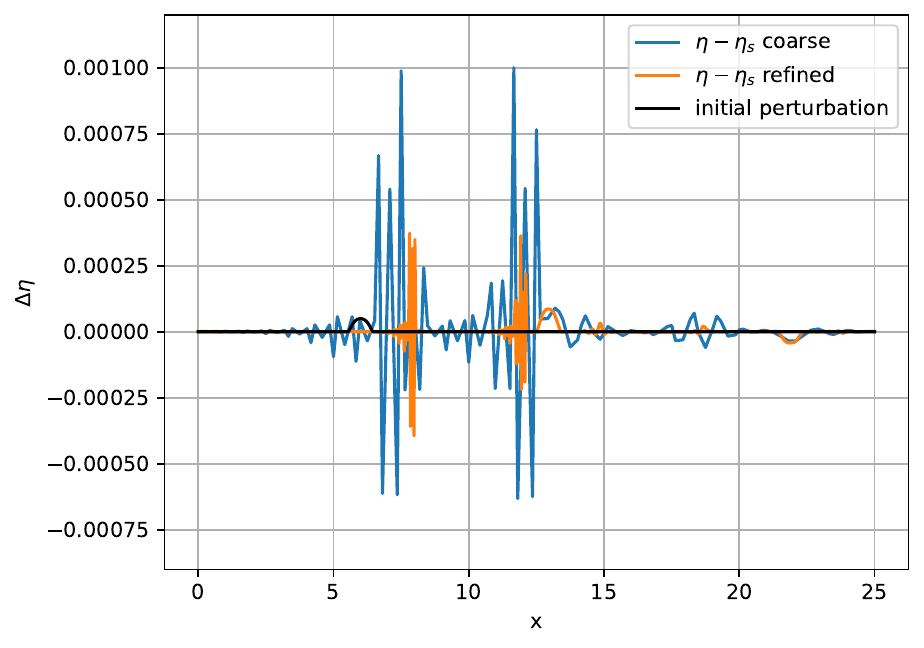}\caption{WB-$\GF$ with jt}
				\end{subfigure}\\
				\begin{subfigure}{0.31\textwidth}
					\includegraphics[width=\textwidth]{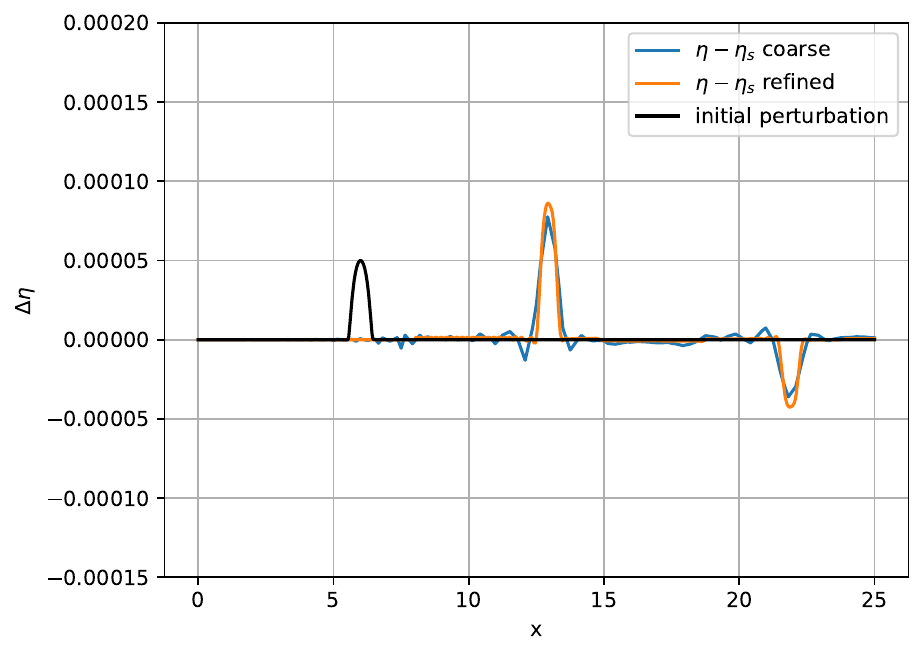}\caption{WB-$\HS$ with jr}
				\end{subfigure}
				\begin{subfigure}{0.31\textwidth}
					\includegraphics[width=\textwidth]{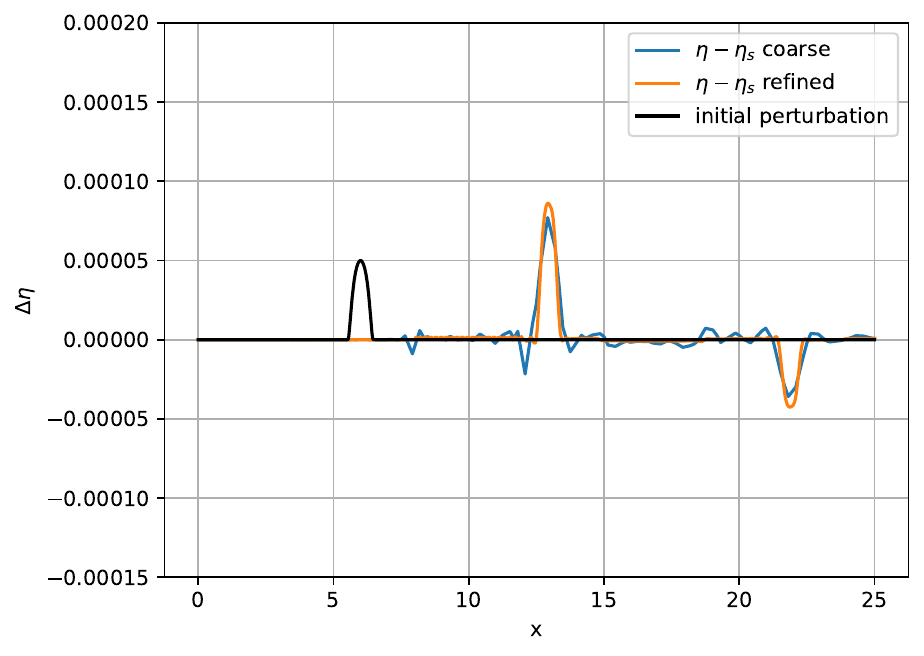}\caption{WB-$\GF$ with jg}
				\end{subfigure}
				\caption{Perturbation of non-smooth supercritical steady state with friction: comparison between different settings. Results referred to PGL4 with $30$ elements for the coarse mesh and $128$ elements for the refined mesh. A different scale has been used for the reference non-WB setting}\label{fig:supwb_fric}
			\end{figure}

			\item[•] \textbf{Subcritical flow}\\
			Also in this case, we focus first on the steady state obtained with different settings, reported in Figure \ref{fig:sub_fric_steady}. As confirmed by the numerical results, in this context we have a general increase in the velocity of the flow from left to right with consequent decrease of the water height.
			The numerical results confirm the consistency of all the settings but, again, there is remarkable difference between jr and jg and the other stabilizations in capturing the constant momentum. The amplitude of the spurious oscillations due to the non-smooth bathymetry, around $10^{-3}\sim 10^{-4}$ for all the other settings, jump down to $10^{-7}$ with jr and to machine precision with jg.

			\begin{figure}
				\centering
				\begin{subfigure}{0.31\textwidth}
					\includegraphics[width=\textwidth]{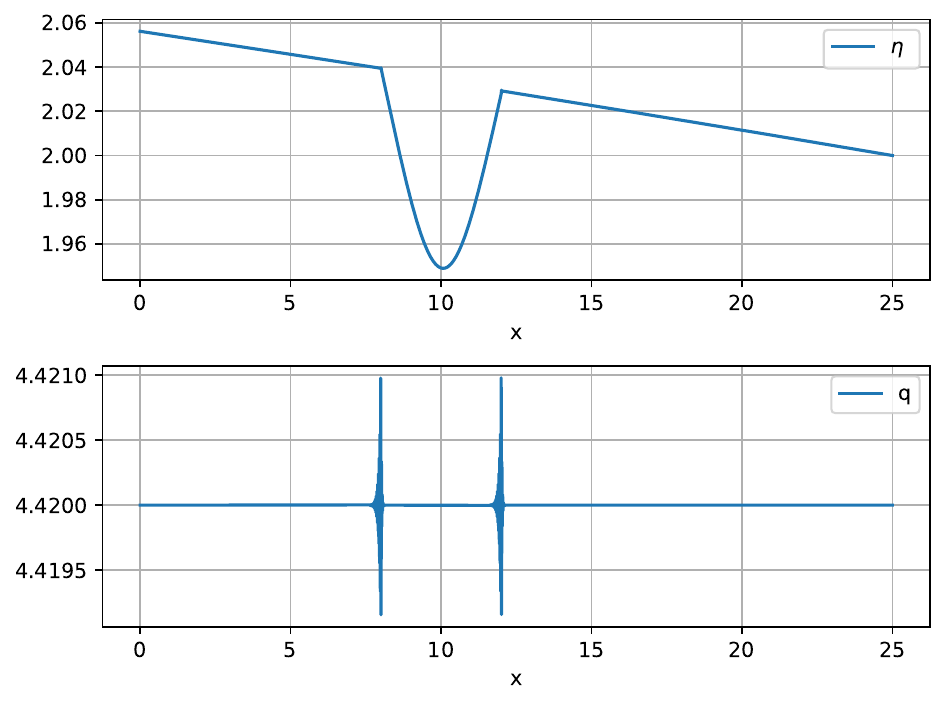}\caption{Reference non-WB}
				\end{subfigure}
				\begin{subfigure}{0.31\textwidth}
					\includegraphics[width=\textwidth]{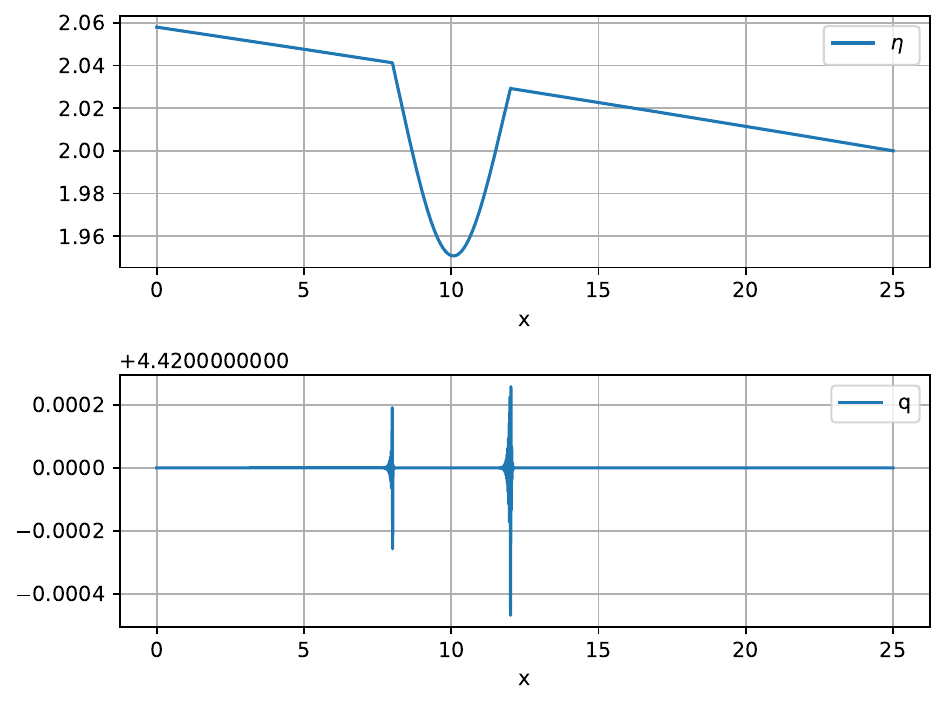}\caption{WB-$\HS$ with jc}
				\end{subfigure}
				\begin{subfigure}{0.31\textwidth}
					\includegraphics[width=\textwidth]{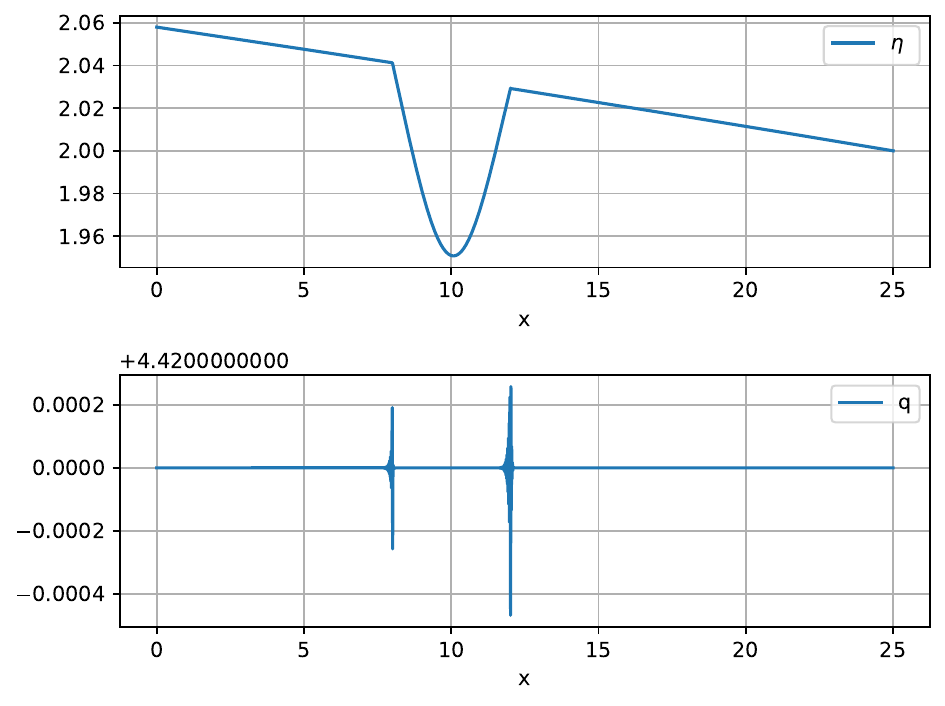}\caption{WB-$\GF$ with jc}
				\end{subfigure}\\
				\begin{subfigure}{0.31\textwidth}
					\includegraphics[width=\textwidth]{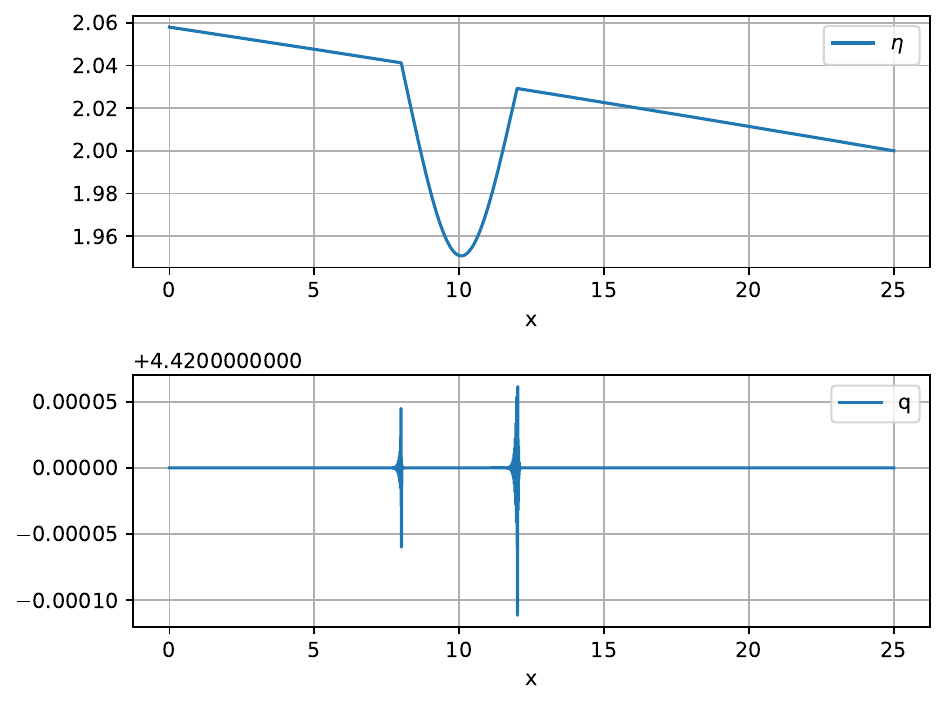}\caption{WB-$\HS$ with jt}
				\end{subfigure}
				\begin{subfigure}{0.31\textwidth}
					\includegraphics[width=\textwidth]{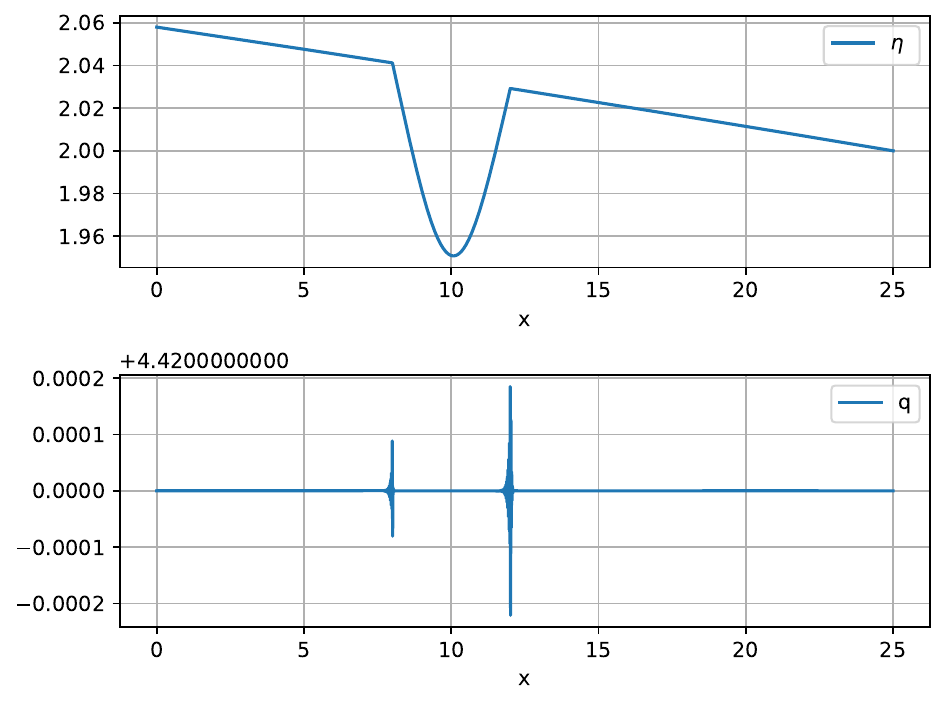}\caption{WB-$\HS$ with je}
				\end{subfigure}
				\begin{subfigure}{0.31\textwidth}
					\includegraphics[width=\textwidth]{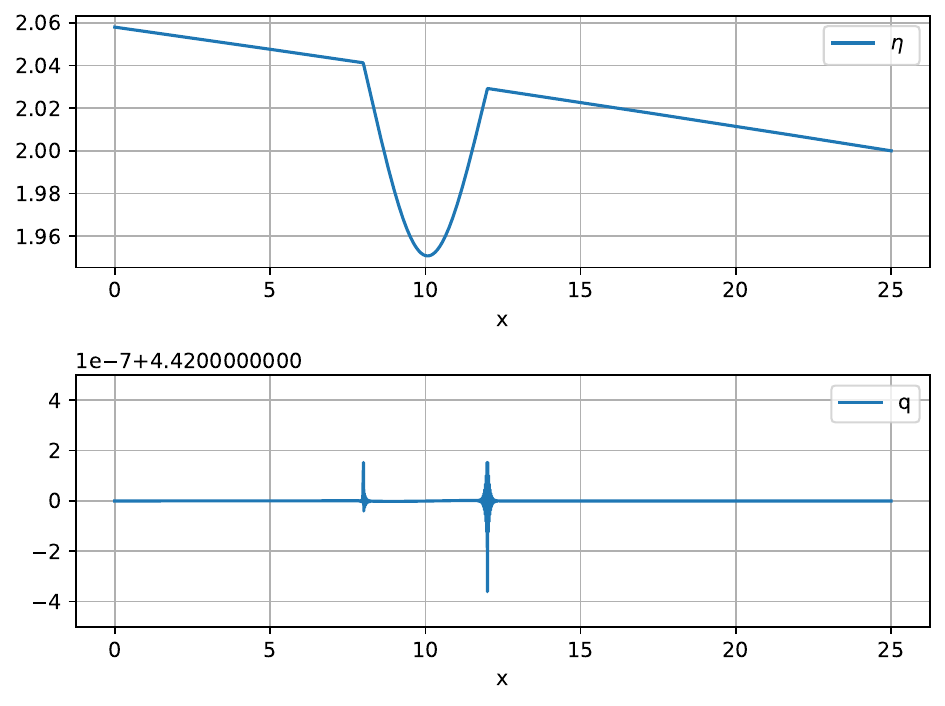}\caption{WB-$\HS$ with jr}
				\end{subfigure}\\
				\begin{subfigure}{0.31\textwidth}
					\includegraphics[width=\textwidth]{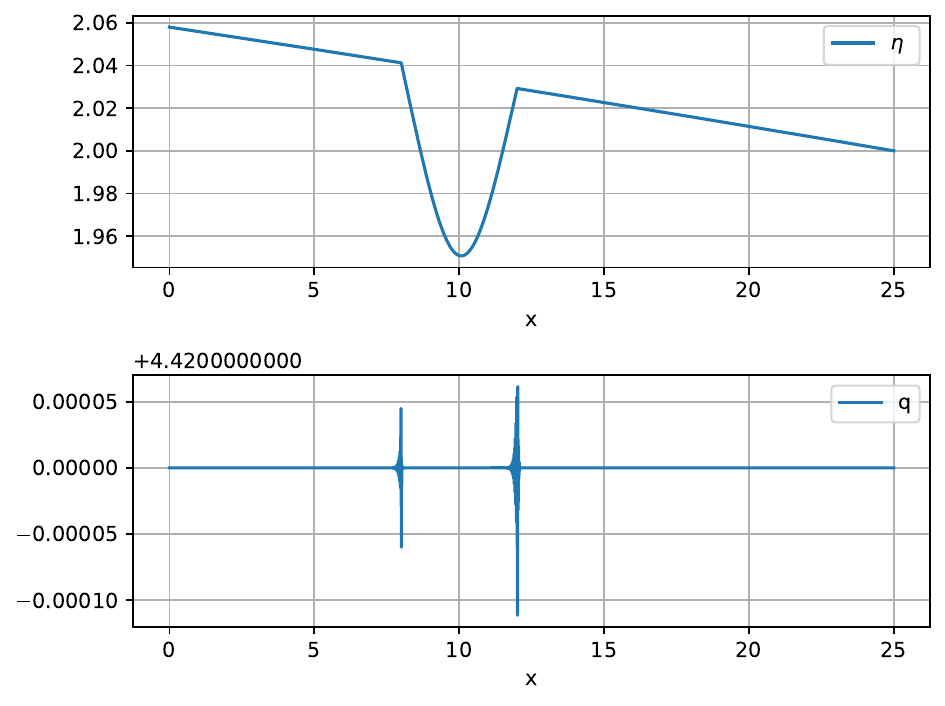}\caption{WB-$\GF$ with jt}
				\end{subfigure}
				\begin{subfigure}{0.31\textwidth}
					\includegraphics[width=\textwidth]{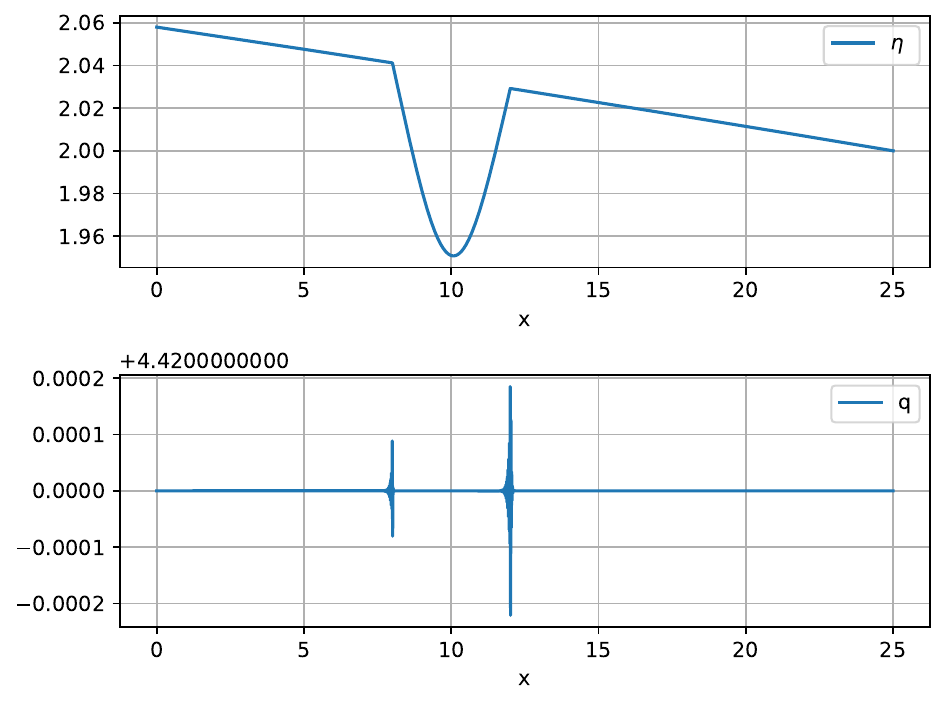}\caption{WB-$\GF$ with je}
				\end{subfigure}
				\begin{subfigure}{0.31\textwidth}
					\includegraphics[width=\textwidth]{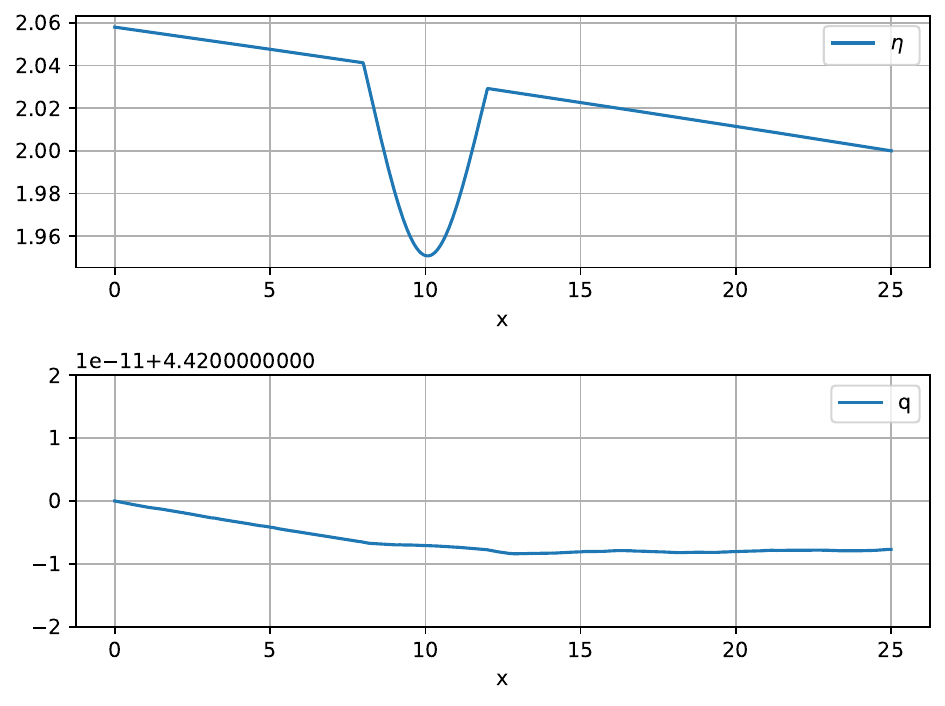}\caption{WB-$\GF$ with jg}
				\end{subfigure}
				\caption{Non-smooth subcritical steady state with friction: numerical steady state obtained with different settings. Results referred to P1 with $2048$ elements. Different scales have been used for $q$}\label{fig:sub_fric_steady}
			\end{figure}

			No fundamental differences concerning the perturbation analysis have been registered with respect to the subcritical test without friction.
			The related numerical results are displayed in Figure~\ref{fig:subwb_fric} and they further confirm the advantages in the adoption of jr and jg.
			
			
			\begin{figure}
				\centering
				\begin{subfigure}{0.31\textwidth}
					\includegraphics[width=\textwidth]{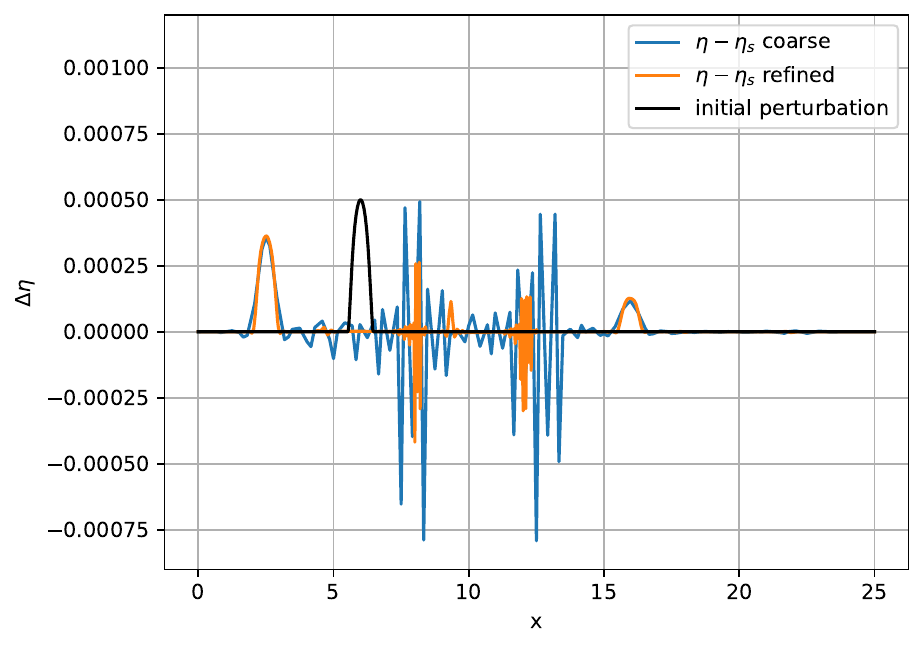}\caption{WB-$\HS$ with jc}
				\end{subfigure}
				\begin{subfigure}{0.31\textwidth}
					\includegraphics[width=\textwidth]{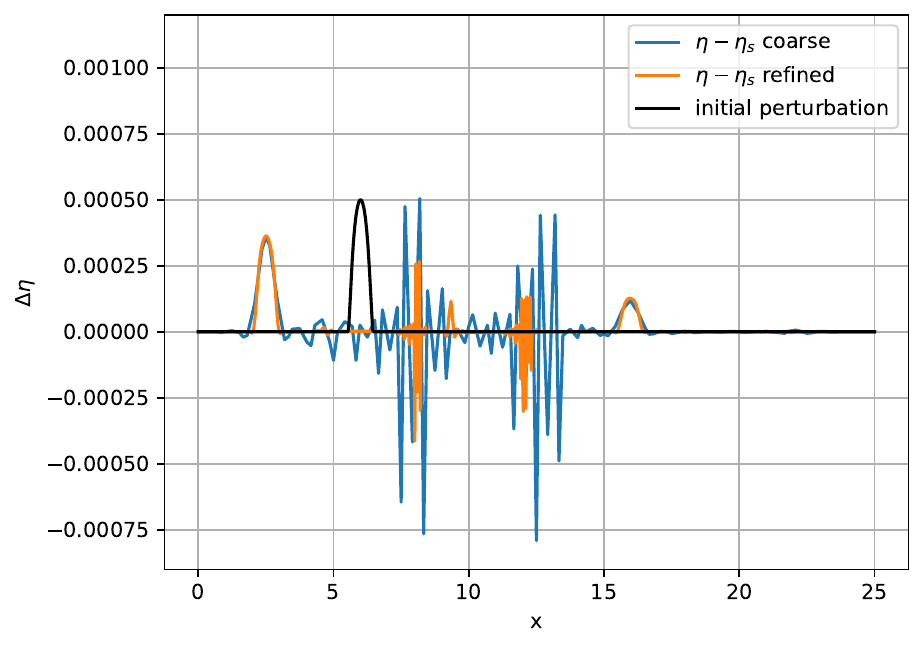}\caption{WB-$\GF$ with jc}
				\end{subfigure}\\
				\begin{subfigure}{0.31\textwidth}
					\includegraphics[width=\textwidth]{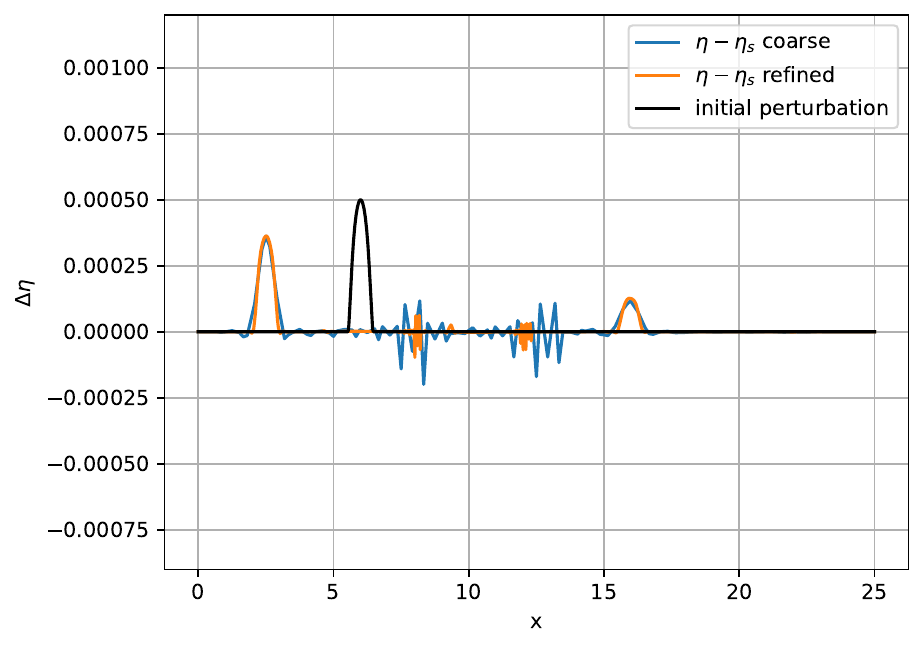}\caption{WB-$\HS$ with jt}
				\end{subfigure}
				\begin{subfigure}{0.31\textwidth}
					\includegraphics[width=\textwidth]{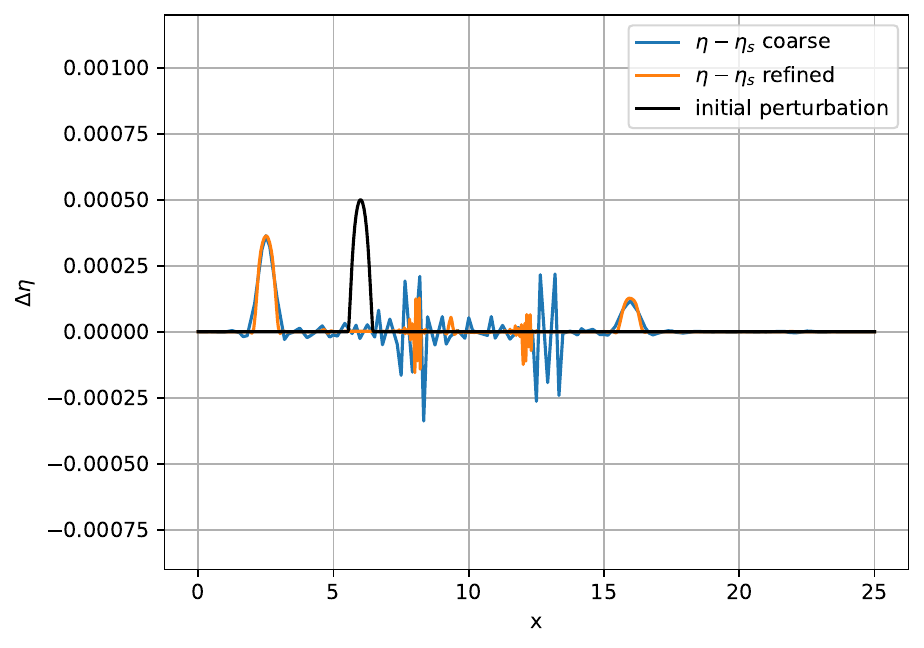}\caption{WB-$\HS$ with je}
				\end{subfigure}
				\begin{subfigure}{0.31\textwidth}
					\includegraphics[width=\textwidth]{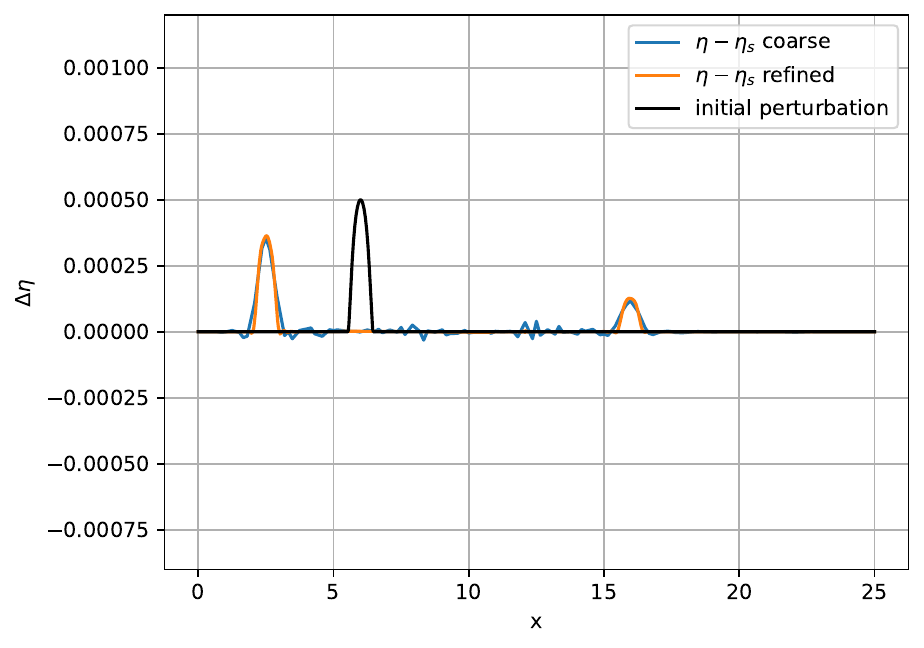}\caption{WB-$\HS$ with jr}
				\end{subfigure}\\
				\begin{subfigure}{0.31\textwidth}
					\includegraphics[width=\textwidth]{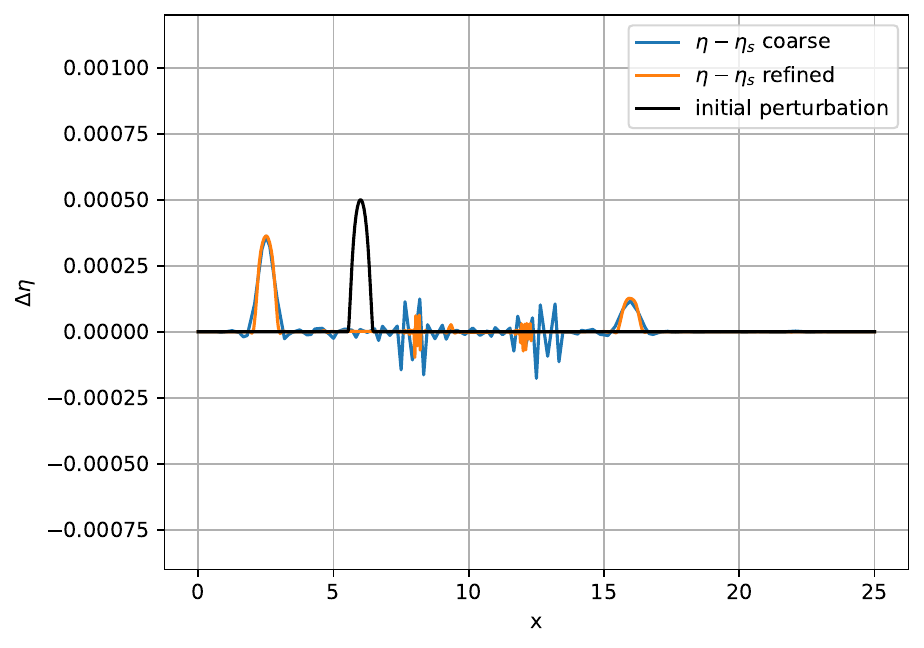}\caption{WB-$\GF$ with jt}
				\end{subfigure}
				\begin{subfigure}{0.31\textwidth}
					\includegraphics[width=\textwidth]{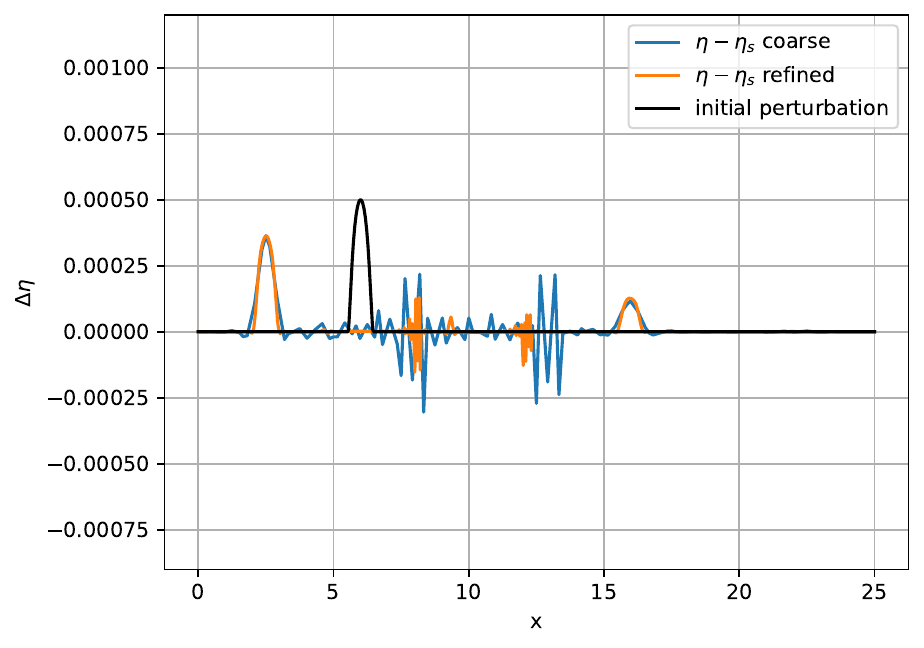}\caption{WB-$\GF$ with je}
				\end{subfigure}
				\begin{subfigure}{0.31\textwidth}
					\includegraphics[width=\textwidth]{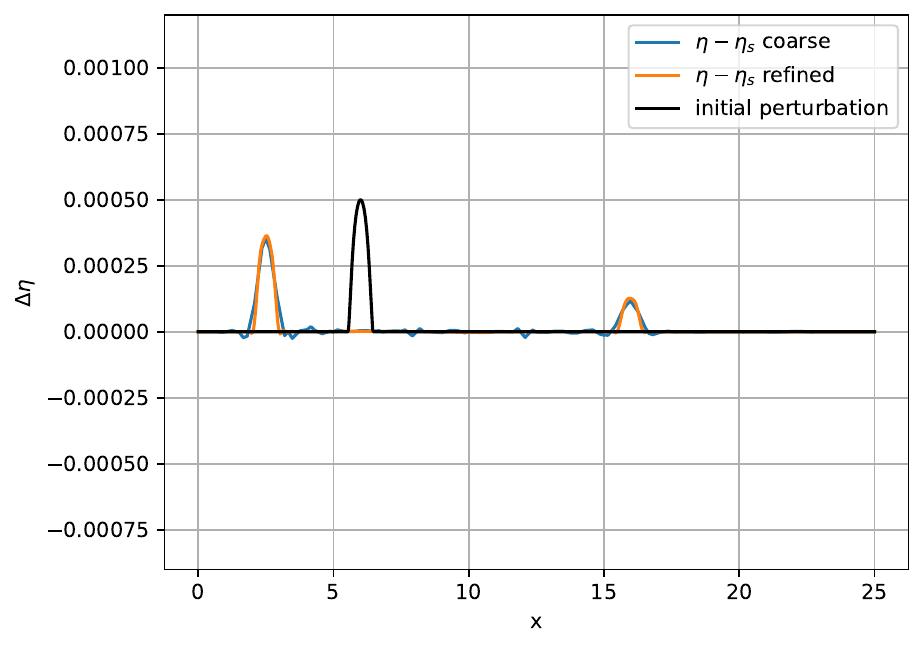}\caption{WB-$\GF$ with jg}
				\end{subfigure}
				\caption{Perturbation of non-smooth subcritical steady state with friction: comparison between the different stabilizations. Results referred to PGL4 with $30$ elements for the coarse mesh and $128$ elements for the refined mesh}\label{fig:subwb_fric}
			\end{figure}

			We close this section with a comparison, in Figure \ref{fig:subwb_compare_PGLn}, between basis functions of different degrees for the best performing settings.
			Like in the context of the frictionless supercritical case, the number of the elements in the different coarse meshes is selected in such a way to have a constant number of DoFs. The results are analogous: the quality of the results improves, as the degree increases, in terms of ability to capture the peaks. Further, amplitude and number of the spurious oscillations decrease. Under this point of view, we remark that the remaining spurious oscillations, whose amplitude is however very small, are due to the fact that here we do not adopt any limiting technique and, moreover, they disappear in the mesh refinement. 
			Indeed, an ``unfair'' comparison, with a constant number of elements, would give even better results.
			
			\begin{figure}
				\centering
				\begin{subfigure}{0.31\textwidth}
					\includegraphics[width=\textwidth]{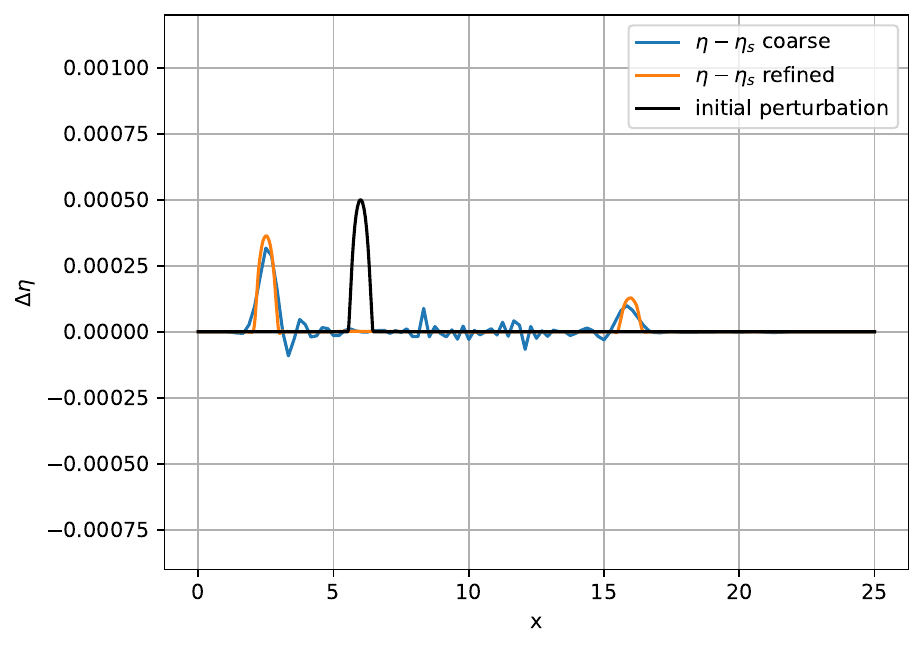}\caption{WB-$\HS$ with jr, PGL2}
				\end{subfigure}
				\begin{subfigure}{0.31\textwidth}
					\includegraphics[width=\textwidth]{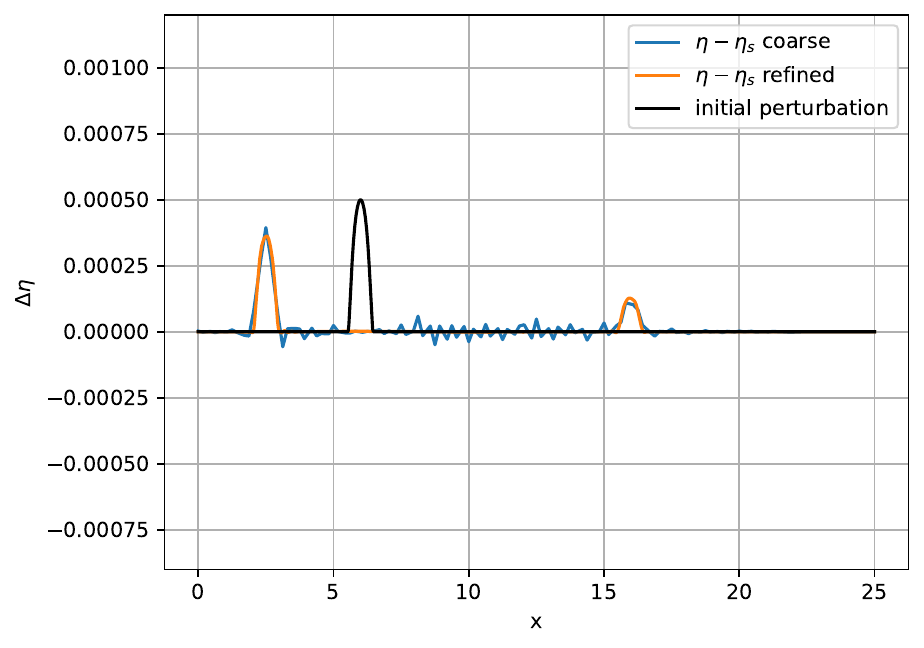}\caption{WB-$\HS$ with jr, PGL3}
				\end{subfigure}
				\begin{subfigure}{0.31\textwidth}
					\includegraphics[width=\textwidth]{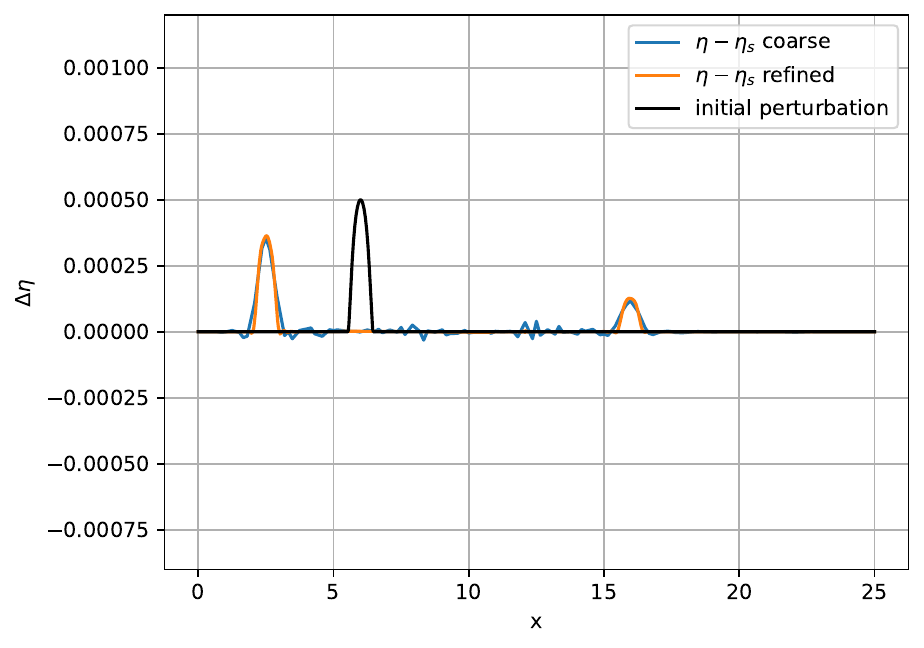}\caption{WB-$\HS$ with jr, PGL4}
				\end{subfigure}\\
				\begin{subfigure}{0.31\textwidth}
					\includegraphics[width=\textwidth]{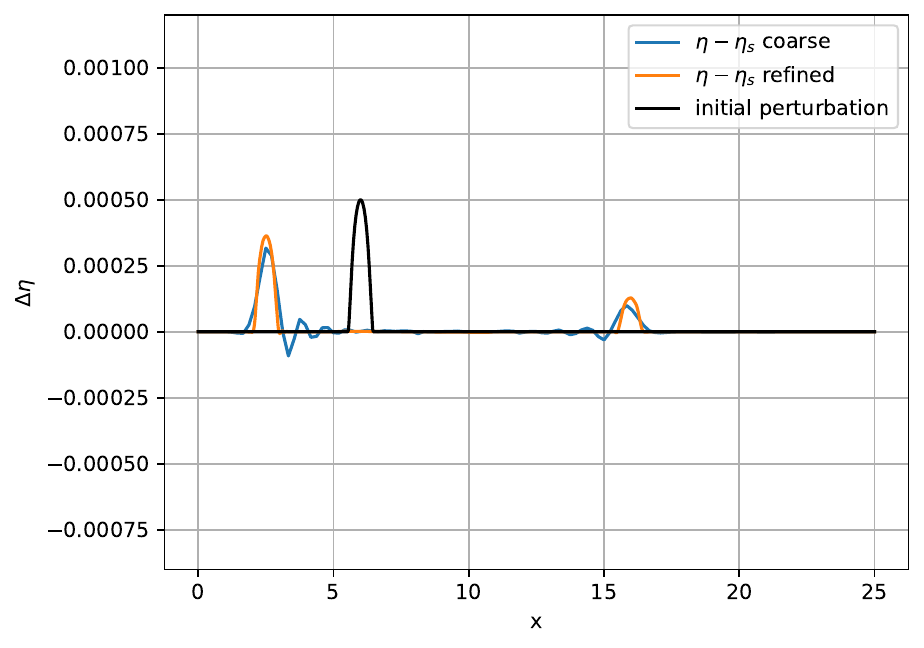}\caption{WB-$\GF$ with jg, PGL2}
				\end{subfigure}
				\begin{subfigure}{0.31\textwidth}
					\includegraphics[width=\textwidth]{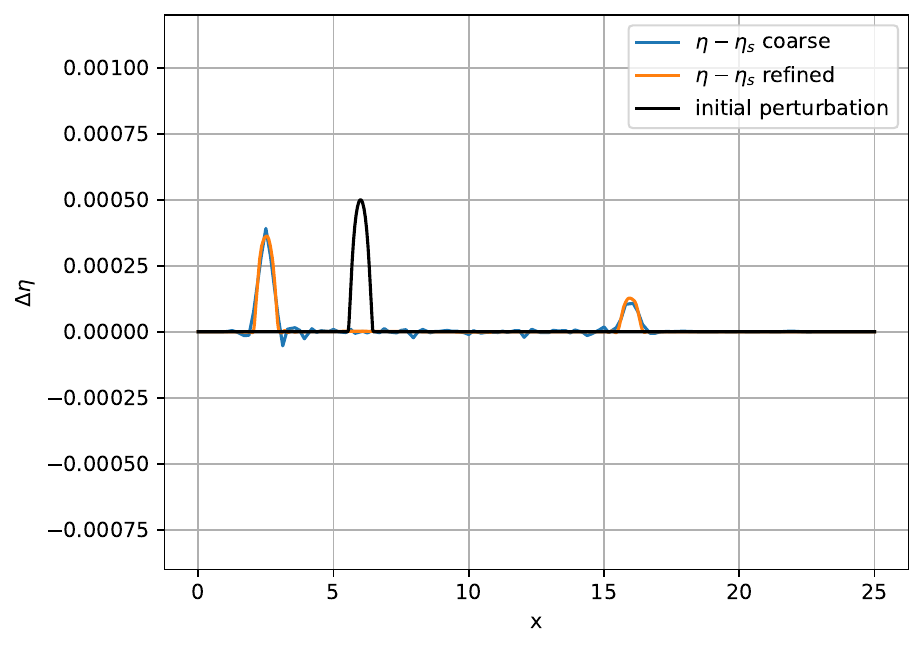}\caption{WB-$\GF$ with jg, PGL3}
				\end{subfigure}
				\begin{subfigure}{0.31\textwidth}
					\includegraphics[width=\textwidth]{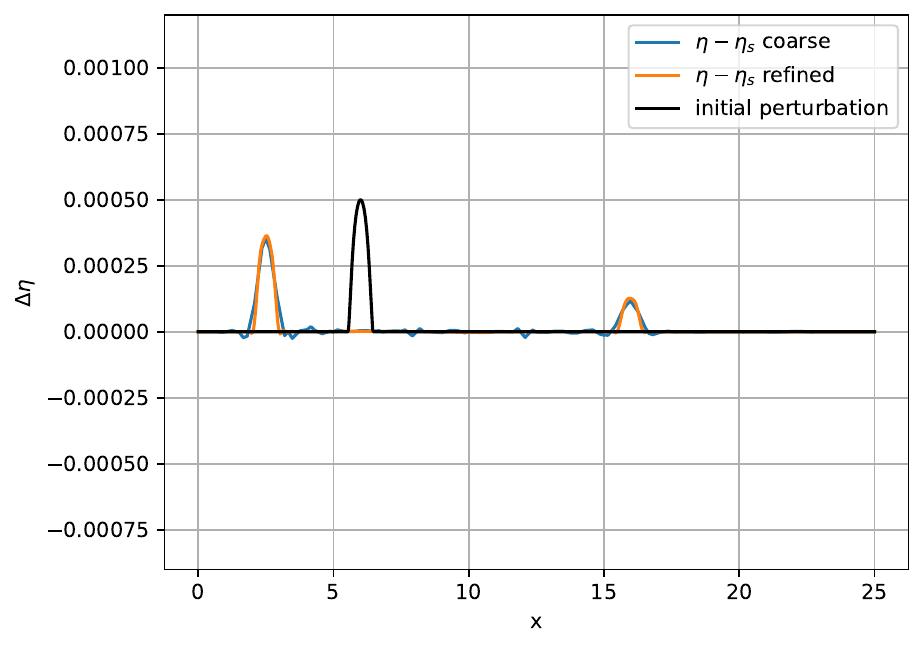}\caption{WB-$\GF$ with jg, PGL4}
				\end{subfigure}
				\caption{Perturbation of non-smooth subcritical steady state with friction: fair comparison between basis functions of different degree with jr and jg. Respectively $30$, $40$ and $60$ elements for PGL4, PGL3 and PGL2 for the coarse meshes and $128$, $256$ and $512$ elements for the refined ones}\label{fig:subwb_compare_PGLn}
			\end{figure}

		\end{itemize}
		
		\RIcolor{
			\subsection{Multidimensional tests}\label{sub:2d_tests}
			In this final section, we report the numerical results for the two-dimensional frictionless SW equations on unstructured triangular meshes.
			The main purpose of this section is to show the possibility to apply some of the presented notions to an unstructured multidimensional framework.

			Let us briefly recall the needed elements characterized to this setting.
			The governing equations read
			\begin{equation}
				\label{eq:sys_2d}
				\frac{\partial}{\partial t}\uvec{u}+div_{\uvec{x}} \boldsymbol{F}(\uvec{u})=\boldsymbol{S}(\uvec{x},\uvec{u}), \quad (\uvec{x},t) \in \Omega\times \mathbb{R}^+_0.
			\end{equation}
			In such a case, $\Omega\subseteq \mathbb{R}^2$ and conserved variables, flux and source are defined as
			\begin{equation}\label{eq:SW}
				\uvec{u}:=\begin{pmatrix}
					H\\
					\uvec{q}
				\end{pmatrix},\quad\uvec{F}(\uvec{u}):=\begin{pmatrix}
					\uvec{q}\\
					\frac{1}{H}\uvec{q}\otimes \uvec{q}+g\frac{H^2}{2}\mathbb{I}
				\end{pmatrix}, \quad \boldsymbol{S}(\uvec{x},\uvec{u}):=-gH\begin{pmatrix}
					0\\
					\nabla_{\uvec{x}} B(\uvec{x})
				\end{pmatrix}.
			\end{equation}
			All the ingredients are defined as in the \RIcolor{one-dimensional} case up to the fact that $\uvec{q}:=H\uvec{v}\in \mathbb{R}^2$ is a vector, with $\uvec{v}:=(v_1,v_2)^T$. 
			
			As already specified in Remark~\ref{rmk:generalization_to_mulitd}, the CG/RD framework described in Section~\ref{chapRD} can be naturally generalized to a multidimensional unstructured framework, and the corresponding semidiscretization is given by
			\begin{align}
				\int_\Omega \biggl( \frac{\partial}{\partial t}\uvec{u}_h + \biggl[div_{\uvec{x}} \boldsymbol{F}(\uvec{u}_h) -\boldsymbol{S}(\uvec{x},\uvec{u}_h) \biggr]_h \biggr) \varphi_i(x)dx +\uvec{ST}_i(\uvec{u}_h) =\uvec{0},  \quad \forall i=1,\dots,I.
				\label{eq:CG_2d}
			\end{align}
			The extension of the reference non-WB framework, presented in Section~\ref{sec:ref_nonwb}, to this context is straightforward.
			
			Concerning the discretization of the spatial part of the PDE, we consider a simple interpolation of $\boldsymbol{F}$ and $\boldsymbol{S}$ in the same polynomial space used for the CG discretization
			\begin{align}
				\biggl[div_{\uvec{x}}   \boldsymbol{F}(\uvec{u}_h) -\boldsymbol{S}(\uvec{x},& \uvec{u}_h) \biggr]_h :=div_{\uvec{x}} \boldsymbol{F}_h-\boldsymbol{S}_h, \label{eq:non_wb_space_discretization1_2d}\\
				\boldsymbol{F}_h:=\sum_{i=1}^I \boldsymbol{F}_i \varphi_i(\uvec{x}),& \quad  \boldsymbol{S}_h:=\sum_{i=1}^I \boldsymbol{S}_i \varphi_i(\uvec{x}),
				\label{eq:non_wb_space_discretization1_2d}
			\end{align}
			while, concerning the stabilization, we consider the multidimensional version of jc~\eqref{eq:CIP_multiD}.
			
			Here, we consider a single WB alternative, given by WB-HS, presented in Section~\ref{sec:WBHS}, coupled with jt.
			With suitable adaptation of the definitions to the multidimensional case, the space discretization WB-HS is defined as
			\begin{align}
				\biggl[div_{\uvec{x}}   \boldsymbol{F} -\boldsymbol{S} \biggr]_h:=div_{\uvec{x}}  \boldsymbol{F}^V_h+\begin{pmatrix}
					0\\
					gH_h\nabla_{\uvec{x}}(H_h+B_h)
				\end{pmatrix},
				\label{eq:WBHS_2d}
			\end{align}
			where the velocity part of the source $\boldsymbol{S}^V_h$ is neglected as we focus here on the frictionless case.
			Note that Remark~\ref{rmk:conservation_WBHS} on the conservative character of such formulation holds also in the multidimensional case.
			The jump stabilization jt is instead obtained by replacing the first component of the conserved variables in~\eqref{eq:CIP_multiD} by $\eta_h$.
			
			In the following tests, we consider the stabilization on the first derivative only ($R:=1$).
			Furthermore, concerning the definition of $\alpha_{f,1}:=\delta_1 \bar{\rho}_f h_f^{2}$, we assume $h_f$ equal to the length of $f$ and $\delta_1:=0.01.$ 
			We will consider B2 basis functions and the adopted time discretization is again given by the bDeCu method presented in Section~\ref{chapDECRD} with CFL$:=0.1.$
			
			\subsubsection{High order accuracy}
			We start by assessing the high order accuracy via a smooth test involving an unsteady compactly supported $C^{6}$ vortex~\cite{micalizzi2023new,ricchiuto2021analytical} in the computational domain $\Omega:=(0,3)\times(0,3)$ without bathymetry $(B\equiv 0)$.

			In particular, the analytical solution is given as a function of the radial distance $r(\uvec{x},t):=\norm{\uvec{x}-\uvec{x}_c(t)}_2$ from the center of the vortex $\uvec{x}_c(t):=\uvec{x}_0 + t\cdot(1,1)^T$, with $\uvec{x}_0=(1,1)^T$, and it reads
			\begin{equation}
				\uvec{u}(\uvec{x},t) := \uvec{u}^\infty + \begin{cases}
					\uvec{u}_{r_0}(r), & \text{ if }r <r_0,\\
					0, & \text{ otherwise,}
				\end{cases}
			\end{equation}
			where $\uvec{u}^{\infty} :=(1,1,1)^T$, and 
			\begin{align}
				&\uvec{u}_{r_0}(r) := \begin{pmatrix}
					\frac{1}{g}\left(\frac{\Gamma}{\omega}\right)^2 \left(\lambda(\omega r)-\lambda(\pi)\right)\\
					\Gamma\left(1+\cos(\omega r)\right)^2(x_2-x_{c,2})\\
					- \Gamma\left(1+\cos(\omega r)\right)^2(x_1-x_{c,1}) 
				\end{pmatrix},\; \quad \Gamma := \frac{12\pi\sqrt{g\Delta H}}{r_0\sqrt{315\pi^2 - 2048}}.
			\end{align}
			We have $\omega := \frac{\pi}{r_0}$ and
			\begin{equation}
				\begin{split}
					\lambda(s)& := \frac{20}{3}\cos(s) + \frac{27}{16}\cos(s)^2 + \frac{4}{9}\cos(s)^3+ \frac{1}{16}\cos(s)^4 +\frac{20}{3}s\sin(s) \\
					& + \frac{35}{16}s^2 + \frac{27}{8}s\cos(s)\sin(s) +\frac{4}{3}s\cos(s)^2\sin(s) + \frac{1}{4}s\cos(s)^3\sin(s).
					\label{eq:supp_vort}	
				\end{split}
			\end{equation}
			We assume $r_0:=1$ and $\Delta H := 0.1$. A final time $T_f:=1$ and inflow-outflow boundary conditions are considered.
			
			The results of the convergence analysis are reported in Table~\ref{tab:conv_analysis_vortex_2d}.
			The expected third order rate of convergence, with respect to the mesh parameter $h$, is obtained for both the considered discretizations and the errors are very similar.
			Note that, as no bathymetry is present in this test, jt reduces to jc, hence, the only difference is in the adopted discretization of the spatial part of the PDE.

			\begin{table}[htbp]
				\centering
				\caption{Two-dimensional unsteady vortex: Convergence analysis with B2}
				\label{tab:conv_analysis_vortex_2d}
				\begin{subtable}{\linewidth}
					\centering
					\caption{Reference non-WB}
					\begin{tabular}{ccccccc}
						\toprule
						$h$ & $L^1 \, \text{error} \, H$ & $\text{Order} \, H$ & $L^1 \, \text{error} \, Hv_1$ & $\text{Order} \, Hv_1$ & $L^1 \, \text{error} \, Hv_2$ & $\text{Order} \, Hv_2$ \\
						\midrule
						0.4 & 1.997e-02 & - & 2.647e-01 & - & 2.994e-01 & - \\
						0.2 & 3.712e-03 & 2.428 & 4.437e-02 & 2.577 & 4.622e-02 & 2.695 \\
						0.1 & 5.363e-04 & 2.791 & 6.467e-03 & 2.778 & 6.523e-03 & 2.825 \\
						0.05 & 8.167e-05 & 2.715 & 8.451e-04 & 2.936 & 8.717e-04 & 2.904 \\
						\bottomrule
					\end{tabular}
				\end{subtable}
				
				\bigskip
				
				\begin{subtable}{\linewidth}
					\centering
					\caption{WB-HS with jt}
					\begin{tabular}{ccccccc}
						\toprule
						$h$ & $L^1 \, \text{error} \, H$ & $\text{Order} \, H$ & $L^1 \, \text{error} \, Hv_1$ & $\text{Order} \, Hv_1$ & $L1 \, \text{error} \, Hv_2$ & $\text{Order} \, Hv_2$ \\
						\midrule
						0.4 & 2.000e-02 & - & 2.648e-01 & - & 2.994e-01 & - \\
						0.2 & 3.713e-03 & 2.429 & 4.437e-02 & 2.577 & 4.622e-02 & 2.696 \\
						0.1 & 5.365e-04 & 2.791 & 6.467e-03 & 2.778 & 6.523e-03 & 2.825 \\
						0.05 & 8.168e-05 & 2.716 & 8.451e-04 & 2.936 & 8.717e-04 & 2.904 \\
						\bottomrule
					\end{tabular}
				\end{subtable}
				
			\end{table}
			
			\subsubsection{Exact well-balancing for lake at rest}
			Let us consider, on the computational domain $\Omega:=(0,3)\times(0,3)$, the following bathymetry, characterized by a smooth bump, and defined in terms of the radial distance $r(\uvec{x}):=\norm{\uvec{x}-\uvec{x}_c}_2$ from the center of the bump $\uvec{x}_c=(1.5,1.5)^T$
			\begin{equation}
				B(\uvec{x}):=\begin{cases}
					0.2\exp{\left(1-\frac{1}{1-\left(\frac{r}{r_0}\right)^2}\right)}, & \text{if}~r<r_0,\\
					0, & \text{otherwise}.
				\end{cases}
				\label{eq:smooth_bathymetry_2d}
			\end{equation}
			%
			The lake at rest steady state under investigation is given by 
			\begin{equation}
				\eta=H+B\equiv \overline{\eta}, \quad \uvec{v}\equiv \uvec{0},
			\end{equation}
			with $\overline{\eta}:=1$. We set the final time $T_f:=0.2$ and assume inflow-outflow boundary conditions.
			
			Rather than a single test, we performed multiple simulations on several meshes with different mesh parameter.
			This allowed to verify not only the WB character of the proposed discretization, WB-HS coupled with jt, but also the expected third order convergence of the reference non-WB framework, assessing thus the correct implementation of the bathymetry.
			The results are reported in Table~\ref{tab:conv_analysis_latr_2d}. As expected, WB-HS with jt is able to exactly capture the steady state up to machine precision. On the other hand, the reference non-WB framework produces errors which scale with the expected order of accuracy when the mesh is refined.

			\begin{table}[htbp]
				\centering
				\caption{Two-dimensional lake at rest: Convergence analysis with B2}
				\label{tab:conv_analysis_latr_2d}
				\begin{subtable}{\linewidth}
					\centering
					\caption{Reference non-WB}
					\begin{tabular}{ccccccc}
						\toprule
						$h$ & $L^1 \, \text{error} \, H$ & $\text{Order} \, H$ & $L^1 \, \text{error} \, Hv_1$ & $\text{Order} \, Hv_1$ & $L^1 \, \text{error} \, Hv_2$ & $\text{Order} \, Hv_2$ \\
						\midrule
						0.2   & 1.171e-02 & - & 7.216e-02 & - & 6.432e-02 & - \\
						0.1   & 2.021e-03 & 2.535 & 1.299e-02 & 2.474 & 1.374e-02 & 2.227 \\
						0.05  & 2.688e-04 & 2.910 & 1.958e-03 & 2.730 & 2.143e-03 & 2.681 \\
						0.025 & 2.087e-05 & 3.687 & 2.126e-04 & 3.203 & 2.111e-04 & 3.344 \\
						\bottomrule
					\end{tabular}
				\end{subtable}

				\bigskip

				\begin{subtable}{\linewidth}
					\centering
					\caption{WB-HS with jt}
					\begin{tabular}{cccc}
						\toprule
						$h$ & $L^1 \, \text{error} \, H$ & $L^1 \, \text{error} \, Hv_1$ & $L^1 \, \text{error} \, Hv_2$ \\
						\midrule
						0.2   & 2.789e-17 & 4.297e-15 & 4.228e-15 \\
						0.1   & 6.985e-17 & 6.955e-15 & 7.735e-15 \\
						0.05  & 5.347e-16 & 1.302e-14 & 1.420e-14 \\
						0.025 & 8.808e-16 & 1.941e-14 & 2.150e-14 \\
						\bottomrule
					\end{tabular}
				\end{subtable}
				
			\end{table}

			\subsubsection{Evolution of small perturbations of lake at rest}
			Finally, we perform a perturbation analysis to show the ability of WB-HS coupled with jt to capture the correct dynamics of little perturbations of lake at rest, avoiding spurious oscillations, even in an unstructured multidimensional framework.

			We consider the steady state of the previous section with the following small perturbation
			\begin{equation}
				\eta(\uvec{x}):=\begin{cases}
					\eta_{s}+A, &\text{if}~\norm{\uvec{x}-\uvec{x}_p}_2<r_p,\\
					\eta_{s}, & \text{otherwise},
				\end{cases}
				\label{eq:perturbation_lake_at_rest_2d}
			\end{equation}
			with $\eta_{s}:=\overline{\eta}\equiv 1$, $A:= 10^{-3}$, $\uvec{x}_p:=(1.2,1.2)^T$ and $r_p:=0.1.$
			We consider the same final time $T_f$ and boundary conditions as in the unperturbed test.
			
			The evolution of the perturbation obtained, on a coarse tessellation with mesh parameter $h=0.1$, through the reference non-WB setting and WB-HS coupled with jt are reported in Figure~\ref{fig:perturbation_analysis_2d}.
			As one can clearly see, this is very well captured by the WB method, while, the reference non-WB scheme suffers from spurious oscillations, which completely overwhelm the perturbation.

			\begin{figure}
				\centering
				\begin{subfigure}{0.45\textwidth}
					\includegraphics[trim= 235 0 0 105, clip,width=\textwidth]{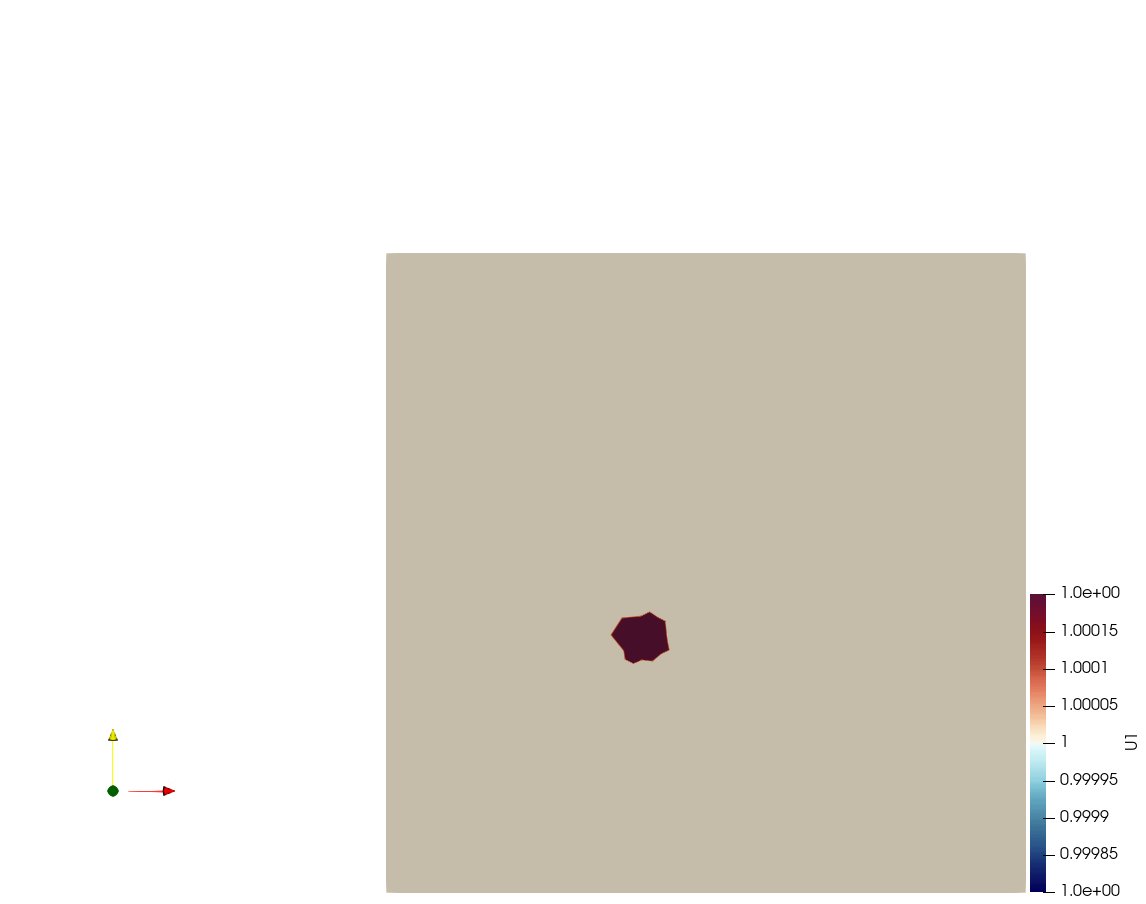}\caption{Initial condition}
				\end{subfigure}\\
				\begin{subfigure}{0.45\textwidth}
					\includegraphics[trim= 235 0 0 105, clip,width=\textwidth]{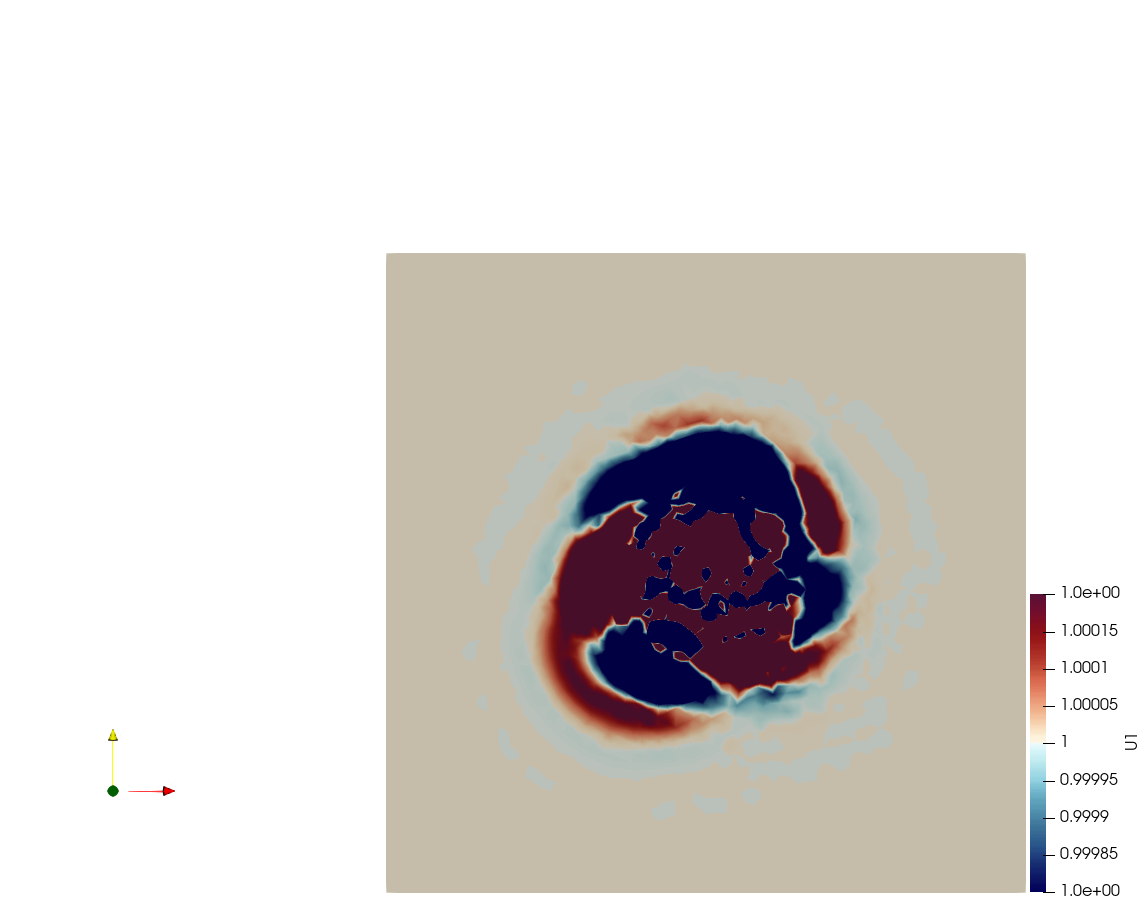}\caption{$t=0.1$}
				\end{subfigure}\quad
				\begin{subfigure}{0.45\textwidth}
					\includegraphics[trim= 235 0 0 105, clip,width=\textwidth]{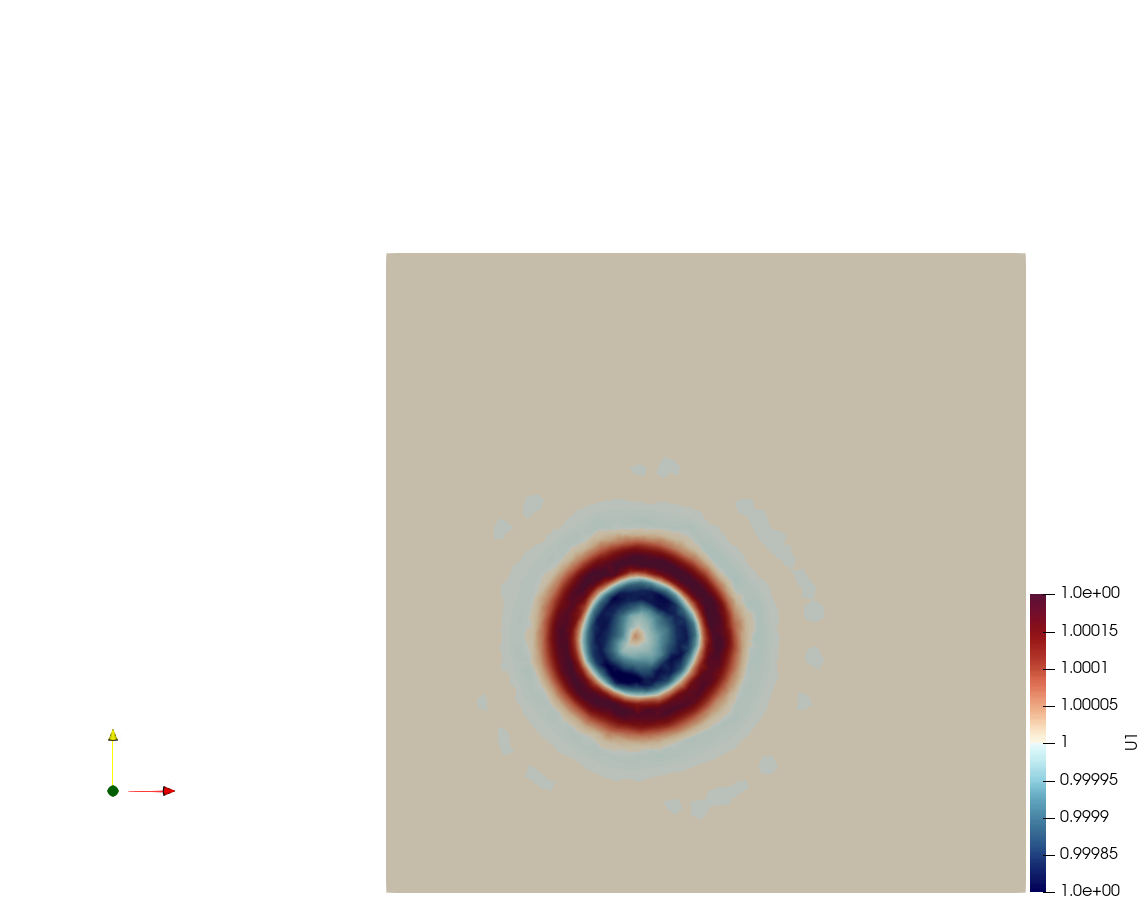}\caption{$t=0.1$}
				\end{subfigure}\\
				\begin{subfigure}{0.45\textwidth}
					\includegraphics[trim= 235 0 0 105, clip,width=\textwidth]{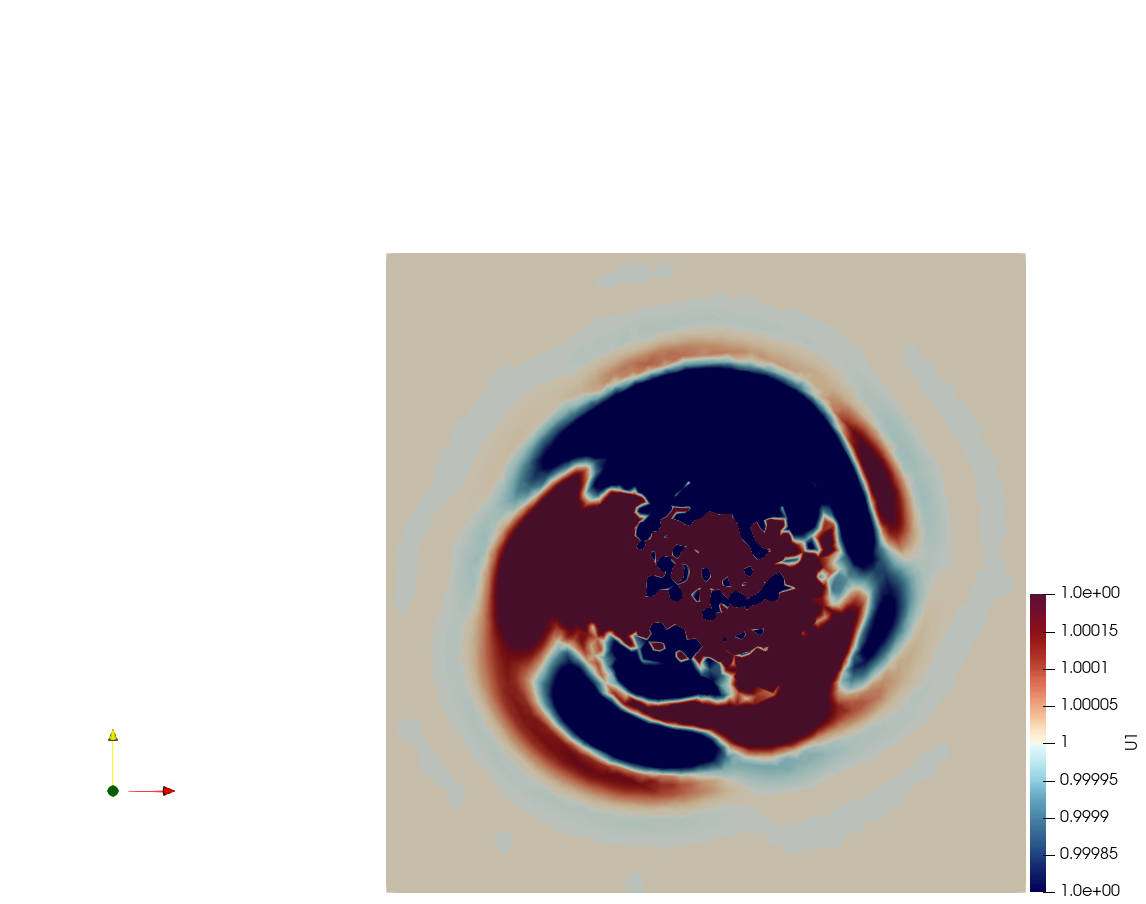}\caption{$t=0.2$}
				\end{subfigure}\quad
				\begin{subfigure}{0.45\textwidth}
					\includegraphics[trim= 235 0 0 105, clip,width=\textwidth]{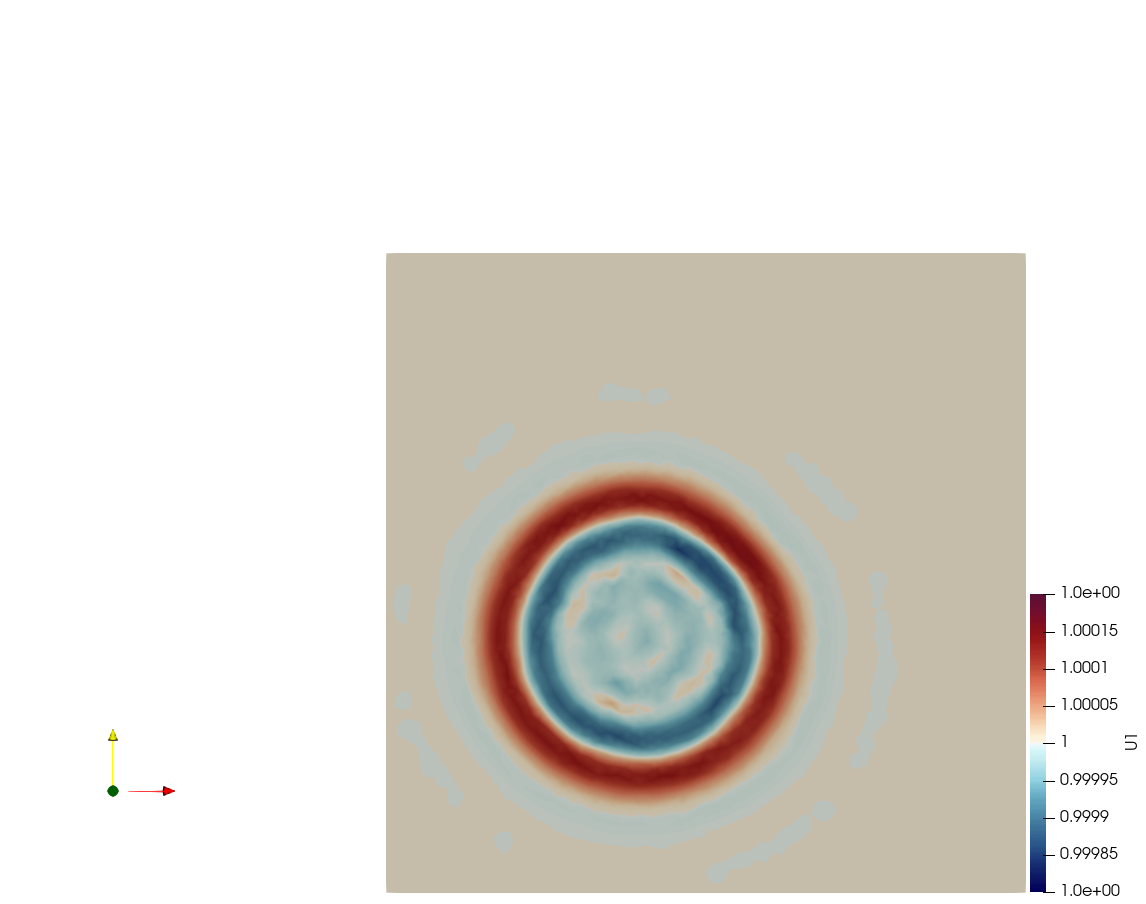}\caption{$t=0.2$}
				\end{subfigure}\caption{Perturbation of two-dimensional lake at rest: comparison between the reference non-WB setting (on the left) and WB-HS coupled with jt (on the right). Variable $\eta$ at different times}\label{fig:perturbation_analysis_2d}
			\end{figure}

		}
		
		\section{Conclusions and further developments}\label{chapConc}
		In this work, we have analyzed the performance of two WB space discretizations and four novel WB CIP stabilizations. All the elements are specifically designed to exactly preserve the lake at rest steady state. In particular, two stabilizations, jr and jg, address the problem of well-balancing toward general steady states, being they based on a discretization of the steady equilibrium. 
		
		The numerical results confirm the exact well-balancing with respect to the lake at rest and the arbitrary high order accuracy. Moreover, in \RIcolor{one-dimensional} tests, jr and jg have been shown able to better handle other general steady states in terms of superconvergences on smooth tests, with a strong propension in retrieving the constant momentum up to machine precision, and ability to capture the evolution of small perturbations of such steady states.
		Further numerical investigations have confirmed the possibility to apply some of the proposed elements to an unstructured multidimensional setting.
		
		Possible future developments, planned for future works, are a deeper investigation of the WB properties of the introduced space discretizations and jump stabilizations towards general steady states in a multidimensional framework, and the extension of some of the presented elements to the Euler equations with gravity.


\section*{Acknowledgements}
	L. Micalizzi has been funded by the SNF grant 200020\_204917 and by the LeRoy B. Martin, Jr. Distinguished Professorship Foundation. 
	R. Abgrall has been funded by the SNF grant 200020\_204917. 
	M. Ricchiuto is a member of the CARDAMOM team at INRIA University of Bordeaux.

\appendix
\section{Proof of Proposition \ref{prop:CIP_RD}}\label{app:CIP_RD}

\begin{proof}
	The first point is a straightforward consequence of the assumptions made on the basis functions. 
	In particular, we can write
	\begin{align}
		\begin{split}
			\sum_{\uvec{x}_i \in K}\uvec{ST}_i^K(\uvec{u}_h)&=\sum_{\uvec{x}_i \in K}\sum_{\substack{f\subset\partial{K}\\f \in \faces}}\sum_{r=1}^R \alpha_{f,r} \int_f  \nabla^r_{\uvec{\nu}_f}  \varphi_i\vert_K  \Big\llbracket \nabla^r_{\uvec{\nu}_f} \uvec{u}_h \Big\rrbracket   d\sigma(\uvec{x})\\
			&=\sum_{\substack{f\subset\partial{K}\\f \in \faces}}\sum_{r=1}^R \alpha_{f,r} \int_f  \nabla^r_{\uvec{\nu}_f}\left(\sum_{\uvec{x}_i \in K}  \varphi_i \vert_K\right)  \Big\llbracket \nabla^r_{\uvec{\nu}_f} \uvec{u}_h \Big\rrbracket   d\sigma(\uvec{x}),
		\end{split}
		\label{eq:sum_to_0}
	\end{align}
	which is indeed equal to zero because of \eqref{eq:local_normalization_bases}.

	Now, let us deal with the second point. In order to light the notation, without loss of generality, we will focus on the $r$-th derivative only, dropping the sum over the orders, and we will neglect the factor $\alpha_{f,r}$. Then, what we want to show is the equivalence
	\begin{equation}
		\sum_{f\in \faces} \int_f \Big\llbracket \nabla^r_{\uvec{\nu}_f} \varphi_i \Big\rrbracket  \Big\llbracket \nabla^r_{\uvec{\nu}_f} \uvec{u}_h \Big\rrbracket  d\sigma(\uvec{x})=\sum_{K \in K_i}\sum_{\substack{f\subset\partial{K}\\f \in \faces}}  \int_f  \nabla^r_{\uvec{\nu}_f} \varphi_i\vert_K  \Big\llbracket \nabla^r_{\uvec{\nu}_f} \uvec{u}_h \Big\rrbracket   d\sigma(\uvec{x}).   
		\label{eq:CIP_thesis}
	\end{equation}
	
	We start by observing that, in a conformal tessellation, all the faces $f\in \faces$ shared by the elements are given by the intersection between two neighboring elements $K$ and $K'$. Thanks to this, the left-hand side of \eqref{eq:CIP_thesis} can be written as
	\begin{equation}
		\begin{split}
			\sum_{f\in \faces} \int_f \Big\llbracket \nabla^r_{\uvec{\nu}_f} \varphi_i \Big\rrbracket  \Big\llbracket \nabla^r_{\uvec{\nu}_f} \uvec{u}_h \Big\rrbracket  d\sigma(\uvec{x})=
			\sum_{K\in\tess}\sum_{\substack{K'\in\tess \\ K' \neq K}} \frac{1}{2}\int_{K \cap K'} \Big\llbracket \nabla^r_{\uvec{\nu}_f} \varphi_i \Big\rrbracket  \Big\llbracket \nabla^r_{\uvec{\nu}_f} \uvec{u}_h \Big\rrbracket  d\sigma(\uvec{x}),
		\end{split}
		\label{jump1}
	\end{equation}
	where, with abuse of notation, we stick to $\uvec{\nu}_f$ to indicate the normal to the face shared by $K$ and $K'$ at the right-hand side. We remark that the orientation of $\uvec{\nu}_f$ is not relevant in this context.
	
	\begin{remark}
		In the previous equation \eqref{jump1}, the sums over $K$ and $K'$ are meant over all the elements of the tessellation: whenever two different elements $K^a$ and $K^b$ do not share any face, then $K^a \cap K^b = \emptyset$ and their contribution is zero. Instead, when they share a face $f$, the contribution of that face is counted twice: once when $K=K^a$ and $K'=K^b$, once when $K=K^b$ and $K'=K^a$. This is why we have to put $\frac{1}{2}$.
		We remark that, due to the assumption of conformal tessellation, any face $f$ not belonging to the boundary is shared exactly by two elements. 
	\end{remark}
	

	Concerning the direction of evaluation of the jump, not relevant in \eqref{eq:CIP_multiD}, we assume here
	$\llbracket z \rrbracket:=z\vert_K-z\vert_{K'}.$
	Thus, from \eqref{jump1}, one gets 
	\begin{subequations}
		\begin{align}
			\sum_{K\in\tess}\sum_{\substack{K'\in\tess \\ K' \neq K}} \frac{1}{2}&\int_{K \cap K'} \Big\llbracket \nabla^r_{\uvec{\nu}_f} \varphi_i \Big\rrbracket  \Big\llbracket \nabla^r_{\uvec{\nu}_f} \uvec{u}_h \Big\rrbracket  d\sigma(\uvec{x})\\
			&=\sum_{K\in\tess}\sum_{\substack{K'\in\tess \\ K' \neq K}} \frac{1}{2}\int_{K \cap K'}  \left( \nabla^r_{\uvec{\nu}_f}\varphi_i\vert_K-\nabla^r_{\uvec{\nu}_f}\varphi_i\vert_{K'} \right)  \Big\llbracket \nabla^r_{\uvec{\nu}_f} \uvec{u}_h \Big\rrbracket d\sigma(\uvec{x})\\
			&=\sum_{K\in\tess}\sum_{\substack{K'\in\tess \\ K' \neq K}} \frac{1}{2}\int_{K \cap K'}  \nabla^r_{\uvec{\nu}_f}\varphi_i\vert_K \Big\llbracket \nabla^r_{\uvec{\nu}_f} \uvec{u}_h \Big\rrbracket d\sigma(\uvec{x})\\
			&\quad\quad-\sum_{K\in\tess}\sum_{\substack{K'\in\tess \\ K' \neq K}} \frac{1}{2}\int_{K \cap K'}  \nabla^r_{\uvec{\nu}_f}\varphi_i\vert_{K'} \Big\llbracket \nabla^r_{\uvec{\nu}_f} \uvec{u}_h \Big\rrbracket d\sigma(\uvec{x})\\
			&=\sum_{K\in\tess}\sum_{\substack{K'\in\tess \\ K' \neq K}} \frac{1}{2}\int_{K \cap K'} \nabla^r_{\uvec{\nu}_f}\varphi_i\vert_K\left( \nabla^r_{\uvec{\nu}_f} \uapp\vert_{K}-\nabla^r_{\uvec{\nu}_f}\uapp\vert_{K'} \right) d\sigma(\uvec{x})\\
			&\quad\quad-\sum_{K\in\tess}\sum_{\substack{K'\in\tess \\ K' \neq K}} \frac{1}{2}\int_{K \cap K'}  \nabla^r_{\uvec{\nu}_f}\varphi_i\vert_{K'} \left( \nabla^r_{\uvec{\nu}_f} \uapp\vert_{K}-\nabla^r_{\uvec{\nu}_f}\uapp\vert_{K'} \right)  d\sigma(\uvec{x}).
			\label{jump2}
		\end{align}
	\end{subequations}

	Let us focus on the term at \eqref{jump2}. By a simple renaming of $K$ and $K'$ in such a way to switch the indices of the sums, by entering the sign $-$ inside the integral and from the fact that $K \cap K' =K' \cap K $, we get
	\begin{subequations}
		\begin{align}
			&-\sum_{K\in\tess}\sum_{\substack{K'\in\tess \\ K' \neq K}} \frac{1}{2}\int_{K \cap K'}  \nabla^r_{\uvec{\nu}_f}\varphi_i\vert_{K'}\left( \nabla^r_{\uvec{\nu}_f} \uapp\vert_{K}-\nabla^r_{\uvec{\nu}_f}\uapp\vert_{K'} \right)  d\sigma(\uvec{x})\\
			&=-\sum_{K'\in\tess}\sum_{\substack{K\in\tess\\ K \neq K'}} \frac{1}{2}\int_{K' \cap K} \nabla^r_{\uvec{\nu}_f}\varphi_i\vert_{K}\left( \nabla^r_{\uvec{\nu}_f} \uapp\vert_{K'}-\nabla^r_{\uvec{\nu}_f}\uapp\vert_{K} \right) d\sigma(\uvec{x})\\
			&=\sum_{K'\in\tess}\sum_{\substack{K\in\tess\\ K \neq K'}} \frac{1}{2}\int_{K' \cap K} \nabla^r_{\uvec{\nu}_f}\varphi_i\vert_{K}\left( \nabla^r_{\uvec{\nu}_f} \uapp\vert_{K}-\nabla^r_{\uvec{\nu}_f}\uapp\vert_{K'} \right) d\sigma(\uvec{x})\\
			&=\sum_{K'\in\tess}\sum_{\substack{K\in\tess\\ K \neq K'}} \frac{1}{2}\int_{K \cap K'} \nabla^r_{\uvec{\nu}_f}\varphi_i\vert_{K}\left( \nabla^r_{\uvec{\nu}_f} \uapp\vert_{K}-\nabla^r_{\uvec{\nu}_f}\uapp\vert_{K'} \right) d\sigma(\uvec{x})\\
			&=\sum_{K\in\tess}\sum_{\substack{K'\in\tess \\ K' \neq K}} \frac{1}{2}\int_{K \cap K'} \nabla^r_{\uvec{\nu}_f}\varphi_i\vert_{K}\left( \nabla^r_{\uvec{\nu}_f} \uapp\vert_{K}-\nabla^r_{\uvec{\nu}_f}\uapp\vert_{K'} \right) d\sigma(\uvec{x}),
			\label{usefulinjump}
		\end{align}

		where the last equality comes from the fact that in this case $\sum_{K\in\tess}\sum_{\substack{K'\in\tess \\ K' \neq K}}$ is equivalent to $\sum_{K'\in\tess}\sum_{\substack{K\in\tess\\ K \neq K'}}$.
		By replacing then \eqref{jump2} with \eqref{usefulinjump}, we get
		\begin{align}
			\begin{split}
				&\sum_{K\in\tess}\sum_{\substack{K'\in\tess \\ K' \neq K}} \frac{1}{2}\int_{K \cap K'} \nabla^r_{\uvec{\nu}_f}\varphi_i\vert_K\left( \nabla^r_{\uvec{\nu}_f} \uapp\vert_{K}-\nabla^r_{\uvec{\nu}_f}\uapp\vert_{K'} \right) d\sigma(\uvec{x})\\
				&\quad\quad-\sum_{K\in\tess}\sum_{\substack{K'\in\tess \\ K' \neq K}} \frac{1}{2}\int_{K \cap K'}  \nabla^r_{\uvec{\nu}_f}\varphi_i\vert_{K'}\left( \nabla^r_{\uvec{\nu}_f} \uapp\vert_{K}-\nabla^r_{\uvec{\nu}_f}\uapp\vert_{K'} \right)  d\sigma(\uvec{x})\\
				&=\sum_{K\in\tess}\sum_{\substack{K'\in\tess \\ K' \neq K}} \frac{1}{2}\int_{K \cap K'} \nabla^r_{\uvec{\nu}_f}\varphi_i\vert_K\left( \nabla^r_{\uvec{\nu}_f} \uapp\vert_{K}-\nabla^r_{\uvec{\nu}_f}\uapp\vert_{K'} \right) d\sigma(\uvec{x})\\
				&\quad\quad +\sum_{K\in\tess}\sum_{\substack{K'\in\tess \\ K' \neq K}} \frac{1}{2}\int_{K \cap K'} \nabla^r_{\uvec{\nu}_f}\varphi_i\vert_{K}\left( \nabla^r_{\uvec{\nu}_f} \uapp\vert_{K}-\nabla^r_{\uvec{\nu}_f}\uapp\vert_{K'} \right) d\sigma(\uvec{x})\\
				&=\sum_{K\in\tess}\sum_{\substack{K'\in\tess \\ K' \neq K}} \int_{K \cap K'} \nabla^r_{\uvec{\nu}_f}\varphi_i\vert_{K}\left( \nabla^r_{\uvec{\nu}_f} \uapp\vert_{K}-\nabla^r_{\uvec{\nu}_f}\uapp\vert_{K'} \right) d\sigma(\uvec{x})\\
				&=\sum_{K\in\tess}\sum_{\substack{K'\in\tess \\ K' \neq K}} \int_{K \cap K'} \nabla^r_{\uvec{\nu}_f}\varphi_i\vert_{K} \Big\llbracket \nabla^r_{\uvec{\nu}_f} \uapp\Big\rrbracket d\sigma(\uvec{x})\\
				&=\sum_{K\in\tess}\sum_{\substack{f\subset\partial{K}\\f \in \faces}} \int_{f} \nabla^r_{\uvec{\nu}_f}\varphi_i\vert_{K}  \Big\llbracket \nabla^r_{\uvec{\nu}_f} \uapp\Big\rrbracket d\sigma(\uvec{x}).
			\end{split}
			\label{almostequivCIPDoDu}
		\end{align}
		
		Now, since $\varphi_i$ has support in the union of elements containing the node $\uvec{x}_i$ to which it is associated, i.e., it is not identically zero just in the elements $K\in K_i$, we can write 
		\begin{align}
			\begin{split}
				\sum_{K\in\tess}\sum_{\substack{f\subset\partial{K}\\f \in \faces}} \int_{f} \nabla^r_{\uvec{\nu}_f}\varphi_i\vert_{K}  \Big\llbracket \nabla^r_{\uvec{\nu}_f} \uapp\Big\rrbracket d\sigma(\uvec{x})=\sum_{K\in K_i}\sum_{\substack{f\subset\partial{K}\\f \in \faces}} \int_{f} \nabla^r_{\uvec{\nu}_f}\varphi_i\vert_{K}  \Big\llbracket \nabla^r_{\uvec{\nu}_f} \uapp\Big\rrbracket d\sigma(\uvec{x}).
			\end{split}
		\end{align}
		
		With this, we have completed the proof of the equivalence \eqref{eq:CIP_thesis}.
	\end{subequations}
\end{proof}

\section{Proof of Proposition \ref{prop:G_lake_at_rest}}\label{app:WB_proof}
\begin{proof}
	\begin{subequations}
		In the context of a lake at rest steady state, the velocity part of the flux and of the source are zero and the considered global flux values \eqref{eq:Gh1}-\eqref{eq:Gh2} reduce to 
		\begin{align}
			\boldsymbol{G}_h(x_i) =\boldsymbol{F}_h(x_i)+&\uvec{R}_h(x_i), \label{eq:GF_1_lake_at_rest}\\
			\boldsymbol{F}_h(x_i)=\sum_{j=1}^I\uvec{F}^{HS}_j\varphi_j(x_i)=\begin{pmatrix}
				0\\
				\left[\frac{gH^2}{2}\right]_h(x_i)
			\end{pmatrix}, \quad \textcolor{black}{\uvec{R}_h}(x_i)=&-\int^{x_i}_{x_L}\left[ -\begin{pmatrix}
				0\\
				\left[gH\frac{\partial}{\partial x} B\right]_h(s)
			\end{pmatrix} \right] ds. \label{eq:GF_2_lake_at_rest}
		\end{align}
		We want to prove that, in such a case, $\uvec{G}_h(x_i)\equiv const$ $\forall i$. Actually, we can see that the first component is identically zero; therefore, let us consider the second component only
		\begin{align}
			G_{h,2}(x_i)&=\left[\frac{gH^2}{2}\right]_h(x_i)+\int^{x_i}_{x_L}\left[gH\frac{\partial}{\partial x} B\right]_h(s)ds\\
			&=\left[\frac{gH^2}{2}\right]_h(x_i)+\int^{x_i}_{x_L}\left(\left[ g (H_h+B_h) \frac{\partial}{\partial x} B_h \right]_h(s)- \frac{\partial}{\partial x} \left[ \frac{gB^2}{2}\right]_h(s)\right)ds.
			\label{eq:second_component_of_G}
		\end{align}

		Let us focus on $K_1$, the leftmost element of the tessellation, and let us consider $x_i\in K_1$. Through basic analysis, thanks to the linearity of the interpolation and to the fact that the total height is constant ($\eta \equiv \overline{\eta}$) for lake at rest, the integral in \eqref{eq:second_component_of_G} can be rewritten as
		\begin{align}
			&\int^{x_i}_{x_L}\left(\left[ g (H_h+B_h) \frac{\partial}{\partial x} B_h \right]_h(s)- \frac{\partial}{\partial x} \left[ \frac{gB^2}{2}\right]_h(s)\right)ds\\
			&=\int^{x_i}_{x_L}\left(\left[ g (H_h+B_h) \frac{\partial}{\partial x} B_h \right]_h(s)\right)ds -  \left[ \frac{gB^2}{2}\right]_h(x_i)+ \left[ \frac{gB^2}{2}\right]_h(x_L)\\
			&=\int^{x_i}_{x_L}\left(\left[ g \overline{\eta} \frac{\partial}{\partial x} B_h \right]_h(s)\right)ds -  \left[ \frac{gB^2}{2}\right]_h(x_i)+ \left[ \frac{gB^2}{2}\right]_h(x_L)\\
			&=g \overline{\eta}\int^{x_i}_{x_L}\left(\left[ \frac{\partial}{\partial x} B_h \right]_h(s)\right)ds -  \left[ \frac{gB^2}{2}\right]_h(x_i)+ \left[ \frac{gB^2}{2}\right]_h(x_L).
			\label{eq:integral}
		\end{align}
		Now, we have a crucial passage: since the interpolation $B_h$, restricted to $K_1$, lives in the space of the polynomials of degree $M$, its derivative lives in the space of the polynomials of degree $M-1$ and, hence, it can be interpolated exactly through the basis functions $\varphi_i$ of degree $M$ with support in $K_1$ and so $\left[ \frac{\partial}{\partial x} B_h \right]_h=\frac{\partial}{\partial x} B_h$. 
		This allows to recast \eqref{eq:integral} as
		\begin{align}
			\begin{split}
				&g \overline{\eta}\int^{x_i}_{x_L}\left(\left[ \frac{\partial}{\partial x} B_h \right]_h(s)\right)ds -  \left[ \frac{gB^2}{2}\right]_h(x_i)+ \left[ \frac{gB^2}{2}\right]_h(x_L)\\
				&=g \overline{\eta}\int^{x_i}_{x_L}\left( \frac{\partial}{\partial x} B_h (s)\right)ds -  \left[ \frac{gB^2}{2}\right]_h(x_i)+ \left[ \frac{gB^2}{2}\right]_h(x_L)\\
				&=g \overline{\eta}\left(B_h (x_i)-B_h (x_L) \right) - \left[ \frac{gB^2}{2}\right]_h(x_i)+ \left[ \frac{gB^2}{2}\right]_h(x_L).
			\end{split}
		\end{align}
		
		Coming back to \eqref{eq:second_component_of_G}, thanks to the fact that $H=\overline{\eta}-B$, we have
		
		\begin{align}
			\begin{split}
				G_{h,2}(x_i)&=\left[\frac{gH^2}{2}\right]_h(x_i)+\int^{x_i}_{x_L}\left(\left[ g (H_h+B_h) \frac{\partial}{\partial x} B_h \right]_h(s)- \frac{\partial}{\partial x} \left[ \frac{gB^2}{2}\right]_h(s)\right)ds\\
				&=\left[\frac{gH^2}{2}\right]_h(x_i)+g \overline{\eta}\left(B_h (x_i)-B_h (x_L) \right) -  \left[ \frac{gB^2}{2}\right]_h(x_i)+ \left[ \frac{gB^2}{2}\right]_h(x_L)\\
				&=\left[\frac{g(\overline{\eta}-B)^2}{2}\right]_h(x_i)+g \overline{\eta}\left(B_h (x_i)-B_h (x_L) \right) -  \left[ \frac{gB^2}{2}\right]_h(x_i)+ \left[ \frac{gB^2}{2}\right]_h(x_L)\\
				&=\left[\frac{g(\overline{\eta}^2+B^2-2\overline{\eta}B)}{2}\right]_h(x_i)+g \overline{\eta}\left(B_h (x_i)-B_h (x_L) \right)-\left[ \frac{gB^2}{2}\right]_h(x_i)+ \left[ \frac{gB^2}{2}\right]_h(x_L)\\
				&=\frac{g\overline{\eta}^2}{2}+\left[\frac{gB^2}{2}\right]_h(x_i)-g\overline{\eta}B_h(x_i)+g \overline{\eta}\left(B_h (x_i)-B_h (x_L) \right)-\left[ \frac{gB^2}{2}\right]_h(x_i)+ \left[ \frac{gB^2}{2}\right]_h(x_L)\\
				&=\frac{g\overline{\eta}^2}{2}-g\overline{\eta}B_h (x_L) + \left[ \frac{gB^2}{2}\right]_h(x_L)= \left[ \frac{g\overline{\eta}^2+gB^2-2g\overline{\eta}B}{2}\right]_h(x_L)=const, \quad \forall x_i \in K^1.
			\end{split}
			\label{eq:second_component_of_G_bis}
		\end{align}
		
		We proved that, for the DoFs $x_i$ in the first element, the global flux is equal to a constant independent of the specific DoF.
		Actually, exactly through the same computations, one can show that this holds more in general for any DoF
		\begin{equation}
			G_{h,2}(x_i)= \left[ \frac{g\overline{\eta}^2+gB^2-2g\overline{\eta}B}{2}\right]_h(x_L)=const, \quad \forall i=1,\dots,I.
		\end{equation}
		The key point is that, despite $\frac{\partial}{\partial x} B_h$ being discontinuous across the interfaces of the elements, $B_h$ is continuous, leading to a cancellation effect in the integration over subsequent elements. 
		This completes the proof.
	\end{subequations}
\end{proof}

\bibliography{biblio}
\bibliographystyle{plain}


\end{document}